\documentclass[graybox, envcountchap, vecarrow]{svmult}

\usepackage{mathptmx}
\usepackage{courier}
\usepackage{mathtools,enumitem}
\usepackage{esint}

\usepackage[hidelinks]{hyperref}

\newcommand{\Lip}{\mathbf{Lip}}
\renewcommand{\L}[1]{\mathbf{L^{#1}}}
\newcommand{\Lloc}[1]{\mathbf{L^{#1}_{loc}}}
\newcommand{\C}[1]{\mathbf{C^{#1}}}
\renewcommand{\d}{\mathrm{d}}
\newcommand{\R}{{\mathbb{R}}}
\newcommand{\reali}{{\mathbb{R}}}
\usepackage{amsfonts}
\newcommand{\N}{{\mathbb{N}}}
\newcommand{\Cc}[1]{\mathbf{C_c^{#1}}}
\newcommand{\supp}{\mathrm{supp}}
\newcommand{\sign}{\mathrm{sign}}
\newcommand{\tv}{\mathrm{TV}}
\newcommand{\BV}{\mathbf{BV}}
\newcommand{\BVloc}{\mathbf{BV_{loc}}}
 \newcommand{\lip}{\mathrm{Lip}}
\usepackage{array}
\newcommand{\modulo}[1]{{\left|#1\right|}}
\newcommand{\norma}[1]{{\left\|#1\right\|}}

\usepackage{calc}
\usepackage[compact]{titlesec}

\begin{document}

\allowdisplaybreaks

\title{Follow-the-leader approximations of macroscopic models for vehicular and pedestrian flows}
\titlerunning{Follow-the-leader approximations}
\author{M.\ Di Francesco, S.\ Fagioli, M.\ D.\ Rosini and G.\ Russo}
\institute{Marco Di Francesco \at DISIM, Universit\`a degli Studi dell'Aquila, via Vetoio 1 (Coppito), 67100 L’Aquila (AQ), Italy, \email{marco.difrancesco@univaq.it}
\and Simone Fagioli \at DISIM, Universit\`a degli Studi dell'Aquila, via Vetoio 1 (Coppito), 67100 L’Aquila (AQ), Italy, \email{simone.fagioli@dm.univaq.it}
\and Massimiliano D.\ Rosini \at Instytut Matematyki, Uniwersytet Marii Curie-Sk\l odowskiej, plac Marii Curie-Sk\l odowskiej 1, 20-031 Lublin, Poland, \email{mrosini@umcs.lublin.pl}
\and Giovanni Russo \at Dipartimento di Matematica ed Informatica, Universit\`a di Catania, Viale Andrea Doria 6, 95125 Catania, Italy, \email{russo@dmi.unict.it}}

\maketitle

\abstract{We review recent results and present new ones on a deterministic \emph{follow-the-leader} particle approximation of first and second order models for traffic flow and pedestrian movements.
We start by constructing the particle scheme for the first order Lighthill-Whitham-Richards (LWR) model for traffic flow.
The approximation is performed by a set of ODEs following the position of discretised vehicles seen as moving particles.
The convergence of the scheme in the many particle limit towards the unique entropy solution of the LWR equation is proven in the case of the Cauchy problem on the real line.
We then extend our approach to the Initial-Boundary Value Problem (IBVP) with time-varying Dirichlet data on a bounded interval.
In this case we prove that our scheme is convergent strongly in $\L1$ up to a subsequence.
We then review extensions of this approach to the Hughes model for pedestrian movements and to the second order Aw-Rascle-Zhang (ARZ) model for vehicular traffic.
Finally, we complement our results with numerical simulations.
In particular, the simulations performed on the IBVP and the ARZ model suggest the consistency of the corresponding schemes, which is easy to prove rigorously in some simple cases.
}

\section{Introduction}

The modeling of vehicular traffic flow can be considered as one of the most important challenges of applied mathematics within the last seventy years. Among its several repercussions on real-world applications we mention e.g.\ the development of smart traffic management systems for integrated applications of communications, control, and information processing technologies to the whole transport system. Other important resultant benefits are the implementation of complex problem solving in traffic management and the addressing of practical problems such as reducing congestion and related costs. These goals can be achieved by optimising the use of transport resources and infrastructures of the transport system as a whole, by bringing more efficiency in terms of traffic fluidity, and by providing procedures for system stabilisation.

Several analytical models for vehicular traffic have been developed in the last decades. In the first instance, they are classified into two main classes: microscopic models~-~taking into account each single vehicle~-~and macroscopic ones~-~dealing with averaged quantities. We refer to \cite{Bellomo2002, bellomo2011modeling, PiccoliTosinsurvey, rosini_book} for a survey of the most commonly used models currently available in the literature.

Recently, the availability of on-line data allows the implementation of real-time strategies aiming at avoiding (or mitigating) congested traffic. To address this task, the development and the application of analytical models that are easy-to-use and with a high performance in terms of time and reliability are essential requirements. In this sense, opposed to direct numerical `individual based' simulations of a large number of interacting vehicles~-~as typical when dealing with microscopic models~-~many researchers recommend using macroscopic models for traffic flow. The main advantages of the macroscopic approach with respect to the microscopic one are
\begin{itemize}
  \item the model is completely evolutive and is able to rapidly describe traffic situations at every time;
  \item the resulting description of queues evolution and of traveling times is accurate as the position of shock waves can be exactly computed and corresponds to queue tails;
  \item the macroscopic theory helps developing efficient numerical schemes suitable to describe very large number of vehicles;
  \item the model can be easily calibrated, validated and implemented as the number of parameters is low;
  \item the theory allows to state and possibly solve optimal management problems.
\end{itemize}

The macroscopic variables are the density $\rho$ (number of vehicles per unit length of the road), the velocity $v$ (space covered per unit time by the vehicles) and the flow $f$ (number of vehicles per unit time).
Clearly, the macroscopic variables are in general functions of time $t>0$ and space $x \in \R$.
By definition
\begin{equation}\label{eq:macro1}
f=\rho \, v.
\end{equation}
Moreover, the conservation of the number of vehicles along a road with neither entrances nor exits is expressed by the one dimensional scalar conservation law \cite{BressanBook}
\begin{equation}\label{eq:macro2}
\rho_t + f_x = 0.
\end{equation}
The system \eqref{eq:macro1}, \eqref{eq:macro2} has three unknown variables.
Hence a further condition has to be imposed.
There are two main approaches to do it: first and second order models.
First order models introduce a further explicit expression of one of the three unknown variables in terms of the remaining two. The prototype first order model is the Lighthill, Whitham \cite{LWR1} and Richards \cite{LWR2} (LWR) model.
The basic assumption of LWR is that the velocity of any driver depends on the density alone
\[v=\mathcal{V}(\rho),\]
where $\mathcal{V} \in \C1([0,\rho_{\max}]; [0,v_{\max}])$ is non-increasing, with $\mathcal{V}(\rho_{\max})=0$ and $\mathcal{V}(0) = v_{\max} > 0$, where $\rho_{\max} > 0$ is the maximal density corresponding to the `bumper to bumper' situation, and $v_{\max}$ is the maximal speed corresponding to the free road.
As a result, the LWR model is given by the scalar conservation law
\begin{equation}\label{eq:LWR}
\rho_t+[\rho\,\mathcal{V}(\rho)]_x=0.
\end{equation}
Second order macroscopic models close the system \eqref{eq:macro1}, \eqref{eq:macro2} by adding a further conservation law.
The most celebrated second order macroscopic model is the Aw, Rascle \cite{ARZ1} and Zhang \cite{ARZ2} (ARZ) model. Away from the vacuum state $\rho=0$, the ARZ model writes
\begin{align}\label{eq:ARZ_intro}
&\rho_t+[\rho\,v]_x=0,&
\left[\rho\left(v+p(\rho)\right)\right]_t + \left[\rho\left(v+p(\rho)\right)v\right]_x=0,
\end{align}
where the function $p(\rho)$ is introduced to take into account drivers' reactions to the state of traffic in front of them.

The main drawback of the LWR model is the unrealistic behaviour of the drivers adjusting instantaneously their velocities according to the densities they are experiencing. 
Moreover, the graph of a map $\rho \mapsto \left[\rho \, \mathcal{V}(\rho)\right]$ can not represent the cloud of points in the $(\rho,f)$-plane obtained by empirical measurements.
The ARZ model avoids these drawbacks of the LWR model. 
However, the system \eqref{eq:ARZ_intro} degenerates into just one equation at the vacuum state $\rho = 0$.
In particular, the solutions to the ARZ model do not depend continuously on the initial data in any neighbourhood of $ \rho=0$.

We point out that \eqref{eq:macro1} and \eqref{eq:macro2} are the only accurate physical laws in vehicular traffic theory.
All other equations result from coarse approximations of empirical observations.
However, as the dynamics of any living system are influenced by psychological effects, nobody would expect that traffic models could reach an accuracy comparable to that attained in other domains of science, such as thermodynamics or Newtonian physics.
Nevertheless, they can have sufficient descriptive power for the specific application-driven purpose, and they can help understanding non-trivial phenomena of vehicular traffic.

The use of macroscopic models relies on the \emph{continuum assumption}, namely on the assumption that the medium is indefinitely divisible without changing its physical nature.
This assumption is not justifiable in the context of vehicular traffic, but is accepted as a technical hypothesis.
In order to make more clear the continuum hypothesis, the study of the micro-to-continuum limit for first and second order models has been proposed in \cite{ColomboMarson, ColomboRossi} and \cite{AKMR2002, degond_rascle} respectively.
Our goal is to address said discrete-to-continuum limit in a rigorous analytic form, both for first and second order models, by proving that the macroscopic models can be \emph{solved} as a \emph{many particle limit} of  discrete (microscopic) ODE-based models.

We sketch here our approach for the LWR model \eqref{eq:LWR}, described in detail in Section~\ref{sec:LWRscheme}.
Fix an initial density $\bar{\rho}$.
Let $L \doteq \|\bar{\rho}\|_{\L1(\R)}$ be the total space occupied by the all vehicles (i.e.\ the total mass in a `continuum PDEs' language).
For a given positive $n \in \N$, we split $\bar{\rho}$ into $n$ platoons of `possibly fractional' vehicles, each one of equal length $\ell_n \doteq L/n$, with the endpoints of each platoon positioned at $\bar{x}_i \in \reali$, $i=0,\ldots,n$.
The points $\bar{x}_i$ are taken as initial condition to the microscopic Follow-The-Leader (FTL) model for vehicular traffic
\begin{equation}\label{eq:ODEintro}
\begin{cases}
   \dot{x}_{i}(t) = \mathcal{V}\left(\frac{\ell_n}{x_{i+1}(t) - x_i(t)}\right),& i \in \{0,\ldots,n-1\},\\
   \dot{x}_{n}(t) = v_{\max}.
\end{cases}
\end{equation}
The points $x_i(t)$ are interpreted as moving particles along the real line $\R$.
In Lemma~\ref{lem:maximum} below we prove that no collisions occur between the particles, as the distance between two consecutive particles is bounded from below by $\ell_n/\rho_{\max}$ for all times.
We then consider the \emph{discrete density}
\begin{align*}
&\rho^n(t,x) \doteq \sum_{i=0}^{n-1} R^n_i(t) \, \mathbf{1}_{[x_i(t),x_{i+1}(t))},
&R^n_i(t)\doteq\frac{\ell_n}{x_{i+1}(t)-x_i(t)},
\end{align*}
and prove that (up to a subsequence) its limit as $n \to \infty$  is the entropy solution to the LWR model \eqref{eq:LWR} in the Oleinik-Kruzhkov sense \cite{Kruzhkov, oleinik}.
The convergence of the particle scheme \eqref{eq:ODEintro} towards \eqref{eq:LWR} is proven rigorously in \cite{MarcoMaxARMA2015}, see also the improved results in \cite{DF_fagioli_rosini_UMI}.
We refer to Section~\ref{sec:LWR} for the details.

The result in \cite{MarcoMaxARMA2015,DF_fagioli_rosini_UMI} can be interpreted as a \emph{particle method} for the one dimensional scalar conservation law \eqref{eq:LWR}, which can be applied in the context of numerics.
Particle methods feature a long standing history as a numerical method for transport equations, see e.g.\ \cite{neunzert_klar_struckmeier} and the references therein.
Moreover, several effective numerical schemes for nonlinear conservation laws are proposed in the literature.
We mention the pioneering work of Glimm \cite{glimm} for systems, and the Wave-Front Tracking (WFT) algorithm proposed by Dafermos in \cite{Daf72} and improved later on by Di~Perna \cite{dip76} and Bressan \cite{Bre92}, see also \cite{holden2015front} and the references therein for more details.
Our approach differs from most of the numerical methods for scalar conservation laws in that it interprets the microscopic limit as a \emph{mean field limit of a system of interacting particles with nearest neighbour type interaction}, in the spirit of interacting particles systems in probability, kinetic theory, statistical mechanics, and mathematical biology, see e.g.\ \cite{dobrushin, morrey, onsager}. We stress in particular the fundamental role of many particle exclusion processes in probability, a subject which has been extensively studied in a vast literature in the past decades, see e.g.\ \cite{ferrari, ferrari_TASEP, landim} and the references therein.
It is worth recalling at this stage that Lions, Perthame and Tadmor proved in \cite{lions_perthame_tadmor} that nonlinear conservation laws can also be solved via kinetic approximation.

Unlike in most of the aforementioned articles, our approach should be regarded as a \emph{deterministic particle approximation} to the target PDE's.
A pioneering result is the one by Russo \cite{russo}, which applies to the linear diffusion equation with the diffusion operator replaced by a nearest neighbour interaction term, see also later generalizations in \cite{gosse_toscani, matthes_osberger}.
Our approach can be considered in the spirit of \cite{russo}, applied to scalar conservation laws.
We also mention the paper by Brenier and Grenier \cite{grenier}, which provides a particle approximation of the pressureless Euler system.

Our approach follows essentially the same strategy in the uniform estimates adopted for the WFT algorithm, except that a lighter notion of time-continuity is needed involving (a scaled version of) the \emph{$1$-Wasserstein distance}, see Section~\ref{sec:preliminaries} or \cite{AGS, villani} for more details.
A major advantage in using the Wasserstein distance relies on its identification with the $\L1$-topology in the space of \emph{pseudo-inverses of cumulative distributions}. Such an identification allows to recover formally the ODE system \eqref{eq:ODEintro} as the most natural way to approximate \eqref{eq:LWR} via Lagrangian particles. We briefly sketch it here. Let $\rho$ be the solution to \eqref{eq:LWR} and let
\[F(t,x)\doteq\int_{-\infty}^x \rho(t,x) \,{\d}x\,\in [0,L],\]
be its primitive.
The pseudo inverse variable $X(t,z)\doteq\inf\left\{x\in \reali \colon F(x)>z\right\}$, $z \in [0,L)$, formally satisfies the \emph{Lagrangian PDE}
\begin{equation*}
    X_t(t,z) = \mathcal{V}\left(X_z(t,z)^{-1}\right).
\end{equation*}
Now, if we replace the above $z$-derivative by a forward finite difference
\begin{equation*}
    X_z\approx \frac{X(t,z+\ell_n)-X(t,z)}{\ell_n} ,
\end{equation*}
and assume that $X$ is piecewise constant on intervals of length $\ell_n$, the ODE system \eqref{eq:ODEintro} is immediately recovered, with the structure
\[X(t,z)=\sum_i x_i(t) \, \chi_{[i\ell_n,(i+1)\ell_n)}(z).\]
The use of pseudo-inverse variables and Wasserstein distances in the framework of scalar conservation laws is not totally new, see e.g.\ \cite{BBL, CDL1}.
As far as the LWR model is concerned, in \cite{Newell} a simplified version of the LWR model is derived by introducing as new variable the cumulative number of vehicles passing through a location $x$ at time $t$, see also \cite{Aubin2010963, Daganzo2005187}.

A natural question concerning the particle approximation procedure described above is whether or not it can be applied to recover the solution to the IBVP with \emph{Dirichlet boundary condition}
\begin{equation}\label{eq:dirichlet_intro}
  \begin{cases}
   \rho_t +f(\rho)_x = 0, & x\in (0,1),\,\, t \in (0,T),\\
   \rho(0,x)=\bar{\rho}(x), & x\in (0,1),\\
   \rho(t,0)=\bar{\rho}_0(t), & t \in (0,T),\\
   \rho(t,1)=\bar{\rho}_1(t), & t \in (0,T).
   \end{cases}
\end{equation}
Such a question is addressed for the first time in the present work. Due to the propagation of the initial and boundary conditions along characteristic lines, it is well known that a concept of Dirichlet condition for a nonlinear conservation law has to be formulated in a set-valued sense. The first rigorous definition of entropy solution in this context was provided in \cite{BardosleRouxNedelec}, in which existence and uniqueness were proven in the scalar multidimensional case. In the one dimensional case, a more intuitive notion of entropy solution was provided in \cite{dubois_lefloch}, where the authors proved that at least in the scalar case the trace of the solution at the boundary is obtained by solving a Riemann problem within the trace itself and the boundary datum.

The substantial mismatch between Lagrangian and Eulerian speeds of propagation suggests that prescribing the behaviour of the particle system \eqref{eq:ODEintro} near the boundary should not involve characteristic speeds.
Inspired by the extremely simple structure of the FTL system \eqref{eq:ODEintro}, the boundary dynamics should follow a very natural process, possibly reminiscent of empirical observation in real contexts (e.g.\ a toll gate).
At the same time, such a dynamics should be able to capture the notion of entropy solution for the limiting IBVP for a large number of particles.
Our choice for the definition of the scheme in this case is pretty natural. We sketch it here in the simple case of constant boundary conditions $\rho(t,0^+)=\bar{\rho}_0$, $\rho(t,1^-)=\bar{\rho}_1$.

Initially $n+N+1$ particles of mass $\ell_n \doteq n^{-1} \, \|\bar{\rho}\|_{\L1(0,1)}$ are set in $\bar{x}_{-N},\ldots,\bar{x}_n$ with $\bar{x}_{-N} < \ldots < \bar{x}_{-1} < \bar{x}_0 \doteq 0 < \bar{x}_1< \ldots < \bar{x}_{n-1} < \bar{x}_n \doteq 1$.
The \emph{entering} condition is set by requiring that $\bar{x}_i \doteq i \, \ell_n / \bar{\rho}_0$, $i\in\{-N,\ldots,-1\}$, so that the queuing particles in $x<0$ are equidistant and matching the boundary datum $\bar{\rho}_0$.
The \emph{exit} condition is set by requiring that $\dot{x}_n = \mathcal{V}(\bar{\rho}_1)$.
We then let evolve the particles according to the corresponding version of the FTL scheme \eqref{eq:ODEintro}.
After some time, some of the queuing vehicles will enter the domain $[0,1]$ and some particle will leave it.
In general, in a finite time the distances between the particles in $x<0$ will not match the boundary datum $\bar{\rho}_0$, as well as the leftmost particle in $x\ge1$ will not necessarily move with velocity $\mathcal{V}(\bar{\rho}_1)$.
For this reason, we introduce a sufficiently small time step $\tau>0$ and, at each time $t=k\,\tau$, $k \in \N$, we rearrange the particles both in $x<0$ and $x>1$ so that the resulting densities match the corresponding boundary data, while on each time interval $[k\,\tau, (k+1)\,\tau)$, $k \in \N$, we let the particles evolve according to the corresponding version of FTL scheme \eqref{eq:ODEintro}, with $\dot{x}_n = \mathcal{V}(\bar{\rho}_1)$.
Let us underline that the number $N$ (which depends on $n$) should be prescribed initially depending on the final time $T$, in a way that some of the queuing vehicles are still left in $x<0$ at time $t=T$.

In order to extend our approach to time-varying boundary data, we discretise the boundary conditions with respect to time via a time step $\tau$, solve the particle system in each time interval with constant boundary data, and then rearrange the particles outside the domain according the boundary condition at the next time step.
We defer to Section~\ref{subsec:scheme} for more details.
We remark that in the case of constant boundary conditions for the continuum equation \eqref{eq:LWR} the rearrangement of the boundary datum at each time step $\tau$ is not necessary as long as no waves hit the boundary from the interior of the domain.
Such a situation also holds in our particle approximation, as we shall see in the simulations in Subsection~\ref{sec:simibvp}.

We prove rigorously in Section~\ref{Sec:dirichlet} that the above particle scheme converges strongly in $\L1$ to a limiting density $\rho$ as $n\rightarrow \infty$ and $\tau \rightarrow 0$. Such a result does not require any condition on how fast (or slow) $n$ should tend to infinity with respect to $\tau$ tending to zero. The consistency of the scheme is provided in simple cases, i.e.\ either constant initial data or boundary conditions yielding outgoing characteristic speeds at the boundary. As we explain below in Section \ref{Sec:dirichlet}, the definition of our approximating scheme is reminiscent of the notion of entropy solution provided in \cite{dubois_lefloch}, see Definition~\ref{def:entropy_sol_lef}, in which the trace of the solution $\rho$ to \eqref{eq:LWR} at the boundary is required to match the solution to a suitable Riemann problem. Our scheme actually prescribes a constant datum outside the domain at each time step, in a way to produce the approximation to a Riemann problem near the boundary. The simulations we provide in Subsection~\ref{sec:simibvp} support our conjecture that our scheme is consistent with the notion of entropy solution in the sense of \cite{BardosleRouxNedelec,dubois_lefloch}.

The deterministic particle approach started in \cite{MarcoMaxARMA2015} has seen significant extensions to similar models. A first one has been performed in \cite{DF_fagioli_rosini} on the ARZ model \eqref{eq:ARZ_intro}.
Despite the second order nature of ARZ, the strategy developed in \cite{MarcoMaxARMA2015} for the first order LWR model \eqref{eq:LWR} applies also in this case. This reveals that the multi-species nature of the ARZ model is quite relevant in the dynamics. Our rigorous results only deal with the convergence towards a weak solution.
The problem of the uniqueness of entropy solutions for the ARZ model is quite a hard task.
For this reason we do not address here the consistency of our scheme.
Let us point out that our approach for the ARZ system deeply differs from the one proposed in \cite{AKMR2002}, which essentially works away from the vacuum state and is implemented via a time discretisation and suitable space-time scaling. Our result in \cite{DF_fagioli_rosini} works near the vacuum state and no scaling is performed. Unlike previous numerical attempts (e.g.\ \cite{chalonsgoatin2007}) our method is conservative and is able to cope with the vacuum.
We briefly recall the result of \cite{DF_fagioli_rosini} in Section~\ref{sec:aw} below.

Another extension of our particle approach has been performed in \cite{DF_fagioli_rosini_russo} on a one-dimensional version of the \emph{Hughes model} \cite{Hughes2002} for pedestrian movements, see \eqref{eq:model} below.
In this model, the movement of a dense human crowd is modelled  via a `thinking fluid' approach in which the crowd is modelled as a continuum medium, with Eulerian velocity computed via a nonlocal constitutive law of the overall distribution of pedestrians.
Such a nonlocal dependence is encoded in the \emph{weighted} distance function $\phi$, computed at a quasi-equilibrium regime via a nonlinear \emph{running cost} function $c(\rho)$.
The function $\phi$ may be interpreted as an estimated exit time for a given distribution of pedestrians.
We refer to \cite{bellomo_bellouquid_2011, rosini_book} and the references therein for the mathematical modelling of human crowds, and to \cite{Hughes2002, DiFrancescoMarkowichPietschmannWolfram, AmadoriDiFrancesco, El-KhatibGoatinRosini, GoatinMimault, BurgerDiFrancescoMarkowichWolfram, AmadoriGoatinRosini, carrillo_martin_wolfram} for the rigorous analytical results and numerical simulations available in the literature on the Hughes model.

A fully satisfactory existence theory for the Hughes model is still missing.
A mathematical theory in this setting was first addressed in \cite{DiFrancescoMarkowichPietschmannWolfram}, in which the eikonal equation was replaced by two regularised versions involving a Laplacian term.
A rigorous mathematical treatment of the Riemann problems for the Hughes model without any regularization was performed independently in \cite{AmadoriDiFrancesco} and \cite{El-KhatibGoatinRosini}.
Said result led the basis to tackle the existence theory via a WFT strategy.
As in the paper \cite{AmadoriGoatinRosini}, we prove in \cite{DF_fagioli_rosini_russo} the existence of entropy solutions when the initial condition yields the formation of two distinct groups of pedestrians moving towards the two exits, with the emergence of a vacuum region in between, persisting until the total evacuation of the domain.
However, differently from \cite{AmadoriGoatinRosini} where the WFT method is applied, in \cite{DF_fagioli_rosini_russo} we develop a FTL particle approximation, taking advantage of the fact that our assumptions ensure that the Hughes model can be formulated as a two-sided LWR equations.
We refer to \cite{DF_fagioli_rosini_russo} and to Section~\ref{sec:particle} below for the precise formulation of the particle scheme.
As a result, we prove that the particle scheme converges under (essentially) the same conditions for which an existence result for entropy solutions is available in the literature (with the results in \cite{AmadoriGoatinRosini} in mind).

This chapter is structured as follows.
In Section~\ref{sec:LWR} we review the results in \cite{MarcoMaxARMA2015} and later improvements in \cite{DF_fagioli_rosini_UMI} about the convergence of the FTL scheme \eqref{eq:ODEintro} towards entropy solutions to the LWR equation \eqref{eq:LWR}. The main result is stated in Theorem \ref{thm:main_LWR}.
In Section~\ref{Sec:dirichlet} we prove our new result concerning the convergence of the FTL scheme for the IBV problem \eqref{eq:dirichlet_intro}.
The strong convergence of the scheme is proven in Theorem \ref{thm:main2}. The consistency of the scheme in some special cases is proven in Theorem \ref{thm:minor}.
In Section~\ref{sec:hughes} we review the results in \cite{DF_fagioli_rosini_russo} on the particle approximation of the Hughes model \eqref{eq:model}, with the main result stated in Theorem \ref{teo:2}.
In Section~\ref{sec:aw} we review the results in \cite{DF_fagioli_rosini} on the ARZ model \eqref{eq:ARZ_intro}. The main result is stated in Theorem \ref{thm:mainARZ}.
In Section~\ref{sec:numerics} we collect all the numerical simulations performed for the particle methods introduced in all the aforementioned models. In particular we present new simulations regarding the IBV problem \eqref{eq:dirichlet_intro} in Subsection \ref{sec:simibvp}

In the next subsection we recall the basic results on the Wasserstein distance that are used in this chapter.

\subsection{The Wasserstein distances}\label{sec:preliminaries}

We collect here the main concepts about one dimensional Wasserstein distances, see \cite{villani} for further details.
As already mentioned, we deal with probability densities with constant mass in time and we need to evaluate their distances at different times in the Wasserstein sense.

For a fixed mass $L>0$, we consider the space
\begin{equation*}
  \mathcal{M}_L \doteq \bigl\{\mu \hbox{ Radon measure on $\reali$ with compact support} \colon \mu\ge 0 \text{ and }\mu(\reali)=L \bigr\}.
\end{equation*}
For a given $\mu\in \mathcal{M}_L$, we introduce the pseudo-inverse variable $X_\mu \in \L1([0,L];\R)$ as
\begin{equation}\label{eq:pseudoinverse}
X_\mu(z) \doteq \inf \bigl\{ x \in \R \colon \mu((-\infty,x]) > z \bigr\}.
\end{equation}
Clearly, $X_\mu$ is non decreasing on $[0,L]$, and locally constant on `mass intervals' on which $\mu$ is concentrated. $X_\mu$ may have (increasing) jumps if the support of $\mu$ is not connected. By abuse of notation, in case $\mu=\rho\,\mathcal{L}_1$ is absolutely continuous with respect to the Lebesgue measure, we denote its pseudo-inverse variable by $X_\rho$.

For $L=1$, the one-dimensional \emph{$1$-Wasserstein distance} between $\rho_1,\rho_2\in \mathcal{M}_1$ (defined in terms of optimal plans in the Monge-Kantorovich problem, see e.g.\ \cite{villani}) can be defined as
\[
W_1(\rho_1,\rho_2) \doteq \norma{X_{\rho_1}-X_{\rho_2}}_{\L1([0,1];\reali)}.
\]
We introduce the \emph{scaled $1$-Wasserstein distance} between $\rho_1,\rho_2\in \mathcal{M}_L$ as
\begin{equation}\label{eq:wass_equiv0}
    W_{L,1}(\rho_1,\rho_2)\doteq \norma{X_{\rho_1}-X_{\rho_2}}_{\L1([0,L];\reali)} .
\end{equation}
Indeed, a straightforward computation yields $W_{L,1}(\rho_1,\rho_2) = L \, W_1(L^{-1}\rho_1,L^{-1}\rho_2)$.
The distance $W_{L,1}$ inherits all the topological properties of the $1$-Wasserstein distance for probability measures. In particular, a sequence $(\rho_n)_n$ in $\mathcal{M}_L$ converges to $\rho\in \mathcal{M}_L$ in $W_{L,1}$ if and only if for any $\varphi\in \C0(\reali;\reali)$ growing at most linearly at infinity
\[
\lim_{n\to\infty}\int_\reali \varphi(x) \,{\d}\rho_n(x) = \int_\reali \varphi(x) \,{\d}\rho(x) .
\]

We now state a technical result which will serve in the sequel of the chapter.

\begin{theorem}[Generalised Aubin-Lions lemma {}]\label{thm:aubin}
Assume $v:[0,\infty)\rightarrow [0,\infty)$ is a continuous and strictly monotone function. 
Let $T, L>0$, $a,b\in\R$ be fixed with $a<b$.
Let $(\rho^n)_n$ be a sequence in $\L\infty((0,T);\,\L1(\R))$ with $\rho^n(t,\cdot) \ge 0$ and $\|\rho^n(t,\cdot)\|_{\L1(\R)}=L$ for all $n \in \N$ and $t\in [0,T]$.
Assume further that
\begin{gather}
\label{Abis}\tag{H1}
\sup_{n\in \N} \left[ \int_0^T \left[\vphantom{\sum}\|v(\rho^n(t,\cdot))\|_{\L1([a,b])} + \tv(v(\rho^n(t,\cdot));\,[a,b])\right]{\d}t\right] < \infty,
\\
\label{Bbis}\tag{H2}
\lim_{h \downarrow 0} \left[ \sup_{n\in \N} \left[ \int_0^{T-h} W_{L,1}(\rho^n(t+h,\cdot),\rho^n(t,\cdot)) ~{\d}t \right] \right] = 0.
\end{gather}
Then, $(\rho^n)_n$ is strongly relatively compact in $\L1([0,T]\times [a,b])$.
\end{theorem}
We refer to the Appendix of \cite{DF_fagioli_rosini_UMI} for the proof of Theorem~\ref{thm:aubin}.
We will sometimes consider the following condition:
\begin{equation}\label{B}\tag{H2$'$}
\text{There exists a constant $C>0$ independent of $n$ such that $W_{L,1}(\rho^n(t,\cdot),\rho^n(s,\cdot))\leq C \, |t-s|$ for all $s,t\in (0,T)$.}
\end{equation}
We point out that \eqref{B} implies \eqref{Bbis} and that it is assumed in both \cite[Theorem~3.5]{DF_fagioli_rosini_UMI} and \cite[Theorem~3.2]{MarcoMaxARMA2015}.

\section{The LWR model}\label{sec:LWR}

In this section we review the results obtained in \cite{MarcoMaxARMA2015}, later improved in \cite{DF_fagioli_rosini_UMI}, on the Cauchy problem for the LWR model \eqref{eq:LWR}
\begin{equation}\label{eq:cauchy}
\begin{cases}
  \rho_t + f(\rho)_x = 0,&  (t,x)\in\R_+\times\R,\\
  \rho(0,x)=\bar\rho(x), & x\in\R,
\end{cases}
\end{equation}
where $f(\rho) \doteq \rho \, v(\rho)$.
If $\rho_{\max} > 0$ is the maximal density corresponding to the situation in which the vehicles are bumper to bumper, and $v_{\max}$ is the maximal speed corresponding to the free road, then the initial datum $\bar\rho$ and the velocity map $v$ are assumed to satisfy the following conditions:
\begin{gather}\tag{I1}\label{I1}
\bar\rho \in \L\infty \cap \L1(\R;[0,\rho_{\max}]),
\\\tag{V1}\label{V1}
v \in \C1([0,\rho_{\max}];[0,v_{\max}]),\quad
v' < 0,\quad
v(0) = v_{\max},\quad
v(\rho_{\max}) = 0.
\end{gather}
In some cases we require also one of the following conditions:
\begin{gather}\tag{I2}\label{I2}
\bar\rho \in \BV(\R;[0,\rho_{\max}]),
\\\tag{V2}\label{V2}
[0,\rho_{\max}]\ni \rho \mapsto \rho \, v'(\rho) \in \overline{\R}_- \text{ is non increasing}.
\end{gather}

\begin{example}[Examples of velocities in vehicular traffic]
The prototype for the velocity in vehicular traffic $v_{GS}(\rho) \doteq v_{\max} \left(1-\frac{\rho}{\rho_{\max}}\right)$ by Greenshields \cite{Greenshields} clearly satisfies the assumptions \eqref{V1}, \eqref{V2}. The same holds for the Pipes-Munjal velocity \cite{pip}
\begin{align*}
    &v_{PM}(\rho) \doteq v_{\max} \, \left[1 - \left(\frac{\rho}{\rho_{\max}}\right)^\alpha \right],
    & \alpha>0,
\end{align*}
in which the concavity of the flux $\rho \, v_{PM}(\rho)$ degenerates at $\rho=0$.
Further examples of speed-density relations that satisfy \eqref{V1}, \eqref{V2} are
\begin{align*}
&v_{GB}(\rho) \doteq v_{\max} \, \left[\log\left(\frac{\rho_{\max}+\alpha}{\alpha}\right)\right]^{-1} \log\left(\frac{\rho_{\max}+\alpha}{\rho+\alpha}\right) ,&
&\alpha >0,
\\
&v_U(\rho) \doteq v_{\max} \left[ 1 - e^{-\rho_{\max}} \right]^{-1} \left[ e^{-\rho} - e^{-\rho_{\max}} \right],
\end{align*}
that result from a slight modification of that ones proposed by Greenberg \cite{Greenberg} and Underwood \cite{Underwood} respectively.
\end{example}

\begin{definition}\label{def:entro_sol_LWR}
Assume \eqref{I1} and \eqref{V1}.
We say that $\rho\in \L\infty(\R_+\times\R)$ is an \emph{entropy solution} to the Cauchy problem \eqref{eq:cauchy} if $\rho(t,\cdot)\rightarrow \bar{\rho}$ in the weak$^*$ $\L\infty$ sense as $t\downarrow 0$ and
\[
  \iint_{\R_+\times\R} \, \Bigl[|\rho(t,x)-k| \, \varphi_t(t,x) + \sign(\rho(t,x) - k) \bigl[ f(\rho(t,x)) - f(k)\bigr] \varphi_x(t,x)\Bigr] {\d}x ~{\d}t\geq 0
\]
for all $\varphi\in \Cc\infty((0,\infty)\times \R)$ with $\varphi\geq 0$ and for all $k\geq 0$.
\end{definition}

We point out that the above definition is slightly weaker than the definition in \cite{Kruzhkov}.
The next theorem collects the uniqueness result in \cite{Kruzhkov} and its variant in \cite{chen_rascle}.
\begin{theorem}[\cite{chen_rascle, Kruzhkov}]
Assume \eqref{I1} and \eqref{V1}.
Then there exists a unique entropy solution to the Cauchy problem \eqref{eq:cauchy} in the sense of Definition~\ref{def:entro_sol_LWR}.
\end{theorem}

\subsection{The follow-the-leader scheme and main result}\label{sec:LWRscheme}

We now introduce rigorously our FTL approximation scheme for \eqref{eq:cauchy}.
Assume \eqref{I1} and \eqref{V1}.
Let
\begin{align*}
&L \doteq \|\bar\rho\|_{\L1(\R)},
&R \doteq \|\bar\rho\|_{\L\infty(\R)}.
\end{align*}
Fix $n\in\mathbb{N}$ sufficiently large.
Let $\ell_n \doteq L/n$ and $\bar{x}^n_1,\ldots,\bar{x}^n_{n-1}$ be defined recursively by
\[
\begin{cases}
\bar{x}^n_1 \doteq \sup \Bigl\{x\in\R\colon\int_{-\infty}^x\bar\rho(x) \,{\d}x < \ell_n \Bigr\},
\\[5pt]
\bar{x}^n_i \doteq \sup \Bigl\{x\in \R\colon\int_{\bar{x}^n_{i-1}}^x\bar\rho(x) \,{\d}x < \ell_n \Bigr\},
&i \in \{2,\ldots,n-1\}.
\end{cases}
\]
It follows that $\bar{x}^n_1 < \bar{x}^n_2 < \ldots < \bar{x}^n_{n-1}$ and
\begin{align}\label{eq:mass}
&\int_{-\infty}^{\bar{x}^n_1}\bar\rho(x) \,{\d}x = \int_{\bar{x}^n_{i-1}}^{\bar{x}^n_{i}}\bar\rho(x) \,{\d}x = \int_{\bar{x}^n_{n-1}}^{\infty}\bar\rho(x) \,{\d}x = \ell_n\leq (\bar{x}^n_i-\bar{x}^n_{i-1})R,
&i \in \{2,\ldots,n-1\}.
\end{align}
We let the $(n-1)$ particles defined above evolve according to the FTL system
\begin{align}\label{eq:FTL}
 &\begin{cases}
    \dot{x}_i^n(t)=v(R^n_i(t)), & i \in \{1,\ldots,n-2\}, \\
    \displaystyle{\dot{x}_{n-1}^n(t)=v_{\max}},  \\
    x^n_i(0) = \bar{x}^n_i, & i \in \{1,\ldots,n-1\},
  \end{cases}
  &R^n_i(t) \doteq \frac{\ell_n}{x^n_{i+1}(t)-x^n_i(t)}.
\end{align}
\begin{lemma}[Discrete maximum principle {\cite[Lemma~1]{MarcoMaxARMA2015}}]\label{lem:maximum}
Assume \eqref{I1} and \eqref{V1}.
Then, the solution $(x_i^n)_{i=1}^{n-1}$ to \eqref{eq:FTL} satisfies for any $t \ge 0$
  \begin{align*}
  &x^n_{i+1}(t)-x^n_i(t)\geq \ell_n/R,
  &i \in \{1,\ldots, n-2\}.
  \end{align*}
\end{lemma}
The above lemma ensures that the particles strictly preserve their initial order.
Hence the solution $(x_i^n)_{i=1}^{n-1}$ to \eqref{eq:FTL} is well defined.

We introduce two \emph{artificial particles} $x_0^n$ and $x_n^n$ as follows
\begin{align}\label{eq:phantom}
  &x^n_0(t) \doteq 2x^n_{1}(t)-x^n_2(t),
  &x^n_n(t) \doteq 2x^n_{n-1}(t)-x^n_{n-2}(t),
\end{align}
and let $R^n_0 \doteq R^n_1$ and $R^n_{n-1} \doteq R^n_{n-2}$.
We then set
\begin{equation}\label{eq:disdens}
\rho^n(t,x) \doteq \sum_{i=0}^{n-1}R^n_i(t) ~ \mathbf{1}_{[x^n_{i}(t),x_{i+1}^n(t))}(x) = \sum_{i=0}^{n-1} \frac{\ell_n}{x^n_{i+1}(t)-x^n_i(t)} ~ \mathbf{1}_{[x^n_{i}(t),x_{i+1}^n(t))}(x).
\end{equation}
We notice that $\|\rho^n(t,\cdot)\|_{\L1(\R)} = L$, $\|\rho^n(t,\cdot)\|_{\L\infty(\R)} \le R$ and that $\rho^n(t,\cdot)$ is compactly supported for all $t \ge 0$.
For future use we compute
\begin{equation}\label{eq:ODE_density}
  \begin{cases}
  \dot{R}^n_i(t)=-\frac{R^n_i(t)^2}{\ell_n} \, \bigl[v(R^n_{i+1}(t))-v(R^n_i(t))\bigr], &i \in \{1,\ldots,n-3\}, \\[10pt]
  \dot{R}^n_{n-2}(t)=-\frac{R^n_{n-2}(t)^2}{\ell_n} \, \bigl[v_{\max}-v(R^n_{n-2}(t))\bigr].
  \end{cases}
\end{equation}

\begin{remark}\label{rem01}
In case $\supp[\bar\rho]$ is bounded either from above or from below, it is possible to improve the above construction.
In the former case, the particle $x^n_n$ can be set on $\max \{\supp[\bar\rho]\}$ initially and let evolve with maximum speed $v_{\max}$, and the preceding particle $x^n_{n-1}$ let evolve according to $\dot{x}^n_{n-1}(t)=v(\ell_n/(x^n_n(t)-x^n_{n-1}(t)))$.
In the latter case, the particle $x^n_0$ can be set on $\min \{\supp[\bar\rho]\}$ initially and let evolve according to $\dot{x}^n_{0}(t)=v(\ell_n/(x^n_1(t)-x^n_{0}(t)))$.
In \cite{MarcoMaxARMA2015} both these conditions are required for the initial datum and such construction is applied.
\end{remark}

The main result of \cite{MarcoMaxARMA2015,DF_fagioli_rosini_UMI} reads as follows.
\begin{theorem}[{\cite[Theorem~2.3]{DF_fagioli_rosini_UMI}, \cite[Theorem~3]{MarcoMaxARMA2015}}]\label{thm:main_LWR}
Assume \eqref{I1} and \eqref{V1}.
Moreover, assume at least one of the two conditions \eqref{I2} and \eqref{V2}.
Then, $(\rho^n)_n$ converges (up to a subsequence) a.e.\ and in $\Lloc1$ on $\R_+\times \R$ to the unique entropy solution to the Cauchy problem \eqref{eq:cauchy} in the sense of Definition~\ref{def:entro_sol_LWR}.
\end{theorem}
We sketch the proof of Theorem~\ref{thm:main_LWR} in the next two subsections.
For simplicity, we assume that $\bar{\rho}$ is compactly supported and apply the corresponding construction explained in Remark~\ref{rem01}.

\subsection{Estimates}\label{sec:LWRestimates}

The result in Lemma~\ref{lem:maximum} ensures that $\|\rho^n(t,\cdot)\|_{\L\infty(\R)}\leq R \doteq \|\bar\rho\|_{\L\infty(\R)}$ for all $t\geq0$.
As usual in the context of scalar conservation laws, a uniform control of the $\BV$ norm is necessary in order to gain enough compactness of the approximating scheme.
We achieve compactness in two distinct ways. The first one is a uniform $\BV$ contraction property for $(\rho^n)_n$, and it obviously requires \eqref{I2}.

\begin{proposition}\label{pro:bv_contraction}
  Assume \eqref{I1}, \eqref{I2} and \eqref{V1}.
  Then, the discretized density $\rho^n$ defined in \eqref{eq:disdens} satisfies for any $t \ge 0$
  \[\tv[\rho^n(t,\cdot)]\leq \tv[\rho^n(0,\cdot)]\leq \tv [\bar\rho].\]
\end{proposition}

\begin{proof}
By \eqref{eq:ODE_density} and \eqref{V1} we have that
\begin{align*}
  \frac{\d}{{\d}t}\tv[\rho^n(t,\cdot)]
  =&~ \bigl[ 1+\sign\bigl(R_1(t)-R_2(t)\bigr) \bigr] \dot{R}_1(t)
  + \bigl[ 1-\sign\bigl(R_{n-3}(t)-R_{n-2}(t)\bigr) \bigr] \dot{R}_{n-2}(t)\\
  &~ + \sum_{i=2}^{n-3} \, \bigl[ \sign\bigl(R_i(t)-R_{i+1}(t)\bigr) - \sign\bigl(R_{i-1}(t)-R_i(t)\bigr) \bigr] \dot{R}_i(t)
\end{align*}
is not positive.
Finally, the estimate $\tv[\rho^n(0,\cdot)]\leq \tv [\bar\rho]$ is a simple exercise.
\end{proof}

The second way to achieve compactness is via the following \emph{discrete Oleinik-type inequality}. Here we require \eqref{V2} in place of \eqref{I2}.

\begin{proposition}[{\cite[Proposition~3.2]{DF_fagioli_rosini_UMI}}]\label{pro:oleinik}
  Assume \eqref{I1}, \eqref{V1} and \eqref{V2}.
  Then, $(x_i^n)_{i=0}^{n}$ satisfies for any $t > 0$
  \begin{align}\label{eq:oleinik}
    &\frac{\dot{x}^n_{i+1}(t)-\dot{x}^n_i(t)}{x^n_{i+1}(t)-x^n_i(t)}\leq \frac{1}{t},
    &i\in\{0,\ldots,n-1\}.
  \end{align}
\end{proposition}

\begin{proof}
\eqref{eq:oleinik} is equivalent to
\begin{align*}
&z_i(t) \doteq t \, R_i(t) \Bigl[\dot{x}_{i+1}(t)-\dot{x}_i(t)\Bigr] \leq \ell_n&
&\text{for all } t>0,
&i\in\{1,\ldots,n-2\}.
\end{align*}
We prove the above estimate inductively on $i$ by using \eqref{eq:ODE_density}.
Since $z_{n-2}(0)=0$ and
\[
\dot{z}_{n-2} \le R_{n-2} \, \bigl[ v_{\max}-v(R_{n-2}) \bigr] \left[1-\frac{z_{n-2}}{\ell_n}\right],
\]
a simple comparison argument shows that $z_{n-1}(t) \ell_n$ for all $t\ge0$, see \cite[Proposition 3.2]{DF_fagioli_rosini_UMI}.
Next we prove that if $z_{i+1}(t)\leq \ell_n$ for all $t\geq 0$ and for some $i\in\{1,\ldots,n-2\}$, then $z_i(t) = t \, R_i(t) \, [ v(R_{i+1}(t)) - v(R_i(t)) ] \leq \ell_n$ for all $t\geq 0$.
Observe that $\sign_+(z_i) = \sign_+(v(R_{i+1})-v(R_i))= \sign_+(R_i-R_{i+1})$ for all $i\in\{1,\ldots,n-3\}$, where $(z)_+ \doteq \max\{z,0\}$.
The inequality $z_{i+1}\leq \ell_n$ and \eqref{V2} imply
\[
  \frac{\d}{{\d}t}(z_i)_+ \leq R_i \left[\vphantom{\frac{(z_i)_+}{\ell_n}} \, \bigl(v(R_{i+1})-v(R_i)\Bigr)_+ -v'(R_i) \, R_i\right] \left(1-\frac{(z_i)_+}{\ell_n}\right).
\]
We observe that the term in the squared bracket in the above estimate is nonnegative.
Therefore, again a comparison argument shows that $z_i(t)\leq \ell_n$ for all $t\geq 0$.
\end{proof}
\begin{remark}
We point out that for $i\in\{1,\ldots,n-2\}$ the estimate \eqref{eq:oleinik} reads
\[\frac{v(R^n_{i+1}(t))-v(R^n_{i}(t))}{x^n_{i+1}(t)-x^n_i(t)}\leq \frac{1}{t},\]
which recalls the one-sided Lipschitz condition in \cite{oleinik, Hoff83}, which characterises entropy solutions to \eqref{eq:LWR}.
\end{remark}
\begin{remark}
The result in Proposition~\ref{pro:oleinik} implies a uniform bound for $(\rho^n)_n$ in $\BVloc(\R_+\times \R)$. In this sense, the $\L\infty\rightarrow\BV$ smoothing effect featured by genuinely nonlinear scalar conservation laws is intrinsically encoded in the particle scheme \eqref{eq:FTL}. We omit the details of the proof, and refer to \cite[Proposition 3.3]{DF_fagioli_rosini_UMI}.
\end{remark}

We prove now \eqref{B}, namely we provide a uniform time continuity estimate in the scaled $1$-Wasserstein distance $W_{L,1}$ defined in~\eqref{eq:wass_equiv0}, which ensures strong $\L1$ compactness with respect to both space and time.

\begin{proposition}\label{pro:time_continuity}
Assume \eqref{I1} and \eqref{V1}.
Then the sequence $(\rho^{n})_{n}$ satisfies \eqref{B}.
\end{proposition}

\begin{proof}
By \eqref{eq:disdens} and \eqref{eq:pseudoinverse} we have that
\[
X_{\rho^n(t,\cdot)}(z)
= \sum_{i=0}^{n-1} \, \Bigl[x^n_i(t) + \left(z-i\,\ell_n\right) R^n_i(t)^{-1}\Bigr] ~ \mathbf{1}_{[i\ell_n,(i+1) \, \ell_n)}(z).
\]
For any $0<s<t$, by \eqref{eq:wass_equiv0},  \eqref{eq:ODE_density} and \eqref{eq:phantom}
\[
  W_{L,1}(\rho^n(t,\cdot),\rho^n(s,\cdot))
  \le
  v_{\max} \, |t-s|
  +\sum_{i=1}^{n-2} \frac{\ell_n}{2} \int_{s}^{t} |v(R^n_{i+1}(\tau))-v(R^n_i(\tau))| ~{\d}\tau
  +\frac{\ell_n}{2} \int_{s}^{t} |v_{\max} - v(R^n_{n-2}(\tau))| ~{\d}\tau
  \leq v_{\max} \, |t-s|,
\]
and this concludes the proof.
\end{proof}

\subsection{Convergence to entropy solutions}

If besides \eqref{I1} and \eqref{V1} we assume either \eqref{I2} or \eqref{V2}, then the propositions~\ref{pro:bv_contraction} and \ref{pro:oleinik} show that $(\rho^n)_n$ satisfies \eqref{Abis} of Theorem~\ref{thm:aubin} on every time interval $[\delta,T]$ with $0<\delta<T$.
Proposition~\ref{pro:time_continuity} then implies that $(\rho^n)_n$ satisfies \eqref{B}, hence also \eqref{Bbis} of Theorem~\ref{thm:aubin}.
Thus, by a simple diagonal argument stretching the time interval $[\delta,T]$ to $(0,T]$, we get that $(\rho^n)_n$ converges (up to a subsequence) a.e.\ and in $\Lloc1$ on $(0,T)\times \R$.
Let $\rho$ be such limit.
\begin{enumerate}[label={\textbf{Step~\arabic*}},itemindent=*,leftmargin=0pt]\setlength{\itemsep}{0cm}

\item
\textbf{: $\boldsymbol\rho$ is a weak solution to \eqref{eq:cauchy}.}
Let $\varphi\in \Cc\infty(\R_+\times \R)$.
By \eqref{eq:disdens} we compute
\begin{align*}
  &~ \iint_{\R_+\times\R} \, \Bigl[ \rho^n(t,x) \, \varphi_t(t,x) + f(\rho^n(t,x)) \, \varphi_x(t,x) \Bigr] {\d}x ~{\d}t + \int_{\R} \rho^n(0,x) \, \varphi(0,x) \,{\d}x\\
  =&~ \sum_{i=0}^{n-1}\int_{\R_+} \, \Biggl[ -\dot{R}^n_i(t)\left(\int_{x^n_i(t)}^{x^n_{i+1}(t)} \varphi(t,x) \,{\d}x\right)
  + R^n_i(t) \Bigl[ \dot{x}_i^n(t)-v(R^n_i(t)) \Bigr] \varphi(t,x^n_i(t))
  \\
  &~ \hphantom{\sum_{i=0}^{n-1}\int_{\R_+} \, \Biggl[}
  - \frac{R^n_i(t)^2}{\ell_n} \, \Bigl[ \dot{x}^n_{i+1}(t)-v(R^n_i(t)) \Bigr] \Biggl[\int_{x^n_i(t)}^{x^n_{i+1}(t)} \varphi(t,x^n_{i+1}(t)) \,{\d}x\Biggr]
  \Biggr]{\d}t.
\end{align*}

Assuming that $\supp[\varphi]\subset [\delta,T]\times \R$ for some $0<\delta<T$, we obtain
\begin{align*}
\Biggl|\iint_{\R_+\times\R} \, \bigl[\rho^n(t,x) \, \varphi_t(t,x) + f(\rho^n(t,x)) \, \varphi_x(t,x)\bigr] {\d}x ~{\d}t \Biggr|
\leq
\frac{T \, \lip[\varphi] \, \ell_n}{2}
\left[v_{\max}+\sup_{t\in [\delta,T]} \tv\bigl(v(\rho^n(t,\cdot));\,J(T)\bigr)\right],
\label{eq:weak_compact}\tag{$\spadesuit$}
\end{align*}
where $J(T) \doteq \bigl[\min\{\supp[\bar\rho]\} + v(R) \, T , \max\{\supp[\bar\rho]\} + v_{\max} \, T \bigr]$.
Hence, by Proposition~\ref{pro:bv_contraction} the right hand side in \eqref{eq:weak_compact} tends to zero as $n\rightarrow \infty$, and since $\rho^n$ tends (up to a subsequence) to $\rho$ a.e., we have that $\rho$ is a weak solution to the Cauchy problem \eqref{eq:cauchy} for positive times.
By \eqref{eq:mass} and the definition of $R^n_i$ we have that
\begin{align*}
&~\modulo{\int_{\R} \, \Bigl[\bar{\rho}(x) - \rho^n(0,x)\Bigr] \varphi(0,x) \,{\d}x}
\le
2\ell_n \, \|\varphi(0,\cdot)\|_{\L\infty(\R)}
+\sum_{i=0}^{n-1} \modulo{\int_{\bar{x}^n_i}^{\bar{x}^n_{i+1}} \bar{\rho}(x) \left[ \varphi(0,x) - \fint_{\bar{x}^n_i}^{\bar{x}^n_{i+1}} \varphi(0,y) \,{\d}y
\right] {\d}x}
\end{align*}
and clearly the above quantity goes to zero as $n \to + \infty$.

\item\textbf{: $\boldsymbol\rho$ is an entropy solution to \eqref{eq:cauchy}.}
Let $\varphi\in \Cc\infty(\R_+ \times \R)$ with $\varphi\geq 0$ and $k\geq 0$.
By \eqref{eq:disdens}
\begin{align*}
&~\iint_{\R_+\times\R} \, \biggl[|\rho(t,x)-k| \, \varphi_t(t,x) + \sign(\rho(t,x) - k) \Bigl[ f(\rho(t,x)) - f(k)\Bigr] \varphi_x(t,x)\biggr] {\d}x ~{\d}t
\\
=&~
k \int_{\R_+} \, \Bigl[
\bigl[ v(k) - \dot{x}_0^n(t) \bigr] \varphi(t,x_0^n(t))
-
\bigl[ v(k) - \dot{x}_n^n(t) \bigr] \varphi(t,x_n^n(t))
\Bigr] {\d}t
\\&~
+\sum_{i=0}^{n-1} \int_{\R_+}
\sign(R^n_i(t) - k)
\Biggl[ -\dot{R}^n_i(t) \, \left( \int_{x^n_{i}(t)}^{x_{i+1}^n(t)} \varphi(t,x) \,{\d}x \right)
- \Bigl[ R^n_i(t) \bigl[ \dot{x}_{i+1}^n(t) - v(R^n_i(t)) \bigr] - k \bigl[ \dot{x}_{i+1}^n(t) -  v(k) \bigr] \Bigr] \varphi(t,x_{i+1}^n(t))
\\
&~\hphantom{+\sum_{i=0}^{n-1} \int_{\R_+}\Biggl[}
+ \Bigl[R^n_i(t) \bigl[ \dot{x}_{i}^n(t) - v(R^n_i(t)) \bigr] - k \bigl[ \dot{x}_{i}^n(t) - v(k) \bigr] \Bigr]  \varphi(t,x_{i}^n(t)) \Biggr] {\d}t.
\end{align*}

Now we use the equations \eqref{eq:ODE_density} and \eqref{eq:FTL} as follows.
Assuming that $\supp[\varphi]\subset [\delta,T]\times \R$ for some $0<\delta<T$, we obtain
\begin{align*}
&\hphantom{=}\iint_{\R_+\times\R} \, \Biggl[|\rho(t,x)-k| \, \varphi_t(t,x) + \sign(\rho(t,x) - k) \Bigl[ f(\rho(t,x)) - f(k)\Bigr] \varphi_x(t,x)\Biggr] {\d}x ~{\d}t
\\
&=
k \int_{\R_+} \, \Bigl[
\bigl[ v(k) - v(R^n_0(t)) \bigr] \varphi(t,x_0^n(t))
-
\bigl[ v(k) - v_{\max} \, \bigr] \varphi(t,x_n^n(t))
\Bigr] {\d}t
\\
&\hphantom{=}
+\sum_{i=0}^{n-2} \int_{\R_+}
\sign(R^n_i(t) - k)
\Biggl[ \frac{R^n_i(t)^2}{\ell_n} \Bigl[v(R^n_{i+1}(t))-v(R^n_i(t))\Bigr] \Biggl[ \int_{x^n_{i}(t)}^{x_{i+1}^n(t)} \bigl[ \varphi(t,x) - \varphi(t,x_{i+1}^n(t)) \bigr]{\d}x \Biggr]
\\
&\hphantom{+\int_{\R_+}}
+ k \Bigl[ \bigl[ v(R_{i+1}^n(t)) - v(k) \bigr] \varphi(t,x_{i+1}^n(t))
- \bigl[ v(R_{i}^n(t)) - v(k) \bigr]  \varphi(t,x_{i}^n(t)) \Bigr] \Biggr] {\d}t
\\
&\hphantom{=}
+\int_{\R_+}
\sign(R^n_{n-1}(t) - k)
\Biggl[ \frac{R^n_{n-1}(t)^2}{\ell_n} \, \Bigl[v_{\max}-v(R^n_{n-1}(t))\Bigr] \Biggl[ \int_{x^n_{n-1}(t)}^{x_{n}^n(t)} \, \bigl[ \varphi(t,x) - \varphi(t,x_{n}^n(t)) \bigr] {\d}x \Biggr]
\\&\hphantom{+\int_{\R_+}}
+k \Bigl[ \bigl[v_{\max} - v(k) \bigr] \varphi(t,x_{n}^n(t))
- \bigl[ v(R_{n-1}^n(t)) - v(k) \bigr] \varphi(t,x_{n-1}^n(t)) \Bigr] \Biggr] {\d}t.
\end{align*}
We already proved, see \eqref{eq:weak_compact}, that
\begin{align*}
&
\sum_{i=0}^{n-2} \int_{\R_+}
\sign(R^n_i(t) - k) \,
\frac{R^n_i(t)^2}{\ell_n} \, \Bigl[v(R^n_{i+1}(t))-v(R^n_i(t))\Bigr]
\Biggl[ \int_{x^n_{i}(t)}^{x_{i+1}^n(t)} \, \bigl[ \varphi(t,x) - \varphi(t,x_{i+1}^n(t)) \bigr]{\d}x \Biggr] {\d}t
\\&
+\int_{\R_+}
\sign(R^n_{n-1}(t) - k)
\frac{R^n_{n-1}(t)^2}{\ell_n} \, \Bigl[v_{\max}-v(R^n_{n-1}(t))\Bigr]
\Biggl[ \int_{x^n_{n-1}(t)}^{x_{n}^n(t)} \, \bigl[ \varphi(t,x) - \varphi(t,x_{n}^n(t)) \bigr] {\d}x \Biggr] {\d}t
\end{align*}
converges to zero as $n\to\infty$.
Hence, to conclude it suffices to observe that
\begin{align*}
&~ k \Biggl[
\bigl[ v(k) - v(R^n_0(t)) \bigr] \varphi(t,x_0^n(t))
-
\bigl[ v(k) - v_{\max} \, \bigr] \varphi(t,x_n^n(t))
\\
&~\hphantom{k \Biggl[}
+\sum_{i=0}^{n-2}
\sign(R^n_i(t) - k)
\Bigl[ \bigl[ v(R_{i+1}^n(t)) - v(k) \bigr] \varphi(t,x_{i+1}^n(t))
- \bigl[ v(R_{i}^n(t)) - v(k) \bigr]  \varphi(t,x_{i}^n(t)) \Bigr]
\\
&~\hphantom{k \Biggl[}
+\sign(R^n_{n-1}(t) - k)
\Bigl[ \bigl[v_{\max} - v(k) \bigr] \varphi(t,x_{n}^n(t))
- \bigl[ v(R_{n-1}^n(t)) - v(k) \bigr] \varphi(t,x_{n-1}^n(t)) \Bigr] \Biggr]
\\
=&~
k \Biggl[ \sum_{i=1}^{n-1}
\bigl[ \sign(R^n_{i-1}(t) - k) - \sign(R^n_i(t) - k) \bigr] \bigl[ v(R_{i}^n(t)) - v(k) \bigr] \varphi(t,x_{i}^n(t))
\\
&~\hphantom{k \Biggl[}
+ \bigl[ 1 +\sign(R^n_0(t) - k) \bigr] \bigl[ v(k) - v(R^n_0(t)) \bigr] \varphi(t,x_{0}^n(t))
+\bigl[1+\sign(R^n_{n-1}(t) - k)\bigr] \bigl[v_{\max} - v(k) \bigr] \varphi(t,x_{n}^n(t)) \Biggr]
\ge 0.
\end{align*}
\end{enumerate}

\section{The LWR model with Dirichlet boundary conditions}\label{Sec:dirichlet}

In this section we tackle a new problem in the context of the FTL approximation for traffic flow models, namely the approximation of the IBVP with time-varying Dirichlet boundary conditions
\begin{equation}\label{eq:ibvp}
\begin{cases}
\rho_t+[\rho \, v(\rho)]_x=0,& (t,x) \in \R_+ \times \Omega,
\\
\rho(0,x) = \bar{\rho}(x),& x \in \Omega,
\\
\rho(t,0) = \bar{\rho}_0(t),& t\in \R_+,
\\
\rho(t,1) = \bar{\rho}_1(t),& t\in \R_+,
\end{cases}
\end{equation}
where, for notational simplicity, we let $\Omega \doteq (0,1)$.
We assume that the velocity map satisfies \eqref{V1}; further we assume that there exists $\delta>0$ such that the initial datum and the boundary data satisfy respectively
\begin{gather}\tag{I3}\label{I3}
\bar\rho \in \L\infty \cap \BV (\Omega;[\delta,\rho_{\max}]),
\\\tag{B}\label{B1}
\bar{\rho}_0, \bar{\rho}_1 \in \L\infty \cap \Lip  \cap \BV(\R_+;[\delta,\rho_{\max}]).
\end{gather}

We adapt the definition of entropy solution given in \cite[Definition 2.1]{ColomboRosiniboundary}, see also \cite{AmadoriColomboboundary1997, dubois_lefloch}, to the case under consideration.

\begin{definition}\label{def:entropy_sol_lef}
Assume \eqref{I3}, \eqref{B1} and \eqref{V1}.
We say that $\rho \in \C0(\R_+;\Lloc\infty(\bar{\Omega};[0,\rho_{\max}]))$ is an \emph{entropy solution} to the IBVP \eqref{eq:ibvp} if
\begin{itemize}
\item
for any test function $\phi \in \Cc\infty(\R_+ \times \Omega)$ with $\phi \ge 0$ and for any $k \in [0,\rho_{\max}]$
\begin{align*}
0\le&~
\iint_{\R_+\times\Omega}
\Bigl[ |\rho-k| \, \phi_t + {\rm sign}(\rho-k) \, [f(\rho)-f(k)]\, \phi_x \Bigr] \d x \, \d t
+\int_{\Omega}
|\bar{\rho}-k| \, \phi(0,x) \, \d x;
\end{align*}

\item
for a.e.\ $\tau \ge 0$ we have $\rho(\tau,0^+) = u(1,x)$ for all $x > 0$, where $u$ is the self-similar Lax solution to the Riemann problem
\[
\begin{cases}
u_t+f(u)_x = 0,&(t,x) \in \R_+ \times \R,
\\
u(0,x) =
\begin{cases}
\bar{\rho}_0(\tau)&\text{if }x<0,
\\
\rho(\tau,0^+)&\text{if }x>0,
\end{cases}
&x \in \R;
\end{cases}
\]

\item
for a.e.\ $\tau \ge 0$ we have $\rho(\tau,1^-) = w(1,x)$ for all $x < 0$, where $w$ is the self-similar Lax solution to the Riemann problem
\[
\begin{cases}
w_t+f(w)_x = 0,&(t,x) \in \R_+ \times \R,
\\
w(0,x) =
\begin{cases}
\rho(\tau,1^-)&\text{if }x<0,
\\
\bar{\rho}_1(\tau)&\text{if }x>0,
\end{cases}
&x \in \R.
\end{cases}
\]

\end{itemize}
\end{definition} 

\subsection{The follow-the-leader scheme and main result}\label{subsec:scheme}

We now introduce rigorously our FTL approximation scheme for \eqref{eq:ibvp}.
Assume \eqref{I3}, \eqref{B1} and \eqref{V1}. For a given $T>0$ and an integer $m\in \mathbb{N}$, we set $\tau_m\doteq T/m$.
We approximate the boundary data $\bar{\rho}_0,\bar{\rho}_1$ with $\bar{\rho}_{0,m},\bar{\rho}_{1,m}$ defined by
\begin{align*}
&\bar{\rho}_{i,m}\doteq\sum_{k=0}^{m-1}\bar{\rho}^k_{i}\mathbf{1}_{[k\,\tau_m,(k+1)\,\tau_m)}
&\text{with }\bar{\rho}_{i}^k\doteq \bar{\rho}_i(k\,\tau_m),&
&i\in\{0,1\}.
\end{align*}
Let again $L\doteq \|\bar\rho\|_{\L1(\Omega)}$ and $R\doteq \|\bar{\rho}\|_{\L\infty(\Omega)}$.
Fix $n \in \N$ sufficiently large and set $\ell_n \doteq L/n$.
Let $\bar{x}_0^0,\ldots,\bar{x}_n^0$ be defined recursively by
\[
\begin{cases}
\bar{x}_0^0 \doteq 0,\\
\bar{x}_i^0 \doteq \sup\left\{ x\in\Omega \colon  \int_{\bar{x}_{i-1}^0}^x \bar{\rho}_\delta(x) ~ {\d}x<\ell_n\right\},
&i \in \{1,\ldots,n\}.
\end{cases}
\]
By construction $\bar{x}_n^0=1$.
We introduce the \emph{artificial queuing mass} $Q$ and the \emph{number of queuing particles} $N$ defined by
\begin{align}\label{eq:QN}
&Q \doteq 2 \, T \, v_{\max} \, \rho_{\max},
&N \doteq \left\lceil Q/\ell_n \right\rceil.
\end{align}
Let the initial positions of the \emph{queuing} particles $\bar{x}_{-N}^0,\ldots,\bar{x}_{-1}^0$ be defined by
\[
\begin{cases}
\bar{x}_{i}^0 \doteq i \, \frac{\ell_n}{\bar{\rho}_0^{0}},
&i \in \{-N+1,\ldots,-1\},\\
\bar{x}_{-N}^0 \doteq \bar{x}_{-N+1}^0 - \frac{q_n}{\bar{\rho}_0^0},
\end{cases}
\]
where $q_n \doteq Q-\ell_n(N-1)\in [0,\ell_n]$ and $\bar{\rho}_0^{0} \doteq \bar{\rho}_0(0)$.
The queuing particles are set in $\R_-$, with equal distances from each other in order to match the density $\bar{\rho}^0_0$, with the only exception of the leftmost particle $\bar{x}_{-N}^0$, which carries a mass $q_n$ (possibly less than $\ell_n$) in order to have a fixed total mass $Q$ for the whole set of queuing particles.

The $(n+N+1)$ positions $\bar{x}_{-N}^0, \ldots, \bar{x}_{n}^0$ are taken as initial conditions of the FTL system
\[
\begin{cases}
  \dot{x}_i^0(t) = v(R_i^0(t)), & t \in [0,\tau_m],\ i \in \{-N,\ldots,n-1\},\\
  \dot{x}_n^0(t) = v(\bar{\rho}_1^0), & t \in [0,\tau_m],\\
  x_i^0(0)=\bar{x}_i^0, & \hphantom{t \in (0,\tau_m),\ } i \in \{-N,\ldots,n\},
\end{cases}
\]
where we have denoted
\[
\begin{cases}
R_i^0(t) \doteq \frac{\ell_n}{x_{i+1}(t)-x_i(t)},
&t \in [0,\tau_m],
~i \in \{-N+1,\ldots,n-1\},\\
R_{-N}^0(t)\doteq \frac{q_n}{x_{-N+1}(t)-x_{-N}(t)},
&t \in [0,\tau_m].
\end{cases}
\]
We then extend the above definitions to $[0,T]$ recursively as follows.
For any $k \in \N$ with $k \ge 1$, we denote by $h_0^k$ the number of particles that \emph{strictly} crossed $x=0$ during the time interval $(0, k \, \tau_m]$, and by $h_1^k$ the number of particles that crossed $x=1$ during the same time interval (counting the possible particle positioned at $x=1$ at time $k\,\tau_m$). We rearrange the particles positions at time $t=k\,\tau_m$ by setting
\[
\bar{x}_i^k \doteq
\begin{cases}
  x_i(k \, \tau_m), & i \in \{-h_0^k-1,\ldots,n-h_1^k+1\},\\
  x_{n-h_1^k+1}(k\,\tau_m) + (i-n+h_1^k-1) \, \frac{\ell_n}{\bar{\rho}_1^{k}} , & i \in \{n-h_1^k + 2,\ldots,n\},\\
  x_{-h_0^k-1}(k \, \tau_m)+(i+h_0^k+1) \, \frac{\ell_n}{\bar{\rho}_0^{k}} , & i \in \{-N+1,\ldots,-h_0^k-2\},\\
  \bar{x}^k_{-N+1}-\frac{q_n}{\bar{\rho}_0^{k}} , & i=-N.
\end{cases}
\]
In other words, we maintain the same position for all the particles that are positioned in $\Omega$, plus the rightmost particle in $(-\infty,0]$ and the leftmost particle in $[1,\infty)$, and we move all the other particles to make them equidistant in order to match the updated boundary conditions for the density $\bar{\rho}^k_0$ and $\bar{\rho}^k_1$.

Then, the $(n+N+1)$ positions $\bar{x}_{-N}^k, \ldots, \bar{x}_{n}^k$ are taken as initial conditions of the following FTL system
\begin{equation}\label{eq:FTLB}
\begin{cases}
  \dot{x}_i^k(t) = v(R_i^k(t)), & t \in [k \, \tau_m, (k+1) \, \tau_m],\ i \in \{-N,\ldots,n-1\},\\
  \dot{x}_n^k(t) = v(\bar{\rho}_1^k), & t \in [k \, \tau_m, (k+1) \, \tau_m],\\
  x_i^k(k \, \tau_m)=\bar{x}_i^k , & \hphantom{t \in (k \, \tau_m, (k+1) \, \tau_m],\ } i \in \{-N,\ldots,n\},
\end{cases}
\end{equation}
where we have denoted
\[
\begin{cases}
R_i^k(t) \doteq \frac{\ell_n}{x_{i+1}(t)-x_i(t)},
&t \in [k \, \tau_m, (k+1) \, \tau_m],~
i \in \{-N+1,\ldots,n-1\},\\
R_{-N}^k(t)\doteq \frac{q_n}{x_{-N+1}(t)-x_{-N}(t)},
&t \in [k \, \tau_m, (k+1) \, \tau_m].
\end{cases}
\]
We observe that the number of queuing particles $N$ has been chosen in order to guarantee that a number of particles of the order $N/2$ (as $n\rightarrow \infty$) will not cross $x=0$ within the time interval $[0,T]$.

\begin{remark}\label{rem:rem}
Our choice for the above particle scheme is motivated as follows. In order to approach the entropy solution according to Definition~\ref{def:entropy_sol_lef} in the $n\rightarrow \infty$ limit, on each time interval $[k\,\tau_m,(k+1)\,\tau_m)$ we tend to the entropy solution with constant boundary conditions by `extending' the discrete particle density at time $t = k\,\tau_m$ in a way to match said boundary conditions, see a similar construction in e.g.\ \cite{ColomboRosiniboundary}.
\end{remark}

The following discrete maximum-minimum principle ensures that the particles strictly preserve their initial order.

\begin{lemma}[Discrete maximum-minimum principle]
Assume \eqref{I3}, \eqref{V1} and \eqref{B1}.
Then, the solution to \eqref{eq:FTLB} satisfies for any $t \ge 0$
\begin{align}\label{eq:estB}
&\frac{\ell_n}{R}\leq x_{i+1}(t)-x_i(t)\leq \frac{\ell_n}{\delta},
&i \in \{ -N,\ldots,n-1 \}.
\end{align}
\end{lemma}

\begin{proof}
The lower bound on the time interval $[0,\tau_m)$ is a consequence of the result in Lemma~\ref{lem:maximum}.
We now prove the upper bound on $[0,\tau_m)$.
We consider first $i = n-1$.
By contradiction assume that there exist $t_1,t_2 \in (0,\tau_m)$ such that $t_1<t_2$, $x_n(t)-x_{n-1}(t) \le \ell_n/\delta$ for $t < t_1$, $x_n(t_1)-x_{n-1}(t_1) = \ell_n/\delta$ and $x_n(t)-x_{n-1}(t)> \ell_n/\delta$ for $t \in (t_1,t_2)$.
Then, for any $t\in(t_1,t_2)$ we have $\bar{\rho}^0_1 \geq \delta > R_{n-1}^0(t)$ and therefore
\begin{align*}
x_n(t)-x_{n-1}(t) &=
x_n(t_1) - x_{n-1}(t_1) + \int_{t_1}^{t} [v(\bar{\rho}^0_1)-v(R_{n-1}^0(s))]~ {\d}s 
= \frac{\ell_n}{\delta} + \int_{t_1}^{t} [v(\bar{\rho}^0_1)-v(R_{n-1}^0(s))]~ {\d}s
\le \frac{\ell_n}{\delta},
\end{align*}
which gives a contradiction.
We prove now the upper bound on $[0,\tau_m)$ for all the other vehicles inductively.
Assume
\[
\sup_{t\in[0,\tau_m)}[x_{i+2}(t)-x_{i+1}(t)]\leq \frac{\ell_n}{\delta},
\]
and by contradiction that there exist $t_1,t_2 \in (0,\tau_m)$ such that $t_1<t_2$, $x_{i+1}(t)-x_i(t) \le \ell_n/\delta$ for $t < t_1$, $x_{i+1}(t_1)-x_i(t_1) = \ell_n/\delta$ and $x_{i+1}(t) - x_i(t) > \ell_n/\delta$ for $t\in(t_1,t_2)$.
Then, for any $t\in(t_1,t_2)$ we have $R_{i+1}^0(t) \ge \delta > R_i^0(t)$ and therefore
\[
x_{i+1}(t)-x_i(t)=x_{i+1}(t_1)-x_i(t_1)+\int_{t_1}^t [v(R_{i+1}^0(s))-v(R_i^0(s))]~ {\d}s
\le \frac{\ell_n}{\delta}
\]
which gives a contradiction.
This proves the assertion on $[0,\tau_m]$.
Now, at each time step $t = k\,\tau_m$ the set of particles is rearranged outside $\Omega$ in such a way that two consecutive particles satisfy
\[
\begin{cases}
x_{i+1}(k\,\tau_m)-x_i(k\,\tau_m) = \frac{\ell_n}{\bar{\rho}_1^k} \le \frac{\ell_n}{\delta},
&i \in \{ n-h_1^k+1,\ldots,n-1\},
\\
x_{i+1}(k\,\tau_m)-x_i(k\,\tau_m) = \frac{\ell_n}{\bar{\rho}_0^k} \le \frac{\ell_n}{\delta},
&i \in \{ -N+1,\ldots,-h_0^k-1\},
\\
x_{-N+1}(k\,\tau_m)-x_{-N}(k\,\tau_m) = \frac{q_n}{\bar{\rho}_0^k} \le \frac{\ell_n}{\delta}.
\end{cases}
\]
Inside the domain $\Omega$, the inequalities in \eqref{eq:estB} are satisfied due to the maximum-minimum principle holding on the previous time interval.
Hence, we can reapply inductively the above procedure and easily get the assertion.
\end{proof}

We define the discrete density for $t \in (0,T]$ as
\begin{align*}
&\rho^{n,m}(t,x) \doteq \sum_{m=0}^{m-1}\,\sum_{i=-N}^{n-1} R_i^k(t) \, \mathbf{1}_{[x_i(t),x_{i+1}(t))}(x) \, \mathbf{1}_{(k \, \tau_m, (k+1) \, \tau_m]}(t).
\end{align*}
It is easy to verify that $\|\rho^{n,m}(t,\cdot)\|_{\L1(\Omega)} = Q+L$ for all $t\geq 0$.
We state the main result of this section, as well as the main novel result of this chapter.
\begin{theorem}\label{thm:main2}
Assume \eqref{I3}, \eqref{V1} and \eqref{B1}.
Then, $(\rho^{n,m} \, \mathbf{1}_{\Omega})_{n,m}$ converges (up to a subsequence) a.e.\ and in $\L1$ on $\R_+ \times \Omega$ to a weak solution $\rho$ to the IBVP \eqref{eq:ibvp} in the interior of $\Omega$.
\end{theorem}

Our conjecture is that the limit $\rho$ is in fact the unique entropy solution to the IBVP \eqref{eq:ibvp} in the sense of Definition~\ref{def:entropy_sol_lef}.
This is motivated by the construction of our FTL approximation scheme, which relies on the Definition~\ref{def:entropy_sol_lef}, see Remark~\ref{rem:rem}.
Moreover, the numerical simulations performed in Subsection~\ref{sec:simibvp} suggest it. We rigorously prove the consistency of the scheme only in simple cases in Subsection~\ref{subsec:consistency} below.

The fact that the limit $\rho$ in the statement of Theorem~\ref{thm:main2} is a weak solution to the LWR equation in the interior of $\Omega$ can be easily proven as in the proof of Theorem~\ref{thm:main_LWR}, and therefore we omit the details. Hence, we only need to prove convergence of the sequence $\rho^{n,m}$ strongly in $\L1$ up to a subsequence. This task is the goal of the next section.

\begin{remark}\label{rem:zero_dirichlet}
As already explained in the introduction, we recall that the boundary condition does not need to be updated in time as long as no waves coming from $\Omega$ hit the boundary $\partial\Omega$.
In particular, if $\bar{\rho}$, $\bar{\rho}_0$ and $\bar{\rho}_1$ are constant, the solution is simply obtained as the restriction to $\Omega$ of the entropy solution to the Cauchy problem \eqref{eq:cauchy} with initial condition $\bar{\rho}_0 \, \mathbf{1}_{(-\infty,0)} + \bar{\rho} \, \mathbf{1}_{\Omega} + \bar{\rho}_1 \, \mathbf{1}_{(1,\infty)}$ and no update in time of the boundary data is needed.
There are other cases in which $\bar{\rho}$, $\bar{\rho}_0$ and $\bar{\rho}_1$ are not necessarily constant and such situation occurs.
We highlight one of them here in the special case $v(\rho) \doteq 1-\rho$, which yields $f(\rho) \doteq \rho \, (1-\rho)$.
Indeed, a very simple argument based on the WFT approximation (see e.g.\ \cite{ColomboRosiniboundary}) shows that for any $h \in [0,1]$, if we denote with $\rho_k$ the restriction to $\Omega$ of the solution to the Riemann problem with initial datum $h \, \mathbf{1}_{(-\infty,1)} + k \, \mathbf{1}_{(1,\infty)}$, then $\rho_{k} = \rho_{1/2}$ for all $k \in [0,1/2]$.
Arguing in a similar way for $x=0$, it is easy to see that if the boundary data $\bar{\rho}_0$ and $\bar{\rho}_1$ take values in $[1/2,1]$ and $[0,1/2]$ respectively, then no updates of the boundary data is needed.
\end{remark}

\subsection{Estimates}

Similarly to Section~\ref{sec:LWRestimates}, the proof of Theorem~\ref{thm:main2} is based on some estimates which infer suitable space-time compactness.
We now prove the following $\BV$ estimate for $(\rho^{n,m} \, \mathbf{1}_{\Omega})_{n,m}$.

\begin{proposition}
Assume \eqref{I3}, \eqref{V1} and \eqref{B1}.
Then for any $t>0$
\[\tv(\rho^{n,m}(t,\cdot)) \le C,\]
where $C \doteq \tv(\bar{\rho})
+\tv(\bar{\rho}_0)
+\tv(\bar{\rho}_1)
+\left|\bar{\rho}_0(0^+) - \bar{\rho}(0^+)\right|
+\left|\bar{\rho}_1(0^+) - \bar{\rho}(1^-)\right|$.
\end{proposition}

\begin{proof}
We define $\Upsilon : (0,T] \to \R_+$ by letting for any $t \in (k \, \tau_m, (k+1) \, \tau_m]$ with $k \in \N$
\begin{align*}
\Upsilon(t) \doteq&~
|R_{-h_0^{k+1}-2}^k(t) - \bar{\rho}_0^k|
+|R_{n-h_1^k+1}^k(t) - \bar{\rho}_1^k|
+\sum_{i = -h_0^{k+1}-2}^{n-h_1^k} |R_{i+1}^k(t) - R_{i}^k(t)|
+\sum_{j = k}^{m - 1} \left[ |\bar{\rho}_0^{j+1} - \bar{\rho}_0^{j}| + |\bar{\rho}_1^{j+1} - \bar{\rho}_1^{j}| \right]
.
\end{align*}
We observe that
\begin{align*}
&R_{n-h_1^k+1}^k(t)=R_{n-h_1^k+2}^k(t)=\ldots = R_{n-1}^k(t)=\bar{\rho}^k_1
&\text{for all } t\in [k\,\tau_m,(k+1)\,\tau_m],
\end{align*}
due to the fact that the above quantities are all equal at time $t=k\,\tau_m$ and the leader $x_n$ travels with speed $v(\bar{\rho}^k_1)$ during the whole time interval.

We claim that $\Delta\Upsilon(t) \doteq \Upsilon(t^+) - \Upsilon(t^-) \le 0$ for all $t \in (0,T)$.
Let us first consider $t \in (k\,\tau_m,(k+1) \, \tau_m)$. In this case
\begin{align*}
\dot{\Upsilon}(t) =&~
\sign\bigl(R_{-h_0^{k+1}-2}^k(t) - \bar{\rho}_0^k\bigr) \, \dot{R}_{-h_0^{k+1}-2}^k(t)
+\sign\bigl(R_{n-h_1^k+1}^k(t) - \bar{\rho}_1^k\bigr) \, \dot{R}_{n-h_1^k+1}^k(t)
\\&~
+\sum_{i = -h_0^{k+1}-2}^{n-h_1^k} \sign\bigl(R_{i+1}^k(t) - R_{i}^k(t)\bigr) \, \bigl(\dot{R}_{i+1}^k(t) - \dot{R}_{i}^k(t) \bigr)
\\=&~
\Bigl[ \sign\bigl(R_{-h_0^{k+1}-2}^k(t) - \bar{\rho}_0^k\bigr)
-\sign\bigl(R_{-h_0^{k+1}-1}^k(t) - R_{-h_0^{k+1}-2}^k(t)\bigr) \Bigr] \dot{R}_{-h_0^{k+1}-2}^k(t)
\\&~
+ \Bigl[ \sign\bigl(R_{n-h_1^k+1}^k(t) - \bar{\rho}_1^k\bigr)
+ \sign\bigl(R_{n-h_1^k+1}^k(t) - R_{n-h_1^k}^k(t)\bigr) \Bigr] \dot{R}_{n-h_1^k+1}^k(t)
\\&~
+\sum_{i = -h_0^{k+1}-1}^{n-h_1^k} \Bigl[ \sign\bigl(R_{i}^k(t) - R_{i-1}^k(t)\bigr)
-\sign\bigl(R_{i+1}^k(t) - R_{i}^k(t)\bigr) \Bigr] \dot{R}_{i}^k(t)
\le 0.
\end{align*}
The above estimate holds because the quantities
\begin{align*}
&~
\Bigl[ \sign\bigl(R^k_{-h_0^{k+1}-2}(t) - \bar{\rho}_0^k\bigr)
- \sign\bigl(R^k_{-h_0^{k+1}-1}(t) - R^k_{-h_0^{k+1}-2}(t)\bigr) \Bigr] \dot{R}^k_{-h_0^k-2}(t)
\\=&~
\Bigl[ \sign\bigl(\bar{\rho}_0^k - R^k_{-h_0^{k+1}-2}(t)\bigr)
+ \sign\bigl(R^k_{-h_0^{k+1}-1}(t) - R^k_{-h_0^{k+1}-2}(t)\bigr) \Bigr]
\frac{R^k_{-h_0^{k+1}-2}(t)^2}{\ell_n} \Bigl[v(R^k_{-h_0^{k+1}-1}(t)) - v(R^k_{-h_0^{k+1}-2}(t))\Bigr],
\\[10pt]&~
\Bigl[ \sign\bigl(R^k_{n-h_1^k+1}(t) - \bar{\rho}_1^k\bigr)
+ \sign\bigl(R^k_{n-h_1^k+1}(t) - R^k_{n-h_1^k}(t)\bigr) \Bigr] \dot{R}^k_{n-h_1^k+1}(t)
\\=&~
\Bigl[ \sign\bigl(\bar{\rho}_1^k-R^k_{n-h_1^k+1}(t)\bigr)
+ \sign\bigl(R^k_{n-h_1^k}(t)-R^k_{n-h_1^k+1}(t)\bigr) \Bigr]
\frac{R^k_{n-h_1^k+1}(t)^2}{\ell_n}
\Bigl[v(\bar{\rho}^k_1)) - v(R^k_{n-h_1^k+1}(t))\Bigr],
\\[10pt]&~
\Bigl[ \sign\bigl(R^k_{i}(t) - R^k_{i-1}(t)\bigr)
-\sign\bigl(R^k_{i+1}(t) - R^k_{i}(t)\bigr) \Bigr] \dot{R}^k_{i}(t)
\\=&~
\Bigl[ \sign\bigl(R^k_{i-1}(t) - R^k_{i}(t)\bigr)
+\sign\bigl(R^k_{i+1}(t) - R^k_{i}(t)\bigr) \Bigr] \frac{R^k_{i}(t)^2}{\ell_n} \Bigl[v(R^k_{i+1}(t)) - v(R^k_{i}(t))\Bigr]
\end{align*}
are not positive.
Consider now  $t = (k+1) \, \tau_m$ with $k \in \N$.
In this case, using that
\begin{align*}
R^{k+1}_{-h_{k+2}-2}(t^+)=R^{k+1}_{-h_{k+2}-1}(t^+)=\ldots=R^{k+1}_{-h_0^{k+1}-2}(t^+)=\bar{\rho}_0^{k+1},
\\
R^{k+1}_{n-h_1^{k+1}+1}(t^+)= R^{k+1}_{n-h_1^{k+1}+2}(t^+)= \ldots = R^{k+1}_{n-h_1^k+1} (t^+)=  \bar{\rho}^{k+1}_1,
\end{align*}
we easily obtain
\begin{align*}
\Delta\Upsilon(t) =&~
-|R_{-h_0^{k+1}-2}^k(t) - \bar{\rho}_0^{k}|
-|R_{n-h_1^k+1}^k(t) - \bar{\rho}_1^{k}|
+|R_{n-h_1^{k+1}+1}^{k+1}(t)-R_{n-h_1^{k+1}}(t)| - |R_{n-h_1^{k+1}+1}^k(t)-R_{n-h_1^{k+1}}(t)|
\\
&~-\sum_{j=1}^{h_1^{k+1}-h_1^k}|R_{n-h_1^{k+1}+j+1}^k(t)-R_{n-h_1^{k+1}+j}^k(t)|
\\
&~
+|R_{-h_0^{k+1}-1}(t) - R_{-h_0^{k+1}-2}^{k+1}(t)|
-|R_{-h_0^{k+1}-1}(t) - R_{-h_0^{k+1}-2}^k(t)|
-|\bar{\rho}_0^{k+1} - \bar{\rho}_0^{k}|
-|\bar{\rho}_1^{k+1} - \bar{\rho}_1^{k}|
\\=&~
\Bigl[
|R_{-h_0^{k+1}-1}(t) - \bar{\rho}_0^{k+1}|
-|R_{-h_0^{k+1}-1}(t) - R_{-h_0^{k+1}-2}^k(t)|
-|R_{-h_0^{k+1}-2}^k(t) - \bar{\rho}_0^{k}|
-|\bar{\rho}_0^{k} - \bar{\rho}_0^{k+1}|
\Bigr]
\\
&~
+
\Bigl[
|R_{n-h_1^{k+1}}(t) - \bar{\rho}_1^{k+1}|
-|\bar{\rho}_1^{k} - \bar{\rho}_1^{k+1}|-|R_{n-h_1^k+1}^k(t) - \bar{\rho}_1^{k}|
- \sum_{j=0}^{h_1^{k+1}-h_1^k}|R_{n-h_1^{k+1}+j+1}^k(t)-R_{n-h_1^{k+1}+j}^k(t)|
\Bigr]
\le 0,
\end{align*}
where the last inequality follows by a simple triangular inequality.
In conclusion we have that
\begin{align*}
&~ \tv(\rho^{n,m}(t,\cdot))
\le \Upsilon(t) \le \Upsilon(0^+)
= \left|\frac{\ell_n}{1-\bar{x}_{n-1}^0} - \bar{\rho}_1^0\right|
+\sum_{i = -1}^{n-2} \left| \frac{\ell_n}{\bar{x}_{i+2}^0-\bar{x}_{i+1}^0} - \frac{\ell_n}{\bar{x}_{i+1}^0-\bar{x}_i^0} \right|
+\sum_{j = 0}^{m - 1} \left[ |\bar{\rho}_0^{j+1} - \bar{\rho}_0^{j}| + |\bar{\rho}_1^{j+1} - \bar{\rho}_1^{j}| \right]
\\
\le&~ \left|\fint_{\bar{x}_{n-1}^0}^{1} \bar{\rho}(x) ~ {\d}x - \bar{\rho}_1(0)\right|
+\sum_{i = 0}^{n-2} \left| \fint^{\bar{x}_{i+2}^0}_{\bar{x}_{i+1}^0} \bar{\rho}(x) ~ {\d}x - \fint_{\bar{x}_{i}^0}^{\bar{x}_{i+1}^0} \bar{\rho}(x) ~ {\d}x \right|
+\left|\fint^{\bar{x}_{1}^0}_{0} \bar{\rho}(x) ~ {\d}x - \bar{\rho}_0(0)\right|
+\tv(\bar{\rho}_0)+\tv(\bar{\rho}_1)
\le C.
\end{align*}
\end{proof}

We provide now a uniform time continuity estimate with respect to the rescaled $1$-Wasserstein distance $W_{Q+L,1}$ defined in \eqref{eq:wass_equiv0}.

\begin{proposition}
Assume \eqref{I3}, \eqref{V1} and \eqref{B1}.
Then the sequence $(\rho^{n,m})_{n,m}$ satisfies \eqref{Bbis}, which in the present framework writes
\begin{equation}\label{eq:time_continuity}
  \lim_{h\downarrow 0}\left[ \sup_{n,\,m\in \N} \left[ \int_0^{T-h}W_{Q+L,1}(\rho^{n,m}(t+h),\rho^{n,m}(t))~ {\d}t \right] \right] = 0.
\end{equation}
\end{proposition}

\begin{proof}
For simplicity we drop the indexes $n,m$ in the notation and use $W_1$ instead of $W_{Q+L,1}$.
The above Wasserstein distance is computed via the pseudo-inverse variable
\begin{align*}
  &X_{\rho(t,\cdot)}(z)=
  \Bigl[x_{-N}(t) + z \, R_{-N}(t)^{-1}\Bigr]\mathbf{1}_{[0,q_n)}(z)
  +\!\sum_{i=-N+1}^{n-1} \Bigl[x_{i}(t) + \Bigl(z- \bigl(q_n+(i+N-1) \, \ell_n \bigr) \Bigr) R_{i}(t)^{-1}\Bigr] \mathbf{1}_{[q_n+(i+N-1) \, \ell_n,q_n+(i+N) \, \ell_n)}(z).
\end{align*}
We recall that, for all $t\geq 0$, $X_{\rho(t,\cdot)}$ is a strictly increasing function on $[0,Q+L]$.

For $k\,\tau_m<s<t<(k+1) \, \tau_m$ we compute
\begin{align}
&~ W_1(\rho(t,\cdot),\rho(s,\cdot))
=  \|X_{\rho(t,\cdot)}-X_{\rho(s,\cdot)}\|_{\L1([0,Q+L])}
= \int_0^{q_n}\left|x_{-N}(t)-x_{-N}(s) + z \, (R_{-N}(t)^{-1}-R_{-N}(s)^{-1})\right|~ {\d}z
\nonumber\\
&~ + \sum_{i=-N+1}^{n-1} \int_{0}^{\ell_n} \, \left|x_i(t)-x_i(s) + z \, \bigl(R_i(t)^{-1} - R_i(s)^{-1}\bigr)\right|~ {\d}z
\leq q_n \, |x_{-N}(t)-x_{-N}(s)|+ |R_{-N}(t)^{-1}-R_{-N}(s)^{-1}| \, \int_0^{q_n}z ~ {\d}z
\nonumber\\
&~ +\sum_{i=-N+1}^{n-1} \ell_n \, |x_i(t)-x_i(s)|
+ \sum_{i=-N+1}^{n-1} \left|R_i(t)^{-1} - R_i(s)^{-1}\right| \int_{0}^{\ell_n} z ~ {\d}z
\nonumber\\
\leq&~ (Q+L) \, v_{\max} \, |t-s|
+ \frac{q_n^2}{2}\int_s^t\left|\frac{\d}{\d\tau}\left[\frac{1}{R_{-N}(\tau)}\right]\right| {\d}\tau+  \sum_{i=-N+1}^{n-1} \frac{\ell_n^2}{2} \int_{s}^{t}\left|\frac{\d}{\d\tau}\left[\frac{1}{R_i(\tau)}\right]\right| {\d}\tau
\nonumber\\=&~
(Q+L) \, v_{\max} \, |t-s| + \frac{q_n}{2} \int_{s}^{t} |v(R_{-N+1}(\tau))-v(R_{-N}(\tau))| \, \d\tau
+\sum_{i=-N}^{n-2} \frac{\ell_n}{2} \int_{s}^{t} |v(R_{i+1}(\tau))-v(R_i(\tau))| \, \d\tau
\nonumber\\
&~
+\frac{\ell_n}{2} \int_{s}^{t} |v(\bar{\rho}_1^k) - v(R_{n-2}(\tau))| \, \d\tau
\leq
\frac{3}{2}(Q +L) \, v_{\max} \, (t-s).
\label{FabioAru}
\end{align}
As a consequence of the above computation, the curve $[0,T]\ni t\mapsto \rho^{n,m}(t,\cdot)$ is equi-continuous in the $W_1$-topology on open intervals of the form $(k\,\tau_m,(k+1) \, \tau_m)$ and
\begin{align}\label{eq:continuity_open}
&W_1(\rho(((k+1) \, \tau_m)^-,\cdot),\rho((k \, \tau_m)^+,\cdot)) \le C \, \tau_m,
&k \in \N,
\end{align}
where $C$ is some positive constant independent of $n$, $m$, and $h$.
On the other hand, due to the rearrangements of the particles outside $\Omega$ at each time step, such curve may feature a jump discontinuity.
Let $t = (k+1)\,\tau_m$.
We estimate the jump
\begin{align}
&~ W_1(\rho(t^+,\cdot),\rho(t^-,\cdot)) =  \|X_{\rho(t^+,\cdot)}-X_{\rho(t^-,\cdot)}\|_{\L1([0,L])}
\nonumber\\&~
\leq q_n \, |x_{-N}(t^+)-x_{-N}(t^-)|
+ \sum_{i=-N+1}^{-h_0^{k+1}-2} \ell_n \, |x_i(t^+)-x_i(t^-)| +\sum_{i=n-h_1^{k+1}+2}^{n-1} \ell_n \, |x_i(t^+)-x_i(t^-)|
\nonumber\\&~
+ \left|(\bar{\rho}_0^{k+1})^{-1} - R_{-N}(t^-)^{-1}\right| \frac{\ell_n^2}{2}+ \sum_{i=-N+1}^{-h_0^{k+1}-2} \left|(\bar{\rho}_0^{k+1})^{-1} - R_i(t^-)^{-1}\right|  \frac{\ell_n^2}{2}
+ \sum_{i=n-h_1^{k+1}+1}^{n-1} \left|(\bar{\rho}^{k+1}_1)^{-1}-R_{i}(t^-)^{-1})\right|\frac{\ell_n^2}{2},
\label{eq:dirichlet_continuity1}
\end{align}
where we have use the fact that $t \mapsto R_i(t)^{-1}$ is continuous for all $i \in \{-h_0^{k+1}-1,\ldots,n-h_1^{k+1}\}$.
We claim that for any $i \in \{-N,\ldots,-h_0^{k+1}-2\}$ we have the estimate
\begin{equation}
|x_i(t^+)-x_i(t^-)|
\leq
\left[2 \, v_{\max} + \frac{Q}{\delta^2} \, \lip(\bar{\rho}_0)\right] \tau_m .\label{eq:dirichlet_continuity2}
\end{equation}
Indeed, for any $i \in \{-N+1,\ldots,-h_0^{k+1}-2\}$ we have
\begin{align*}
&~ |x_i(t^+)-x_i(t^-)|\leq |x_i(t^+)-x_i((t-\tau_m)^+)| + |x_i((t-\tau_m)^+)-x_i(t^-)|
\\=&~
\left|x_{-h_0^{k+1}-1}(t) + (i+h_0^{k+1}+1) \, \frac{\ell_n}{\bar{\rho}_0^{k+1}}
-x_{-h_0^{k+1} -1}((t-\tau_m)^+) - (i+h_0^{k+1}+1) \, \frac{\ell_n}{\bar{\rho}^k_0}\right|
+\int_{t-\tau_m}^t v(R_i(s))~ {\d}s
\\\le&~
\int_{t-\tau_m}^t v(R_{-h_0^{k+1}-1}(s))~ {\d}s
+\ell_n \left|i+h_0^{k+1}+1\right| \left|\frac{1}{\bar{\rho}^{k+1}_0}-\frac{1}{\bar{\rho}^{k}_0}\right|
+\int_{t-\tau_m}^t v(R_{i}(s))~ {\d}s
\\\le&~
2 \, \tau_m \, v_{\max} + Q \left|\frac{1}{\bar{\rho}^{k+1}_0}-\frac{1}{\bar{\rho}^{k}_0}\right|
\le \left[2 \, v_{\max} + \frac{Q}{\delta^2} \, \lip(\bar{\rho}_0)\right] \tau_m,
\end{align*}
and analogously
\begin{align*}
&~ |x_{-N}(t^+)-x_{-N}(t^-)| \leq |x_{-N}(t^+)-x_{-N}((k\,\tau_m)^+)| + |x_{-N}((k\,\tau_m)^+)-x_{-N}(t^-)|
\\=&~
\Biggl|x_{-h_0^{k+1}-1}(t)
+ (i+h_0^{k+1}+1) \, \frac{\ell_n}{\bar{\rho}_0^{k+1}}
- \frac{q_n}{\bar{\rho}_0^{k+1}}
-x_{-h_0^{k+1} -1}((t-\tau_m)^+)
- (i+h_0^{k+1}+1) \, \frac{\ell_n}{\bar{\rho}^k_0}
+ \frac{q_n}{\bar{\rho}_0^{k}}\Biggr|
+\int_{t-\tau_m}^t v(R_i(s))~ {\d}s
\\\le&~
\int_{t-\tau_m}^t \left[ v(R_{-h_0^{k+1}-1}(s)) + v(R_{i}(s)) \right] {\d}s
+\left[ \ell_n \left|i+h_0^{k+1}+1\right| + q_n \right] \left|\frac{1}{\bar{\rho}^{k+1}_0}-\frac{1}{\bar{\rho}^{k}_0}\right|
\\\le&~
2 \, \tau_m \, v_{\max} + Q \left|\frac{1}{\bar{\rho}^{k+1}_0}-\frac{1}{\bar{\rho}^{k}_0}\right|
\le \left[2 \, v_{\max} + \frac{Q}{\delta^2} \, \lip(\bar{\rho}_0)\right] \tau_m.
\end{align*}
Moreover, for all $i \in \{-N+1,\ldots,-h_0^{k+1}-2\}$,
\begin{align}
  \left|(\bar{\rho}_0^{k+1})^{-1} - R_i(t^-)^{-1}\right|
  \leq&~
  \left|(\bar{\rho}_0^{k+1})^{-1}-(\bar{\rho}_0^{k})^{-1}\right| + \left|(R_i((t-k\,\tau_m)^+)^{-1}-R_i(t^-)^{-1}\right|
  \nonumber\\\leq&~
  \left|(\bar{\rho}_0^{k+1})^{-1}-(\bar{\rho}_0^{k})^{-1}\right| + \frac{1}{\ell_n}\int_{t-\tau_m}^t |v(R_{i+1}(s))-v(R_i(s))|~ {\d}s
  \le
  \left[\frac{v_{\max}}{\ell_n} + \frac{1}{\delta^2} \, \lip(\bar{\rho}_0)\right] \tau_m
  ,\label{eq:dirichlet_continuity3}
\end{align}
and the same estimate holds for $i=-N$ with $q_n$ replacing $\ell_n$.
For any $i \in \{n-h_1^{k}+2,\ldots,n-1\}$, we estimate
\begin{align}
  &~|x_i(t^+)-x_i(t^-)|\leq |x_i(t^+)-x_i((t-\tau_m)^+)| + |x_i((t-\tau_m)^+)-x_i(t^-)|
  \nonumber\\=&~
  \Biggl|x_{n-h_1^{k+1}+1}(t)+(i-n+h_1^{k+1}-1) \, \frac{\ell_n}{\bar{\rho}_1^{k+1}}
  -x_{n-h_1^k+1}(t-\tau_m)-(i-n+h_1^k-1) \, \frac{\ell_n}{\bar{\rho}^k_1}\Biggr|
  +\int_{t-\tau_m}^t v(R_i(s))\d s
  \nonumber\\\leq&~
  \left|x_{n-h_1^{k+1}+1}(t) - x_{n-h_1^k+1}(t-\tau_m)\right|
  +(Q+L)\left|\frac{1}{\bar{\rho}^{k+1}_1}-\frac{1}{\bar{\rho}^{k}_1}\right|
  +\left(h_1^{k+1}-h_1^k\right) \frac{\ell_n}{\bar{\rho}^{k+1}_1}+ \tau_m \, v_{\max}
  \nonumber\\\leq&~
  \left|x_{n-h_1^{k+1}+1}(t) - x_{n-h_1^{k+1}+1}(t-\tau_m)\right|
  +\sum_{i=n-h_1^{k+1}+1}^{n-h_1^k} \left|x_{i+1}(t-\tau_m)-x_i(t-\tau_m)\right|
  +(Q+L) \, \frac{\tau_m}{\delta^2} \, \lip(\bar{\rho}_1)
  +\frac{\tau_m}{\delta} \, v_{\max} \, \rho_{\max}
  + \tau_m \, v_{\max}
  \nonumber\\\leq&~
  \left[\frac{Q+L}{\delta^2} \, \lip(\bar{\rho}_1)
  +2 \, \frac{v_{\max} \, \rho_{\max}}{\delta}
  +2 \, v_{\max}\right] \tau_m,\label{eq:dirichlet_continuity_3bis}
\end{align}
where we have used the minimum principle $R_i(t)\geq \delta$ for all $t\geq 0$ given in Lemma~\ref{lem:maximum}, and (twice) the estimate
\[h_1^{k+1}-h_1^k \leq \frac{\tau_m \, v_{\max}}{\ell_n/\rho_{\max}},\]
which expresses the fact that the total number of particles crossing a given point on a time interval of size $\tau_m$ is bounded by the maximum distance covered, i.e.\ $\tau_m \, v_{\max}$, divided by the smallest possible distance between two consecutive vehicles, i.e.\ $\ell_n/\rho_{\max}$.
Finally, by a similar procedure as in \eqref{eq:dirichlet_continuity3}, we estimate for $i \in \{n-h_1^{k+1}+1,\ldots,n-1\}$
\begin{align*}
  &~\Bigl|(\bar{\rho}^{k+1}_1)^{-1}-R_i(t^-)^{-1}\Bigr|
  \leq
  \Bigl|R_i(t^-)^{-1}-R_i((t-\tau_m)^+)^{-1}\Bigr| + \Bigl|R_i((t-\tau_m)^+)^{-1}- (\bar{\rho}^{k+1}_1)^{-1}\Bigr|
  \\\leq&~
  \frac{1}{\ell_n}\int_{t-\tau_m}^{t} \Bigl|v(R_{i+1}(s))-v(R_i(s))\Bigr| {\d}s
  + \Bigl|R_i((t-\tau)^+)^{-1}- (\bar{\rho}^{k+1}_1)^{-1}\Bigr|
  \leq
  \frac{v_{\max}}{\ell_n} \, \tau_m
  + \Bigl|R_i((t-\tau)^+)^{-1}- (\bar{\rho}^{k+1}_1)^{-1}\Bigr|.
\end{align*}
Now the last term on the right-hand-side of the above last estimate can be controlled in the case $i\geq n-h_1^k+1$ by
\begin{equation}\label{eq:dirichlet_continuity2bis_intermediate}
\Bigl|R_i((t-\tau)^+)^{-1}- (\bar{\rho}^{k+1}_1)^{-1}\Bigr|
=
\Bigl|(\bar{\rho}^{k+1}_1)^{-1}-(\bar{\rho}^{k}_1)^{-1}\Bigr|
\le
\frac{1}{\delta^2} \, \lip(\bar{\rho}_1) \, \tau_m,
\end{equation}
while in the case $i<n-h_1^k+1$ by
\begin{equation}
\Bigl|R_i((t-\tau)^+)^{-1}- (\bar{\rho}^{k+1}_1)^{-1}\Bigr|
\leq
\sum_{j=n-h_1^{k+1}+1}^i \Bigl|R_{j+1}((t-\tau)^+)^{-1}-R_{j}((t-\tau)^+)^{-1} \Bigr|
\leq
\left(h_1^{k+1}-h_1^k\right) \frac{\rho_{\max}}{\delta^2}
\leq
\frac{v_{\max} \, \rho_{\max}^2}{\ell_n \, \delta^2} \, \tau_m.
\label{eq:dirichlet_continuity2bis}
\end{equation}
Hence, substituting \eqref{eq:dirichlet_continuity2}, \eqref{eq:dirichlet_continuity3}, \eqref{eq:dirichlet_continuity_3bis}, \eqref{eq:dirichlet_continuity2bis_intermediate} and \eqref{eq:dirichlet_continuity2bis} into \eqref{eq:dirichlet_continuity1}, using $q_n\leq \ell_n$ and the arbitrariness of $t = (k+1)\,\tau_m$, we can easily find a positive constant $C = C(\delta,\rho_{\max},v_{\max},T,\bar{\rho},\bar{\rho}_0,\bar{\rho}_1) \ge 0$ such that
\begin{align}
  &~ W_1(\rho(t^+,\cdot),\rho(t^-,\cdot))\leq C \, \tau_m,
  &t \in (\N+1)\,\tau_m.\label{eq:continuity_jump}
\end{align}
Now, we use the two estimates \eqref{eq:continuity_open} and \eqref{eq:continuity_jump} to obtain \eqref{eq:time_continuity}.
Let $h>0$ be fixed.
Let $t\in [0,T-h]$ and assume for simplicity that $t,t+h\not \in \{k\tau_n\}_{k=0}^{m-1}$.
We first assume $\tau_m<h$.
More precisely, let $t\in (k \, \tau_m,(k+1)\tau_m)$ and $t+h\in (r \, \tau_m,(r+1) \, \tau_m)$ for some $k<r<m$.
We have
\begin{align*}
   W_1(\rho(t+h),\rho(t))\leq&~
   W_1\bigl(\rho(t+h),\rho((r \, \tau_m)^+)\bigr)
   +W_1\bigl(\rho((r \, \tau_m)^+),\rho((r \, \tau_m)^-)\bigr)
   + W_1\bigl(\rho(((k+1)\tau_m)^-),\rho(t)\bigr)
   \\&~
   +\sum_{j=k+1}^{r-1} \Bigl[
   W_1\bigl(\rho((j+1) \, \tau_m)^-,\rho(j \, \tau_m)^+\bigr)
   + W_1\bigl(\rho(j \, \tau_m)^+,\rho((j \, \tau_m)^-\bigr)
   \Bigr]
   \le
   2 \, C \, \tau_m + 2 \, C \, (r-k-1) \, \tau_m + C \, \tau_m
   \\\le&~
   C  \, [2 \, (r-k) + 1] \, \tau_m
   \le 5 \, C \, h,
\end{align*}
because by assumption $\tau_m<h$ and $(r-k) \tau_m \leq h + \tau_m \leq 2\,h$.
Since $C$ is some positive constant independent of $n$, $m$, and $h$, we have \eqref{B}, hence \eqref{eq:time_continuity}.
Let us now assume $\tau_m\geq h$.
In this case we have by \eqref{FabioAru} and \eqref{eq:continuity_jump}
\begin{align*}
&~\int_0^{T-h} W_1(\rho(t+h),\rho(t))~ {\d}t =
\\=&~ \int_{(m-1)\tau_m}^{T-h} W_1(\rho(t+h),\rho(t))~ {\d}t
+\sum_{k=1}^{m-1} \left[\int_{(k-1)\tau_m}^{k \tau_m-h} W_1(\rho(t+h),\rho(t))~ {\d}t + \int_{k \, \tau_m-h}^{k \tau_m} W_1(\rho(t+h),\rho(t))~ {\d}t\right]
\\\le&~
\int_{(m-1)\tau_m}^{T-h} W_1(\rho(t+h),\rho(t))~ {\d}t
+\sum_{k=1}^{m-1} \int_{(k-1)\tau_m}^{k \tau_m-h} W_1(\rho(t+h),\rho(t))~ {\d}t
\\&~
+\sum_{k=1}^{m-1} \Biggl[\int_{k \, \tau_m-h}^{k \tau_m}
\Bigl(W_1(\rho(t+h),\rho((k \, \tau_m)^+))
+ W_1(\rho((k \, \tau_m)^+),\rho((k \, \tau_m)^-))
+ W_1(\rho((k \, \tau_m)^-),\rho(t)) \Bigr)~ {\d}t\Biggr]
\\\le&~
\frac{3}{2}(Q + L ) \, v_{\max} \, h \, \tau_m
+ 3 \sum_{k=1}^{m-1} \left[\frac{3}{2}(Q + L ) \, v_{\max} \, h \, \tau_m\right]
+ \sum_{k=1}^{m-1} \left[C \, h \, \tau_m\right]
\le
\left[\frac{9}{2}(Q + L ) \, v_{\max} + C\right] T \, h,
\end{align*}
for some positive constant $C$ independent of $n$, $m$, and $h$.
Hence, \eqref{eq:time_continuity} is proven.
\end{proof}

In order to conclude the proof of Theorem~\ref{thm:main2}, we can proceed exactly as in Theorem~\ref{thm:main_LWR} by using Theorem~\ref{thm:aubin}. We observe that condition \eqref{Bbis} of Theorem~\ref{thm:aubin} is used in this case in order to get a uniform continuity estimate in time.

\subsection{Convergence to entropy solutions}\label{subsec:consistency}

In this subsection we briefly point out that the scheme introduced in Subsection~\ref{subsec:scheme} is consistent in some simple cases.

\begin{theorem}\label{thm:minor}
  Assume \eqref{I3} and \eqref{V1}.
  If $\bar{\rho}_0$ and $\bar{\rho}_1$ are constant and
  \begin{itemize}
    \item either also $\bar{\rho}$ is constant,
    \item or $f'(\bar{\rho}_0(t))<0$ and $f'(\bar{\rho}_1(t))>0$ for all $t\ge 0$,
  \end{itemize}
then $(\rho^{n,m})_{n,m}$ converges (up to a subsequence) to the unique entropy solution to the IBVP \eqref{eq:ibvp} in the sense of Definition~\ref{def:entropy_sol_lef}.
\end{theorem}

\begin{proof}
The proof easily follows from the fact that in both cases the unique entropy solution to \eqref{eq:ibvp} on $[0,\tau]$ is the restriction of the solution to the Cauchy problem with initial condition $\bar{\rho}_0 \, \mathbf{1}_{(-\infty,0)} + \bar{\rho} \, \mathbf{1}_{\Omega} + \bar{\rho}_1 \, \mathbf{1}_{(1,\infty)}$.
This can be easily seen via a WFT argument, see e.g.\ \cite{ColomboRosiniboundary} and Remark~\ref{rem:zero_dirichlet}.
Hence, one can restart the Cauchy problem on $[\tau,2\tau]$ with the same construction and proceed iteratively for all times.
The above claim proves that for any fixed $m$, the limit $\rho^m$ of $\rho^{n,m}$ as $n\rightarrow\infty$ is an entropy solution to \eqref{eq:ibvp}.
The assertion then easily follows by the continuity with respect to the boundary conditions proven in \cite[Theorem~2.3.5b]{ColomboRosiniboundary}.
\end{proof}

\section{The Hughes model}\label{sec:hughes}

In this section we apply the Hughes model \cite{Hughes2002} to simulate the evacuation of a one-dimensional corridor $\Omega \doteq (-1,1)$ ending with two exits.
The resulting model is expressed by the following IBVP with Dirichlet boundary conditions
\begin{equation}\label{eq:model}
\begin{cases}
\rho_t - \left[\rho \, v(\rho) \, \frac{\phi_x}{|\phi_x|}\right]_x = 0,&
x \in \Omega,~ t > 0,
\\
|\phi_x|=c(\rho),&
x \in \Omega,~ t > 0,
\\
(\rho,\phi)(t,-1)=(\rho,\phi)(t,1)=(0,0),&
t>0,
\\
\rho(0,x)=\bar{\rho}(x),&
x \in \Omega.
\end{cases}
\end{equation}
We assume that the initial density $\bar{\rho}$ and the velocity map $v$ satisfy \eqref{I1} and \eqref{V1} respectively, where $\rho_{\max}$ is the maximal crowd density and $v_{\max}$ is the maximal speed of a pedestrian.
Let $L \doteq \|\bar{\rho}\|_{\L1(\Omega)}$ and $R \doteq \|\bar{\rho}\|_{\L\infty(\Omega)}$.
The maximum principle in \cite{El-KhatibGoatinRosini} shows that $\rho$ never exceeds the range $[0,R]$.
We assume also what follows.
\begin{gather}
\text{
There exists a $\hat\rho\in (0,\rho_{\max})$ such that $[v(\rho) + \rho \, v'(\rho)] (\hat\rho - \rho) > 0$ for all $\rho \in (0,\rho_{\max})\setminus\{\hat\rho\}$.
}\label{VH}\tag{V3}
\\
\begin{minipage}{.9\textwidth}\centering
$c \colon [0,\rho_{\max}] \to [1,\infty]$ is $\C2$, $c'\ge0$, $c''>0$, $c(0)=1$, and $c(R) < \infty$.
\end{minipage}\label{C}\tag{C}
\end{gather}

\begin{example}
In the literature, see \cite{AmadoriDiFrancesco, AmadoriGoatinRosini, BurgerDiFrancescoMarkowichWolfram, DiFrancescoMarkowichPietschmannWolfram, Hughes2002, Hughes2003, TwarogowskaGoatinDuvigneau}, the usual choice for the cost function is $c(\rho) \doteq 1/v(\rho)$.
In this case, in order to bypass the technical issue of $c$ blowing up at $\rho=\rho_{\max}$, it is assumed that $R \doteq \|\bar{\rho}\|_{\L\infty(\Omega)} \in (0,\rho_{\max})$.
This assumption, together with the maximum principle obtained in \cite{El-KhatibGoatinRosini}, ensures that the cost computed along any solution of \eqref{eq:model} is well defined.
\end{example}

As observed in \cite{AmadoriDiFrancesco, AmadoriGoatinRosini, El-KhatibGoatinRosini}, the differential equations in \eqref{eq:model} can be reformulated as
\begin{align}\label{eq:hughes_reformulated}
&\rho_t + F(t,x,\rho)_x = 0,&
&\int_{-1}^{\xi(t)} c(\rho(t,y)) ~{\d}y = \int_{\xi(t)}^1 c(\rho(t,y)) ~{\d}y,&
&x\in \Omega,\ t>0,
\end{align}
with $F(t,x,\rho) \doteq \mathrm{sign}(x-\xi(t)) \, f(\rho)$, see \cite{DF_fagioli_rosini_russo} for the details.
The form \eqref{eq:hughes_reformulated} clearly suggests that Hughes' model can be seen as a two-sided LWR model, with the turning point $\xi(t)$ splitting the whole interval $\Omega$ into two subintervals. For this reason, under appropriate assumptions that guarantee the presence of a persistent vacuum region around $\xi(t)$, we can apply the results obtained in Section~\ref{sec:LWR} to \eqref{eq:model}. The notion of solution in the case of a vacuum region around $t \mapsto \xi(t)$ is as follows
\begin{definition}\label{def:entropy_solution}
Assume \eqref{I1}, \eqref{V1}, \eqref{VH} and \eqref{C}.
A map $\rho\in \L\infty(\R_+\times\R;[0,R])$ is a (well-separated) \emph{entropy solution} to \eqref{eq:model} if
\begin{itemize}
\item There exists $\varepsilon>0$ such that $\rho$ is equal to zero on the open cone \[\mathcal{C}\doteq\left \{(t,x) \in \R_+ \times \R \colon |x-\bar\xi| < \varepsilon \, t\right \}.\]
  \item $\rho\,\mathbf{1}_{(-\infty,\bar\xi)}$ is the entropy solution to \eqref{eq:cauchy} with initial datum $\bar\rho\,\mathbf{1}_{(-\infty,\bar\xi)}$ in the sense of Definition~\ref{def:entro_sol_LWR}.
  \item $\rho\,\mathbf{1}_{(\bar\xi,\infty)}$ is the entropy solution to \eqref{eq:cauchy} with initial datum $\bar\rho\,\mathbf{1}_{(\bar\xi,\infty)}$ in the sense of Definition~\ref{def:entro_sol_LWR}.
  \item The turning curve $\mathcal{T}\doteq\left \{(t,x) \in \R_+ \times \Omega \colon x=\xi(t)\right \}$ is continuous and contained in $\mathcal{C}$.
  Moreover $(0,\bar{\xi}) \in \mathcal{T}$ and
      \begin{align*}
        &\int_{-1}^{\xi(t)}c(\rho(t,y)) \, {\d} y = \int_{\xi(t)}^1 c(\rho(t,y)) \, {\d} y,
        &\text{for a.e.\ }t\geq 0.
      \end{align*}
\end{itemize}
\end{definition}

The next theorem collects the main existence result obtained in \cite{AmadoriGoatinRosini}.

\begin{theorem}[{\cite[Theorem~3]{AmadoriGoatinRosini}}]
If $v(\rho) \doteq 1-\rho$, $c(\rho) \doteq 1/v(\rho)$ and the initial datum $\bar{\rho} \in \BV(\Omega;[0,1))$ satisfies the estimate $3 \, R + \tv(c(\bar{\rho})) + [c(\bar{\rho}(-1^+))-c(1/2)]_+ + [c(\bar{\rho}(1^-))-c(1/2)]_+ < 2$, then there exists an entropy solution to \eqref{eq:model} defined globally in time.
\end{theorem}

In Section~\ref{sec:numerics-Hughes} we show the numerical simulations of our particle methods in simple Riemann-type initial conditions.
We stress here that, although the analytical results concerning our deterministic particle method are restricted to cases in which each particle keeps the same direction for all times, the numerical simulations also cover cases with direction switching.

\subsection{The follow-the-leader scheme and main result}\label{sec:particle}

We now introduce our FTL scheme for \eqref{eq:model}.
Assume \eqref{I1}, \eqref{V1}, \eqref{VH} and \eqref{C}.
Fix $n\in \N$ sufficiently large and set $\ell_n \doteq L/n$.
Let $\bar{x}^n_0,\ldots,\bar{x}^n_n$ be defined recursively by
\[
\begin{cases}
\bar{x}^n_0 \doteq \min\left\{\supp(\bar{\rho})\right\},
\\
\bar{x}^n_i \doteq \inf\left\{ x > \bar{x}^n_{i-1} \colon \int_{\bar{x}^n_{i-1}}^x \bar{\rho}(y) ~{\d}y \geq m\right\},
&i\in\left\{1,\ldots,n\right\}.
\end{cases}
\]
It follows that $-1\le\bar{x}^n_0<\bar{x}^n_1<\ldots<\bar{x}^n_{n-1}<\bar{x}^n_n\le1$ and
\begin{align*}
&\int_{\bar{x}^n_i}^{\bar{x}^n_{i+1}} \bar{\rho}(y) ~{\d}y = \ell_n\leq (\bar{x}^n_{i+1}-\bar{x}^n_i) R,&
&i\in\left\{0,\ldots,n-1\right\}.
\end{align*}
We denote the local discrete initial densities
\begin{align*}
&\bar{R}^n_{i} \doteq \frac{\ell_n}{\bar{x}^n_{i+1}-\bar{x}^n_i} \in (0,R],&
&i\in\left\{0,\ldots,n-1\right\},
\end{align*}
and introduce the discretized initial density $\bar{\rho}^n \colon \R \to [0,\rho_{\max}]$ by
\[
\bar{\rho}^n(x) \doteq \sum_{i=0}^{n-1} \bar{R}^n_{i} \, \mathbf{1}_{[\bar{x}^n_i,\bar{x}^n_{i+1})}(x).
\]
We implicitly define the initial approximate turning point $\bar{\xi}^n \in \Omega$ via the formula
\[
\int_{-1}^{\bar{\xi}^n} c\left(\bar{\rho}^n(y)\right) ~{\d}y =
\int_{\bar{\xi}^n}^1 c\left(\bar{\rho}^n(y)\right) ~{\d}y.
\]

The next step is the definition of the evolving particle scheme.
Roughly speaking, $\bar{\xi}^n$ splits the set of particles into left and right particles, the former moving according to a \emph{backward} FTL scheme, the latter according to a \emph{forward} one.
By a slight modification of the initial condition, we may always assume that
there exists $I_0\in \{0,\ldots,n\}$ such that $\bar{\xi}^n \in (\bar{x}^n_{I_0},\bar{x}^n_{I_0+1})$.
We then set
\begin{equation}\label{eq:FTL-Hughes}
\begin{cases}
\dot{x}^n_0(t)= -v_{\max},\\
\dot{x}^n_i(t)=-v\left(\frac{\ell_n}{x^n_{i}(t)-x^n_{i-1}(t)}\right), &i \in \{1,\ldots, I_0\},\\
\dot{x}^n_i(t)=v\left(\frac{\ell_n}{x^n_{i+1}(t)-x^n_i(t)}\right), &i \in \{I_0+1,\ldots, n-1\},\\
\dot{x}^n_n(t)= v_{\max},\\
x^n_i(0)=\bar{x}^n_i,&i \in \{0,\ldots, n\}.
\end{cases}
\end{equation}
We consider the corresponding discrete densities
\[
\begin{cases}
R^n_{i}(t) \doteq \frac{\ell_n}{x^n_{i+1}(t)-x^n_i(t)},
&i \in \{0,\ldots,n-1\}\setminus\{I_0\},
\\
R^n_{i}(t) \doteq 0,
&i \in \{-1,I_0,n\}.
\end{cases}
\]
Notice that in view of Remark~\ref{rem:zero_dirichlet}, we do not impose any boundary condition in \eqref{eq:FTL-Hughes}, and we follow the movement of each particle whether or not they are in $\Omega$.
Moreover, the density has been set to equal zero outside $[x^n_0(t),x^n_n(t))$ and around the turning point, namely in $[x^n_{I_0}(t),x^n_{I_0+1}(t))$. The latter in particular is simply due to a consistency with the numerical simulations, in which the computation of the turning point is made simpler in this way. This clearly introduces an error $\ell_n$ in the total mass.
Finally, the (unique) solution to the system \eqref{eq:FTL-Hughes} is well defined and the density $R^n_{I_0}(t)$ is equal to zero until the turning point does not collide with a particle.

The approximated turning point $\xi^n(t)$ is implicitly uniquely defined by
\[
\int_{-1}^{\xi^n(t)} c(\rho^n(t,y)) ~{\d}y = \int^1_{\xi^n(t)} c(\rho^n(t,y)) ~{\d}y,
\]
where $\rho^n \colon \R_+\times\R \to [0,\rho_{\max}]$ is the discretized density defined by
\begin{equation}\label{eq:density}
  \rho^n(t,x) \doteq \sum_{i=0}^{n-1} R^n_{i}(t) \, \mathbf{1}_{[x^n_i(t),x^n_{i+1}(t))}(x).
\end{equation}
Clearly $\xi^n(t) \in \Omega$ for all $t\ge0$ and $\xi^n(0)$ does not necessarily coincide with $\bar{\xi}^n$.

In the next theorem we state our main result, which deals with a class of \emph{small} initial data in $\BV$. For further use, we define the function $\Upsilon(\rho) \doteq c(\rho) - c'(\rho) \, \rho$, which is strictly decreasing in view of assumption \eqref{C} above. We then set
\begin{align*}
 & \mathcal{L} \doteq \mathrm{Lip}[\Upsilon|_{[0,R]}]=\max \left\{  c''\left(\rho\right) \rho \colon \rho \in \left[0,R\right] \right\},
 \\
 & C \doteq c'(R)\, R = \max\left\{c'(\rho) \, \rho \colon \rho \in \left[0,R\right]\right\}.
\end{align*}

\begin{theorem}\label{teo:2}
Assume \eqref{I1}, \eqref{I2}, \eqref{V1}, \eqref{VH} and \eqref{C}.
If the initial datum $\bar\rho$ satisfies
\begin{align}\label{eq:inequality}
&R \doteq \|\bar{\rho}\|_{\L\infty(\Omega)} < \rho_{\max},
&\frac{v_{\max}}{2}
\left[\vphantom{\sum} \mathcal{L}\,{\tv}(\bar{\rho})
+
3 \, C\right]
<
v(R).
\end{align}
then there exists a unique entropy solution $\rho$ to \eqref{eq:model} in the sense of Definition~\ref{def:entropy_solution} defined globally in time.
Such a solution is obtained as a strong $\L1$-limit of the discrete density $\rho^n$ constructed via the FTL particle system \eqref{eq:FTL-Hughes}.
\end{theorem}
We omit the proof of Theorem \ref{teo:2} and we defer to \cite[Section 2.3]{DF_fagioli_rosini_russo} for the details.
Let us only remark that the assumption $R<\rho_{\max}$ above is essential in order to have the right-hand-side in the inequality \eqref{eq:inequality} strictly positive.

\section{The ARZ model}\label{sec:aw}

Consider the Cauchy problem for the ARZ model \cite{ARZ1, ARZ2}
\begin{equation}\label{eq:CauchyARZ}
\begin{cases}
    \rho_t + \left(\rho \, v\right)_x=0,&t>0,~ x\in\R,	
    \\
    (\rho\,w)_t + \left(\rho\,v\,w\right)_x=0,&t>0,~ x\in\R,
    \\
    (v,w)(0,x)= (\bar{v},\bar{w})(x),&x\in \R,
\end{cases}
\end{equation}
where $v$ is the velocity, $w$ is the Lagrangian marker and $(\bar{v},\bar{w})$ is the corresponding initial datum.
Moreover, $(v,w)$ belongs to $\mathcal{W} \doteq \left\{(v,w) \in \bar{\R}_+^2 \colon v \le w\right\}$ and $\rho\doteq p^{-1}(w-v) \ge 0$ is the corresponding density, where $p \in \C0(\bar{\R}_+;\bar{\R}_+)\cap\C2(\R_+;\bar{\R}_+)$ satisfies
\begin{align}\label{P}\tag{P}
    &p(0^+) = 0,&
    &p'(\rho)>0&
    &\text{and}&
    &2\,p'(\rho)+\rho \, p''(\rho)>0&
    \text{for every }\rho>0.
\end{align}
The typical choice is $p(\rho) \doteq \rho^\gamma$, $\gamma>0$.
By definition, we have that the vacuum state $\rho=0$ corresponds to the half line $\mathcal{W}_0 \doteq \left\{(v,w)^T \in \mathcal{W} \colon  v=w \right\}$ and the non-vacuum states $\rho>0$ to $\mathcal{W}_0^c \doteq \mathcal{W}\setminus\mathcal{W}_0$.

\begin{definition}[{\cite[Definition~2.3.]{BCJMU-order2} and \cite[Definition~2.2]{donadello_rosini}}]\label{def:weak}
Let $(\bar{v},\bar{w})\in \L\infty(\R;\,\mathcal{W})$. We say that a function $(v,w)\in \L\infty(\bar{\R}_+\times \R;\,\mathcal{W})\cap \C0(\bar{\R}_+;\,\Lloc1(\R;\,\mathcal{W}))$ is a weak solution of \eqref{eq:CauchyARZ} if it satisfies the initial condition $(v(0,x),w(0,x))=(\bar{v}(x),\bar{w}(x))$ for a.e.\ $x\in \R$ and for any test function $\phi\in \Cc\infty(\R\times \R)$
\[
\iint_{\bar{\R}_+\times\R} p^{-1}(v,w) \, (\phi_t + v \, \phi_x) \begin{pmatrix} 1 \\ w\end{pmatrix} {\d}x ~\d{t} = \begin{pmatrix} 0 \\ 0\end{pmatrix}.
\]
\end{definition}
We refer to \cite{ferreira} for the existence of solutions to \eqref{eq:CauchyARZ} away from vacuum, and to \cite{godvik} for the existence with vacuum.
Let us briefly recall the main properties of the solutions to \eqref{eq:CauchyARZ}. If the initial density $\bar{\rho} \doteq p^{-1}(\bar{w}-\bar{v})$ has compact support, then the support of $\rho$ has finite speed of propagation. The maximum principle holds true in the Riemann invariant coordinates $(v,w)$, but not in the conserved variables $(\rho,\rho\,w)$ as a consequence of hysteresis processes. Moreover, the total space occupied by the vehicles is time independent: $\int_\R  \rho(t,x) ~{\d}x = \|\bar \rho\|_{\L1(\R)}$ for all $t\ge0$.

\subsection{The follow-the-leader scheme and main result}

We introduce our atomization scheme for the Cauchy problem \eqref{eq:CauchyARZ}.
Let $(\bar{v},\bar{w}) \in \BV(\R;\mathcal{W})$ be such that $\bar\rho \doteq p^{-1}(\bar{w}-\bar{v})$ belongs to $\L1(\R)$ and $\bar\rho$ is compactly supported. Denote by $\bar{x}_{\min} < \bar{x}_{\max}$ the extremal points of the convex hull of the compact support of $\bar\rho$, namely \[\bigcap_{\left[a,b\right]\supseteq\supp\left(\bar\rho\right)} \left[a,b\right] = \left[\bar{x}_{\min}, \bar{x}_{\max}\right].\]
Fix $n \in \N$ sufficiently large. Let $L \doteq \norma{\bar\rho}_{\L1(\R)} > 0$ and $\ell_n \doteq L/n$. Set recursively
\begin{equation}\label{eq:initial_ftl}
\begin{cases}
\bar{x}_0^n \doteq \bar{x}_{\min},
\\
\bar{x}_i^n \doteq \sup\left\{x\in \R  \colon \int_{\bar{x}^n_{i-1}}^x\bar\rho(x) ~{\d}x<\ell_n\right\},&
i \in \{1,\ldots,n\}.
\end{cases}
\end{equation}
It is easily seen that $\bar{x}^n_{n}=\bar{x}_{\max}$ for all $i=0,\ldots,n$.
We approximate then $\bar{w}$ by taking
\begin{align}\label{eq:dw}
&\bar{w}_i^n \doteq \underset{[\bar{x}_i^n,\bar{x}_{i+1}^n]}{\rm ess\,sup} (\bar{w}),
&i \in \{0,\ldots,n-1\}.
\end{align}
We have then
\begin{align*}
  &\ell_n=\int_{\bar{x}_i^n}^{\bar{x}_{i+1}^n} \bar\rho(x) ~{\d}x \le \left(\bar{x}_{i+1}^n - \bar{x}_{i}^n\right) \rho_{i,\max}^n ,&
  i \in \{0,\ldots,n-1\},
\end{align*}
with $\rho_{i,\max}^n \doteq p^{-1}(\bar{w}_i^n)$.
We take the values $\bar{x}_0^n,\ldots,\bar{x}_{n}^n$ as the initial positions of the $(n+1)$ particles in the $n$--depending FTL model
\begin{equation}\label{eq:ftl}
\begin{cases}
    x_{n}^n(t) = \bar{x}_{\max} + \bar{w}_{n-1}^n \,t,\\
    \dot x_i^n(t) = v_i^n \left(\frac{\ell_n}{x_{i+1}^n(t) - x_i^n(t)}\right) , & i \in \{0, \ldots, n-1\} ,\\
    x_i^n(0) = \bar{x}_i^n , & i \in \{0, \ldots, n\} ,
\end{cases}
\end{equation}
where
\begin{align}\label{eq:vni}
&v_i^n(\rho) \doteq \bar{w}_i^n - p(\rho),&
    i \in \{0,\ldots,n-1\}.
\end{align}
The quantity $\bar{w}^n_i = v_i^n(0)$ is the maximum possible velocity allowed for the $i$-th vehicle.
Clearly, only the leading vehicle $x^n_n$ reaches its maximal velocity, as the vacuum state is achieved only ahead of $x^n_n$.
The existence of a global solution to \eqref{eq:ftl} follows from \cite[Lemma~2.3]{DF_fagioli_rosini}, which generalises the discrete maximum principle of Lemma~\ref{lem:maximum}.
Finally, since $v_i^n$ is decreasing, and its argument $\ell_n/[x_{i+1}^n(t)-x_i^n(t)]$ is always bounded above by $\rho_{i,\max}^n$, we have $x_0^n(t)\geq \bar{x}_{\min} + v_0(R_0^n) \, t = \bar{x}_{\min}$.
By introducing in \eqref{eq:ftl}
\begin{align}\label{eq:Deltan}
    &R^n_i(t) \doteq \frac{\ell_n}{x_{i+1}^n(t) - x_i^n(t)},&
    i \in \{0,\ldots,n-1\},
\end{align}
we obtain
\begin{equation}\label{eq:dyi}
\begin{cases}
    \dot{R}^n_{n-1} = -\frac{(R^n_{n-1})^2}{\ell_n} \, p(R^n_{n-1}),
    \\
    \dot{R}^n_i = -\frac{(R^n_i)^2}{\ell_n} \left[v_{i+1}^n(R^n_{i+1})-v_i^n(R^n_i)\right], & i \in \{0, \ldots, n-2\} ,
    \\
    R^n_i(0) = \bar{R}^n_i \doteq \frac{\ell_n}{\bar{x}_{i+1}^n - \bar{x}_i^n} ,&
    i \in \{0,\ldots,n-1\}.
\end{cases}
\end{equation}
Observe that $\ell_n/[\bar{x}_{\max} - \bar{x}_{\min} + \bar{w}_{n-1}^n \, t] \le R^n_i(t) \le \rho_{i,\max}^n$ for all $t\ge0$ in view of the discrete maximum principle. The quantity $R^n_i$ can be seen as a discrete version of the density $\rho$ in Lagrangian coordinates, and \eqref{eq:dyi} is the discrete Lagrangian version of the Cauchy problem \eqref{eq:CauchyARZ}.

Define the piecewise constant (with respect to $x$) Lagrangian marker
\begin{align}\label{eq:hwn}
&W^n(t,x) \doteq
\begin{cases}
\bar{w}^n_0 &\text{if } x \in \left(-\infty, x_0^n(t)\right),
\\
\bar{w}^n_i&\text{if } x \in \left[x_i^n(t), x_{i+1}^n(t)\right),~ i \in \{0, \ldots, n-1\},
\\
\bar{w}^n_{n-1}&\text{if } x \in \left[x^n_{n}(t),\infty\right),
\end{cases}
\intertext{and the piecewise constant (with respect to $x$) velocity}
\label{eq:hvn}
&V^n(t,x) \doteq
\begin{cases}
\bar{w}^n_0&\text{if } x \in \left(-\infty, x_0^n(t)\right),
\\
v^n_i(R^n_i(t))&\text{if } x \in \left[x_i^n(t), x_{i+1}^n(t)\right),~ i \in \{0, \ldots, n-1\},
\\
\bar{w}^n_{n-1}&\text{if } x \in \left[x^n_{n}(t),\infty\right).
\end{cases}
\end{align}
We are now ready to state the main result proved in \cite{DF_fagioli_rosini}.

\begin{theorem}
\label{thm:mainARZ}
Assume \eqref{P}.
Let $(\bar{v},\bar{w})\in \BV(\R\,;\,\mathcal{W})$ be such that $\bar\rho \doteq p^{-1}(\bar{w}-\bar{v})$ is compactly supported and belongs to $\L1(\R)$.
Fix $n\in \N$ sufficiently large and let $\ell_n \doteq L/n$, with $L \doteq \norma{\bar{\rho}}_{\L1(\R)}$. Let $\bar{x}_0^n<\ldots<\bar{x}_{n}^n$ be the atomization constructed in \eqref{eq:initial_ftl}. Let $x_0^n(t),\ldots,x_{n}^n(t)$ be the solution to the FTL system \eqref{eq:ftl}.
Let $\bar{w}_0^n, \ldots, \bar{w}_{n-1}^n$ be given by \eqref{eq:dw}. Set $W^n$ and $V^n$ as in \eqref{eq:hwn} and \eqref{eq:hvn} respectively, where $v^n_i$ and $R^n_i$ are defined by \eqref{eq:vni} and \eqref{eq:Deltan} respectively. Then, $(V^n,W^n)_n$ converges (up to a subsequence) in $\Lloc1(\bar{\R}_+\times \R;\mathcal{W})$ as $n\rightarrow \infty$ to a weak solution of the Cauchy problem \eqref{eq:CauchyARZ} with initial datum $(\bar{v},\bar{w})$ in the sense of Definition~\ref{def:weak}.
\end{theorem}
We omit the proof of Theorem \ref{thm:mainARZ} and we defer to \cite[Theorem 3.2]{DF_fagioli_rosini} for the details.
For completeness, we point out that the corresponding discrete density is
\[\rho^n(t,x) \doteq p^{-1}(W^n(t,x)-V^n(t,x)) = \sum_{i=1}^{n-1} R_i^n(t) \, \chi_{[x_i^n(t),x_{i+1}^n(t)[}(x).\]

\section{Numerical simulations}\label{sec:numerics}
In this section we present numerical simulations for the particle method described above.
We compare the numerical simulations with the exact solutions obtained by the method of characteristics and that with the Godunov method.
The particle system is solved by using the Runge-Kutta MATLAB solver ODE23, with the  initial mesh size determined by the total number of particles $N$ and the initial density values.

\subsection{The Cauchy problem for the LWR equation}
We first furnish one example for the Cauchy problem for the LWR equation \eqref{eq:cauchy} with flux given by $f(\rho) \doteq \rho \; (1-\rho)$.
In \figurename~\ref{fig:Test1} we take $N=200$, final time $T=0.5$ and initial datum
\begin{equation}\label{IC}
\bar{\rho}(x)=\begin{cases}
 0.4 &\text{if } -1\leq x\leq 0, \\
 0.8 &\text{if } 0< x\leq 1, \\
 0 & \text{otherwise}.
\end{cases}
\end{equation}
In \figurename~\ref{fig:confr} we compare the result of the simulation with $N=400$ particles and final time $t=0.5$ with exact solutions.
\begin{figure}
\begin{center}
\includegraphics[width=.32\textwidth]{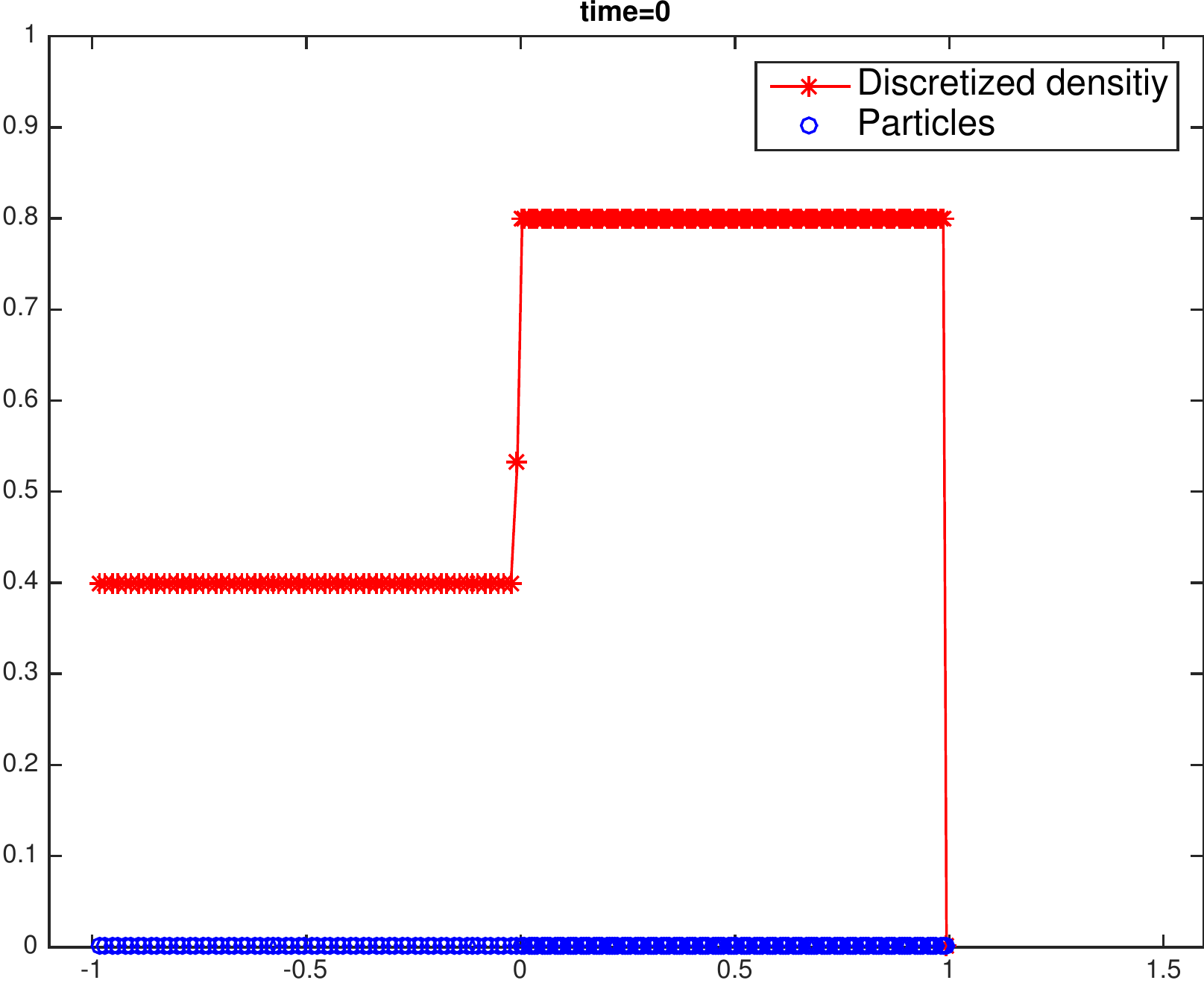}~
\includegraphics[width=.32\textwidth]{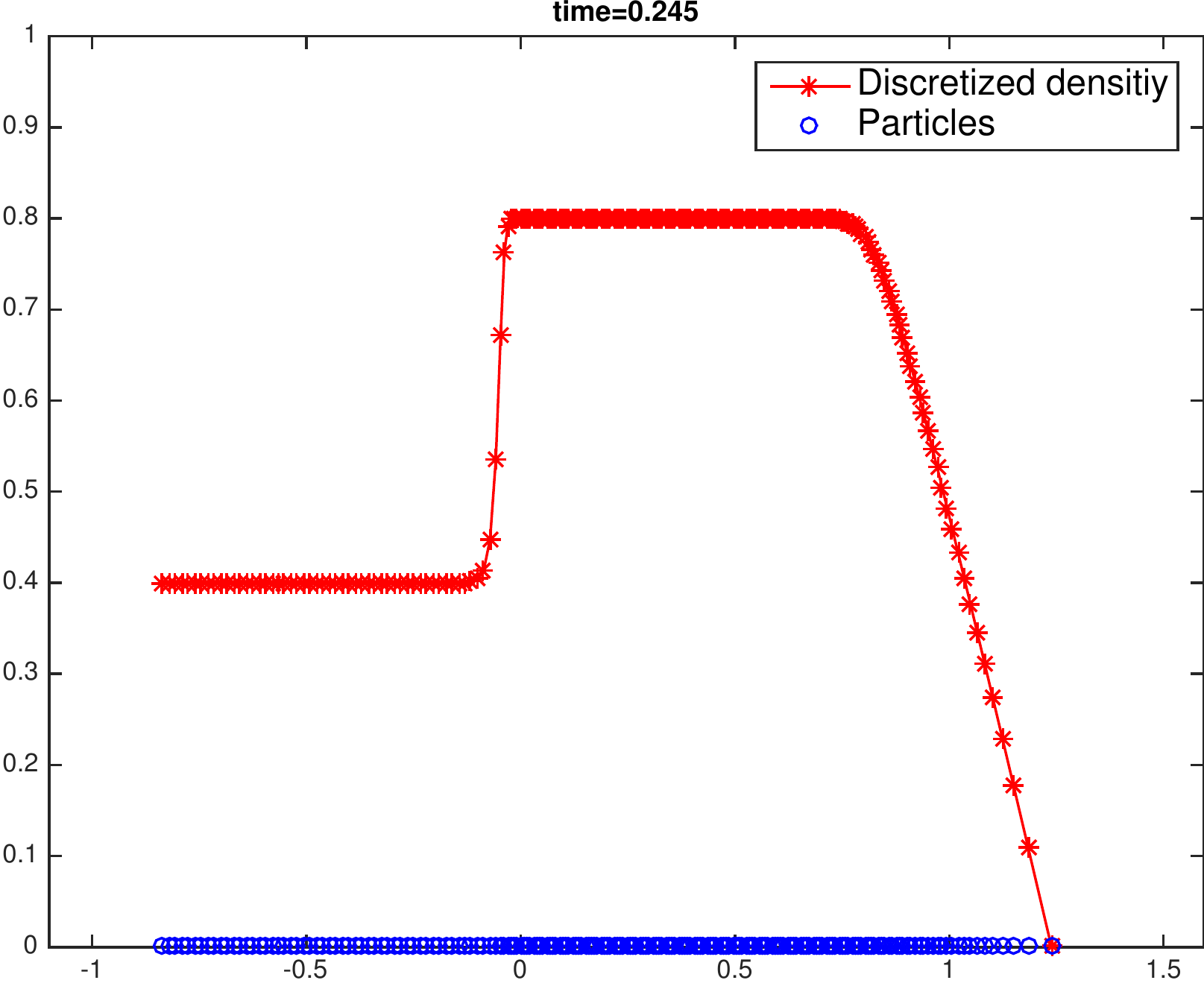}~
\includegraphics[width=.32\textwidth]{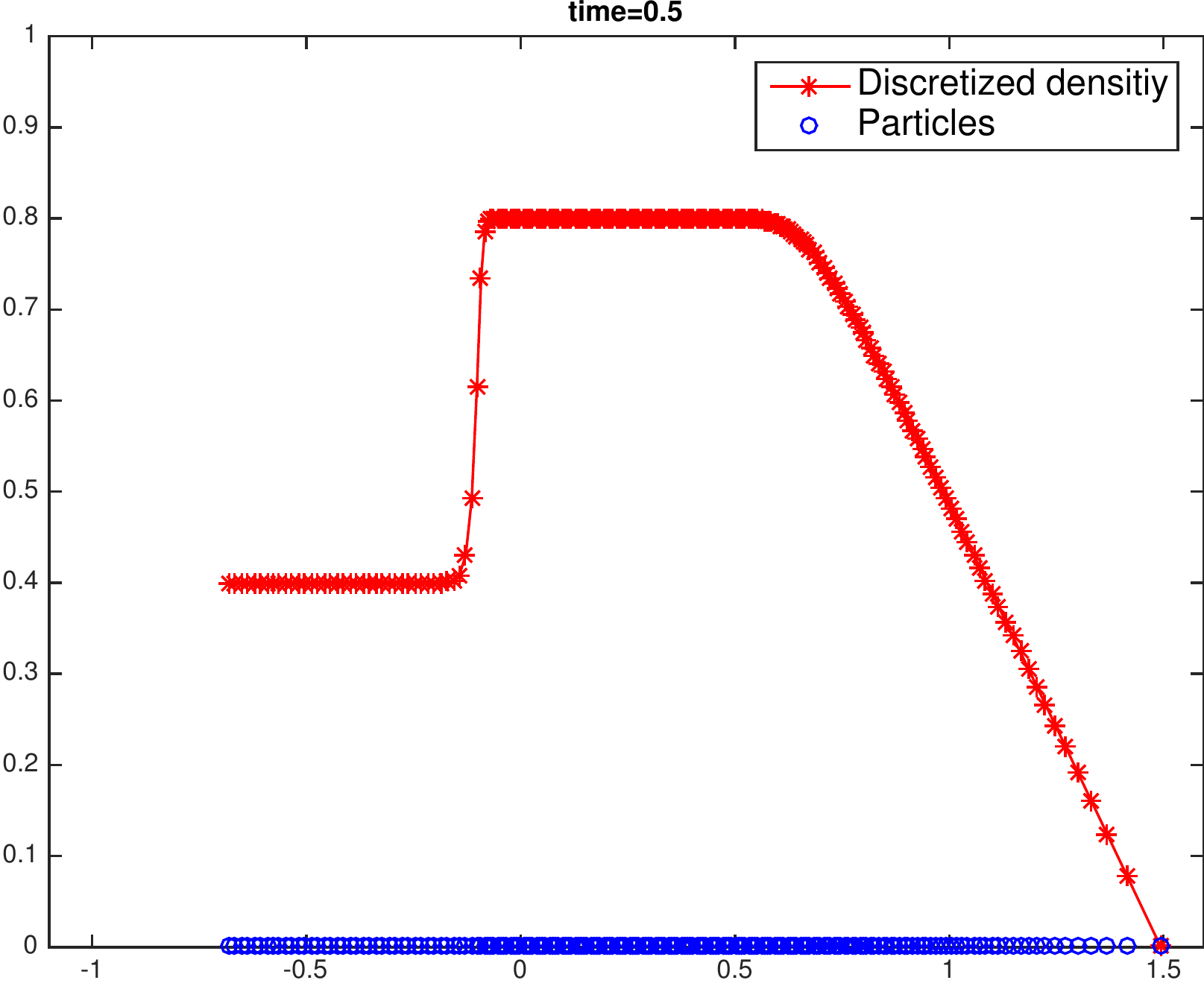}
\end{center}
\caption{The evolution of $\rho^n$ given by \eqref{eq:disdens} and corresponding to the initial datum \eqref{IC}.
The circles in the bottom (in blue) denote particle location, while the stars in the top (in red) denote the computed density.}
\label{fig:Test1}
\end{figure}

\begin{figure}
\begin{center}
\includegraphics[width=.32\textwidth]{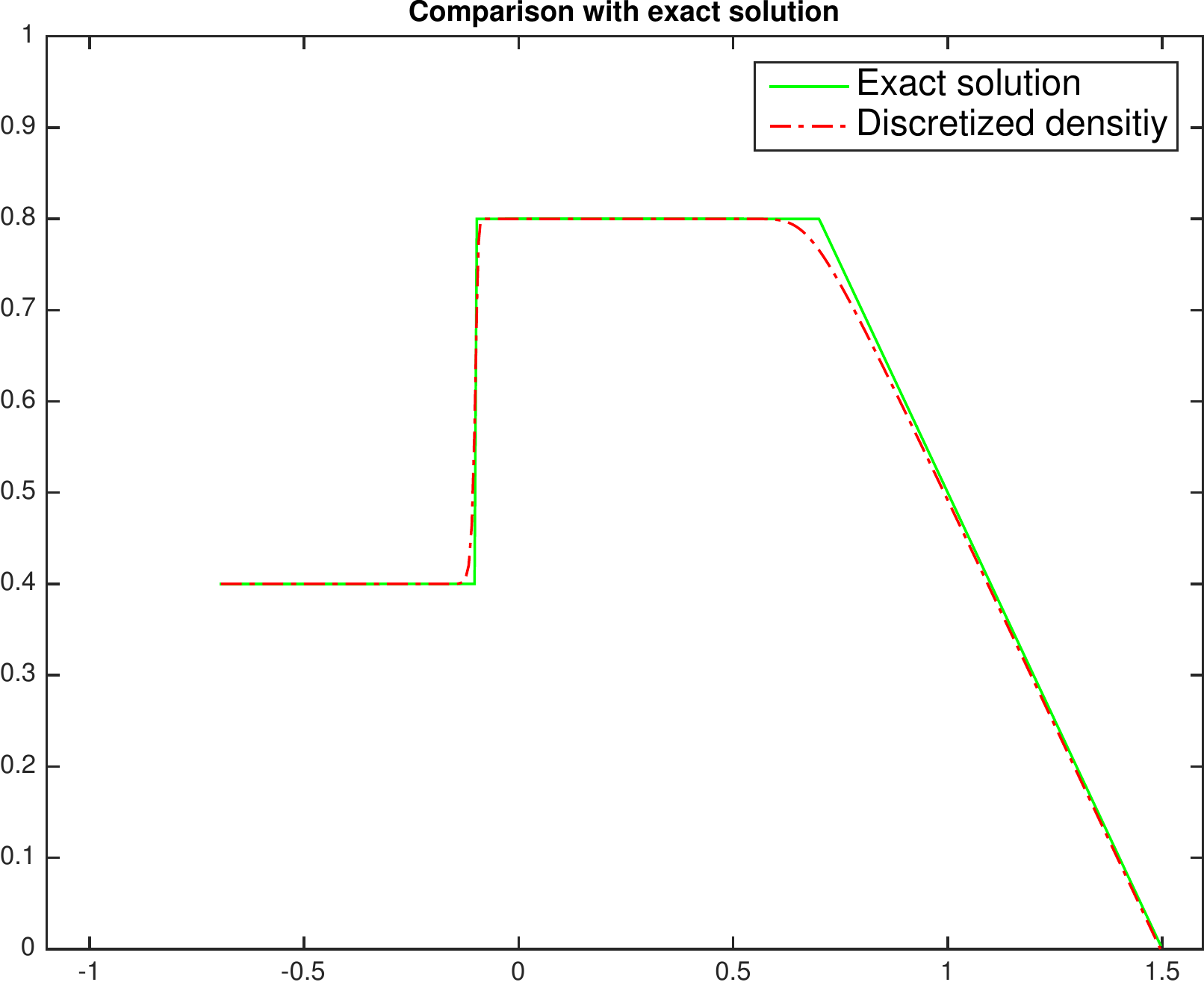}
\end{center}
\caption{Comparison between the exact solution (continuous green line) and $\rho^n$ (``$-\cdot -$'' in red) for $N=400$ and initial datum \eqref{IC}. \label{fig:confr}}
\end{figure}

\subsection{The Cauchy-Dirichlet problem for the LWR equation}\label{sec:simibvp}

More interesting situations can be illustrated in the case of LWR with Dirichelet boundary conditions \eqref{eq:ibvp}, see  \figurename~\ref{fig:Test1bc} and  \figurename~\ref{fig:Test2bc}. As pointed out in Section~\ref{Sec:dirichlet}, the atomization algorithm introduces artificial queuing particles for miming the left boundary condition. In $x<0$ we arrange $N$ queuing particles (the ones that are going to enter in the domain at time $T$), with $N$ given by \eqref{eq:QN}.
Again we take $f(\rho) \doteq \rho(1-\rho)$.
For $N=100$ particles we consider, according to the notation in Section~\ref{Sec:dirichlet}, $\bar{\rho}(x)=0.2$ in \figurename~\ref{fig:Test1bc} and \figurename~\ref{fig:Test2bc} with left boundary condition $\bar{\rho}_0=0.4$ and right boundary conditions $\bar{\rho}_1=0$ and $\bar{\rho}_1=1$ respectively.
In \figurename~\ref{fig:Test3bc} we set
\begin{align}\label{PP}
&\bar{\rho}(x)=\begin{cases}
 0.8 &\text{if } x\in [0,0.5], \\
 0.1 &\text{if } x\in (0.5,1],
\end{cases}&
&\bar{\rho}_0=0.3, &\bar{\rho}_1=0.1.
\end{align}
The latter example is chosen in such way that the actual entropy solution does not match the solution obtained without reupdating the boundary condition at $x=0$ at every time step.
A comparison between discretized densities and numerical solutions obtained via Godunov scheme is also plotted in \figurename~\ref{fig:Test1bc} and \figurename~\ref{fig:Test2bc}. We set the spatial discretization according to the number of particles $N$; the time step is the same for both methods and is selected so that the CFL condition for the Godunov method holds. Empirically, the observed time step restriction for the FTL method is much less severe than for the Godunov method applied to the Eulerian descriprion of the flow.
\begin{figure}
\begin{center}
\includegraphics[width=.32\textwidth]{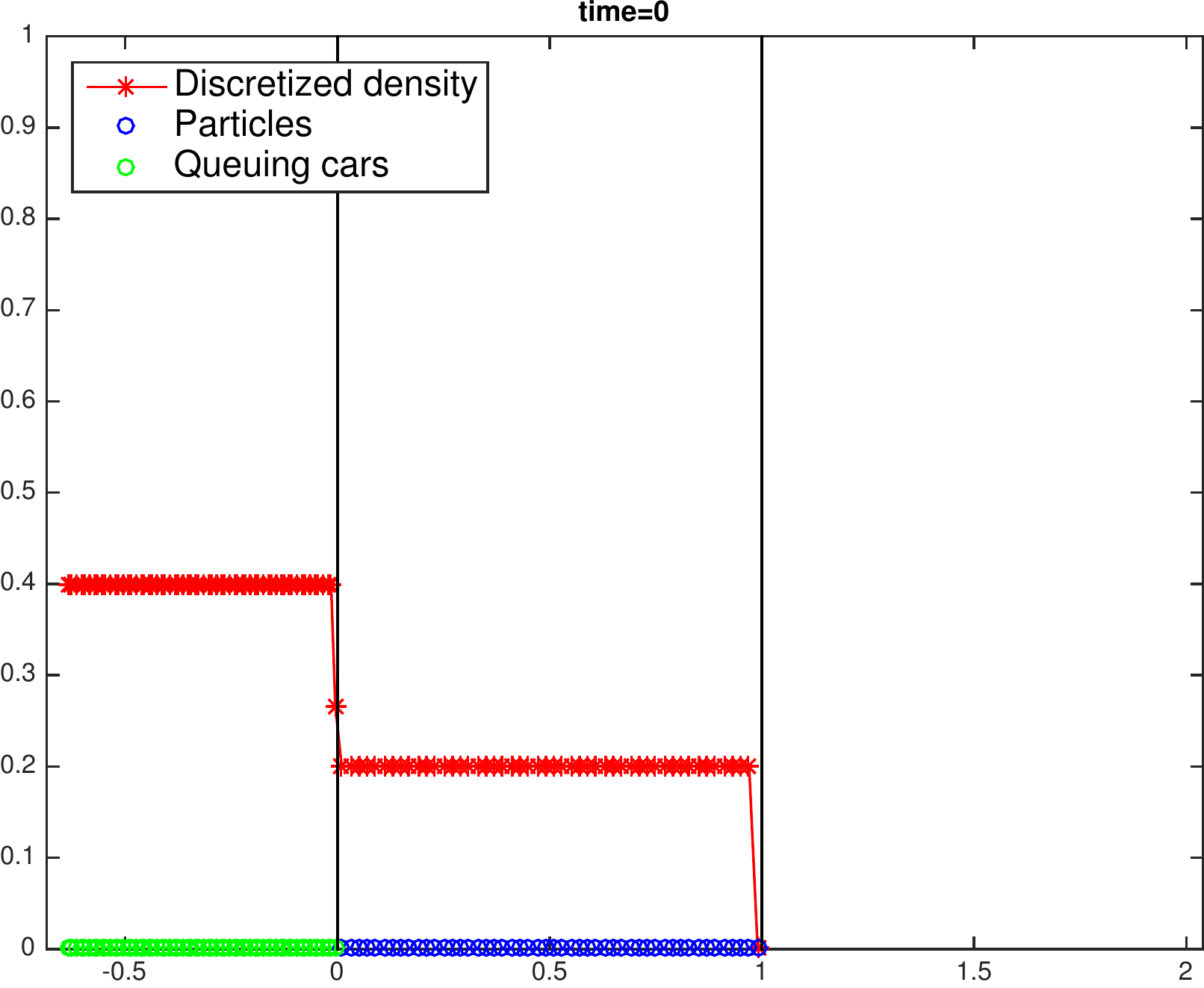}~
\includegraphics[width=.32\textwidth]{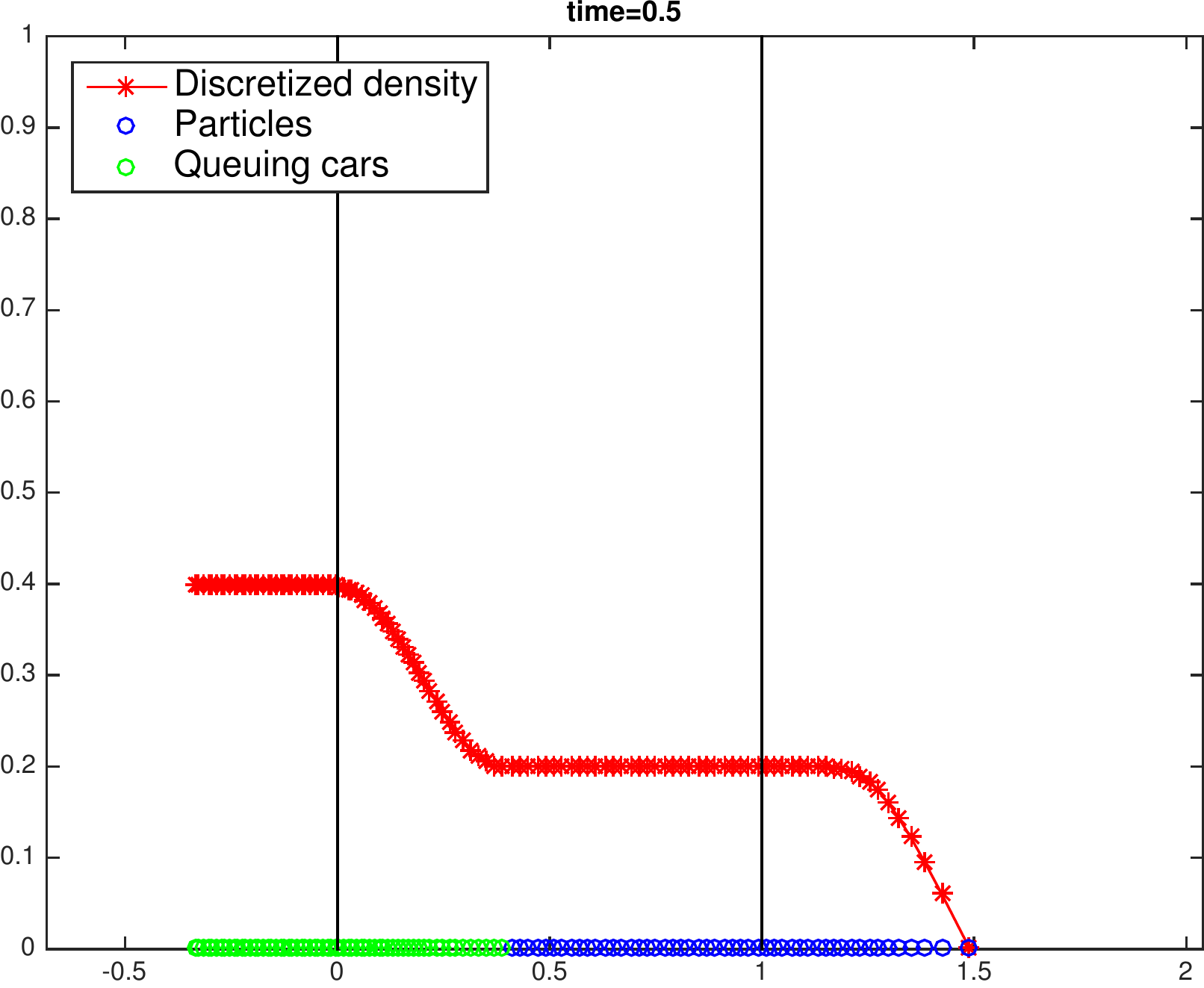}~
\includegraphics[width=.32\textwidth]{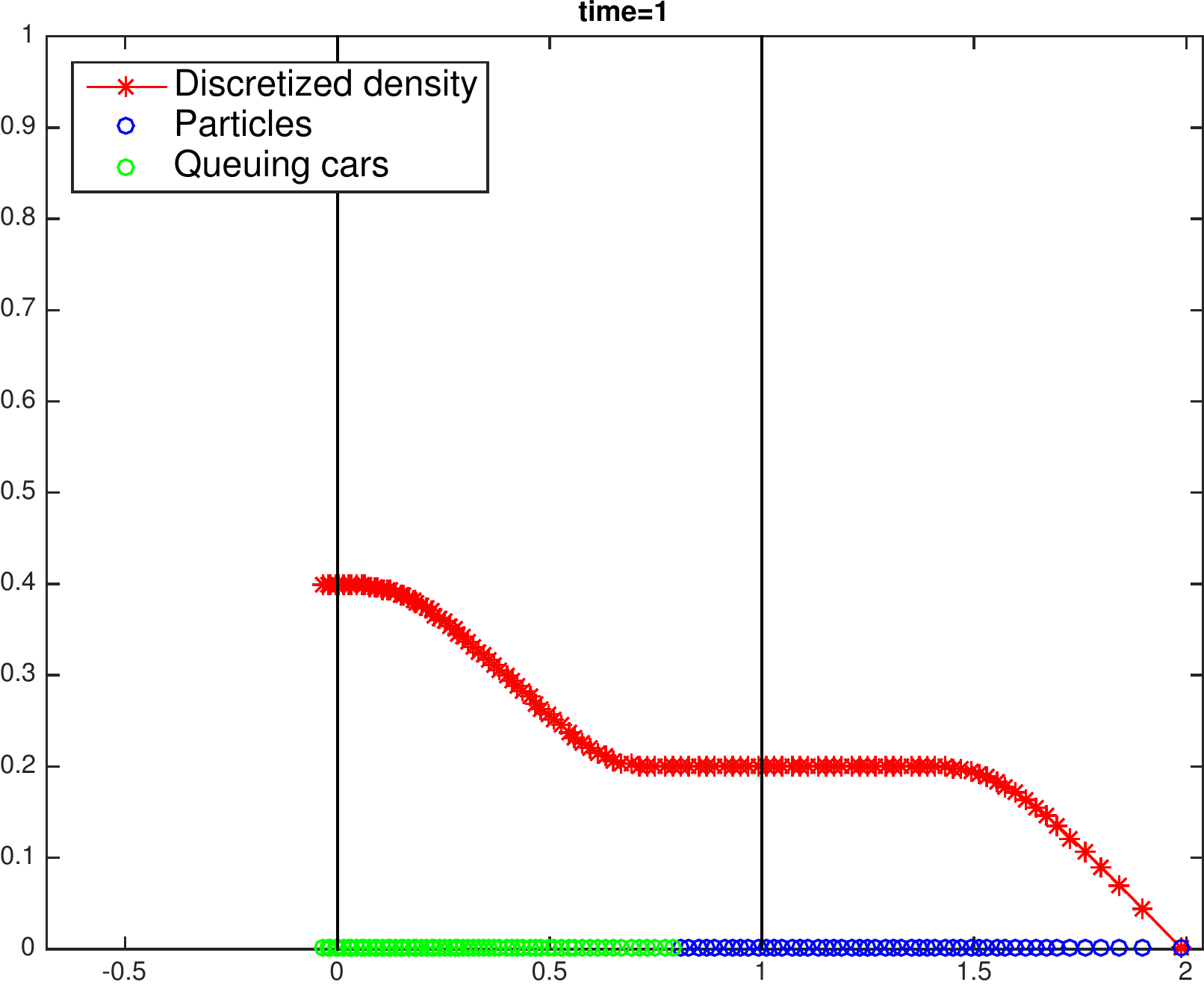}\\
\includegraphics[width=.4\textwidth]{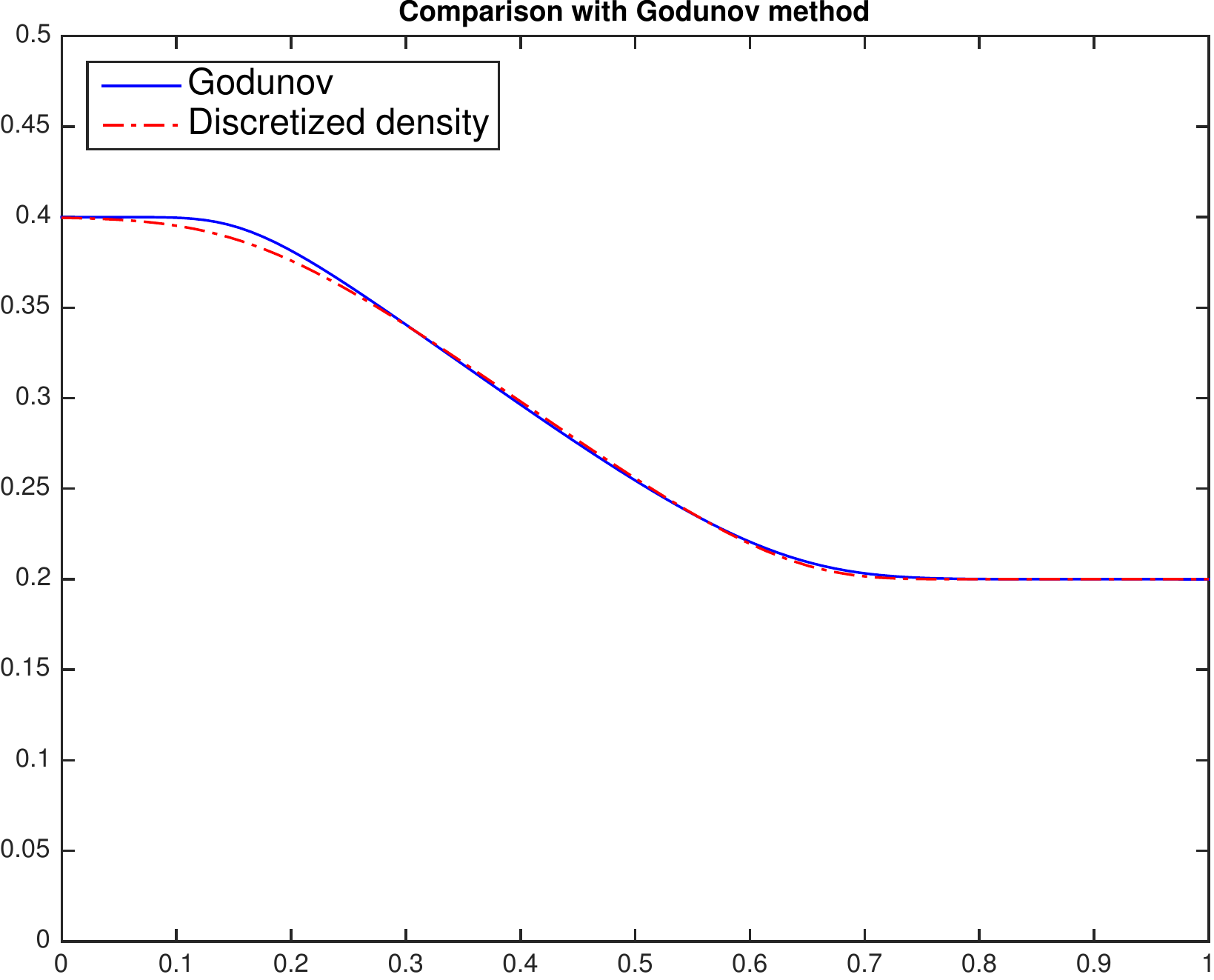}
\end{center}
\caption{The circles in the bottom are now divided in two groups: particles that are initially inside the domain (blue) and the queuing particles (green). The star-shaped line in the top (in red) denotes the computed density. Vertical black lines denote the boundary of $[0,1]$. The initial-boundary setting produces two rarefaction waves both travelling from the left to the right. In the bottom row we present a comparison with Godunov method.}
\label{fig:Test1bc}
\end{figure}

\begin{figure}
\begin{center}
\includegraphics[width=.32\textwidth]{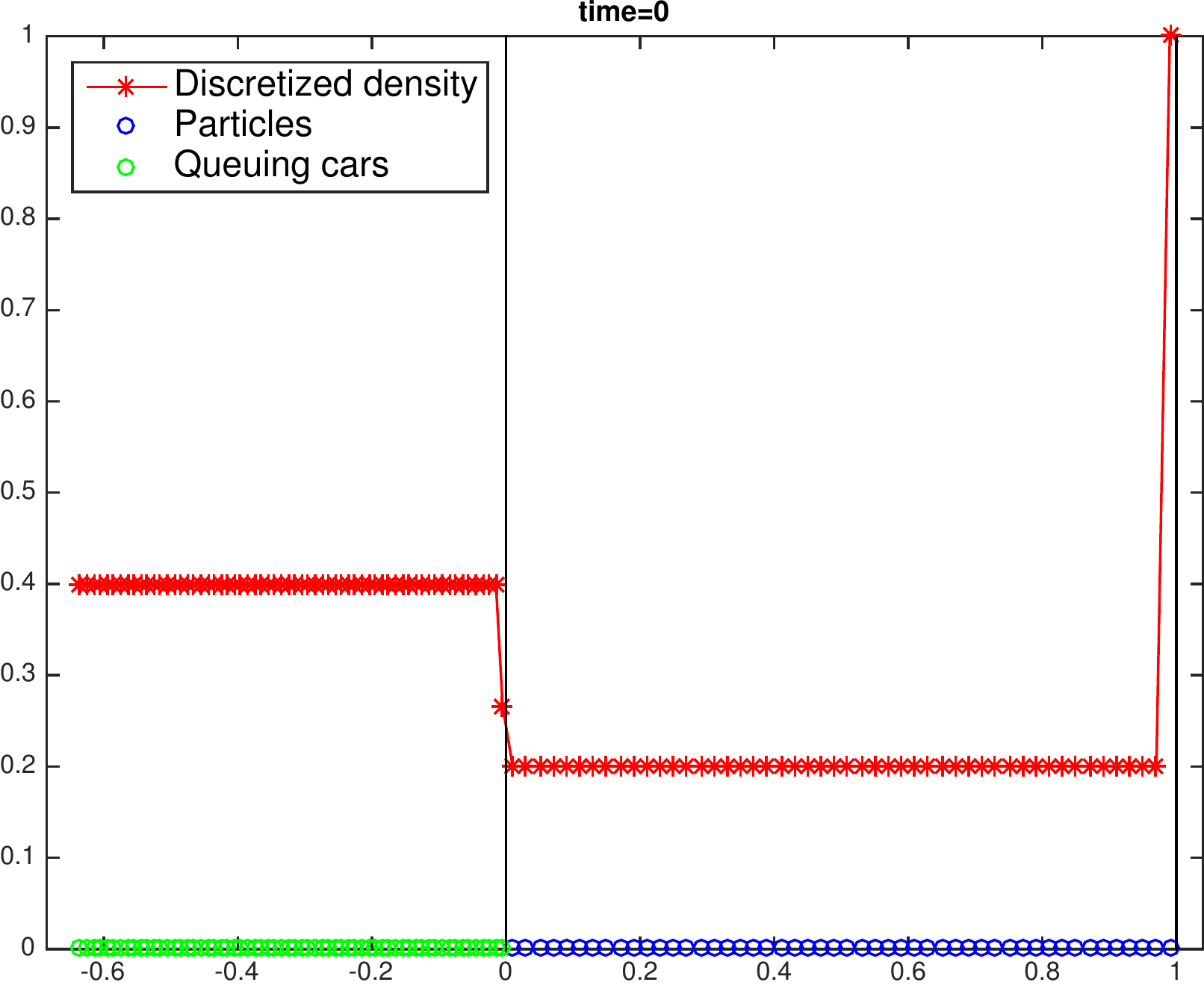}~
\includegraphics[width=.32\textwidth]{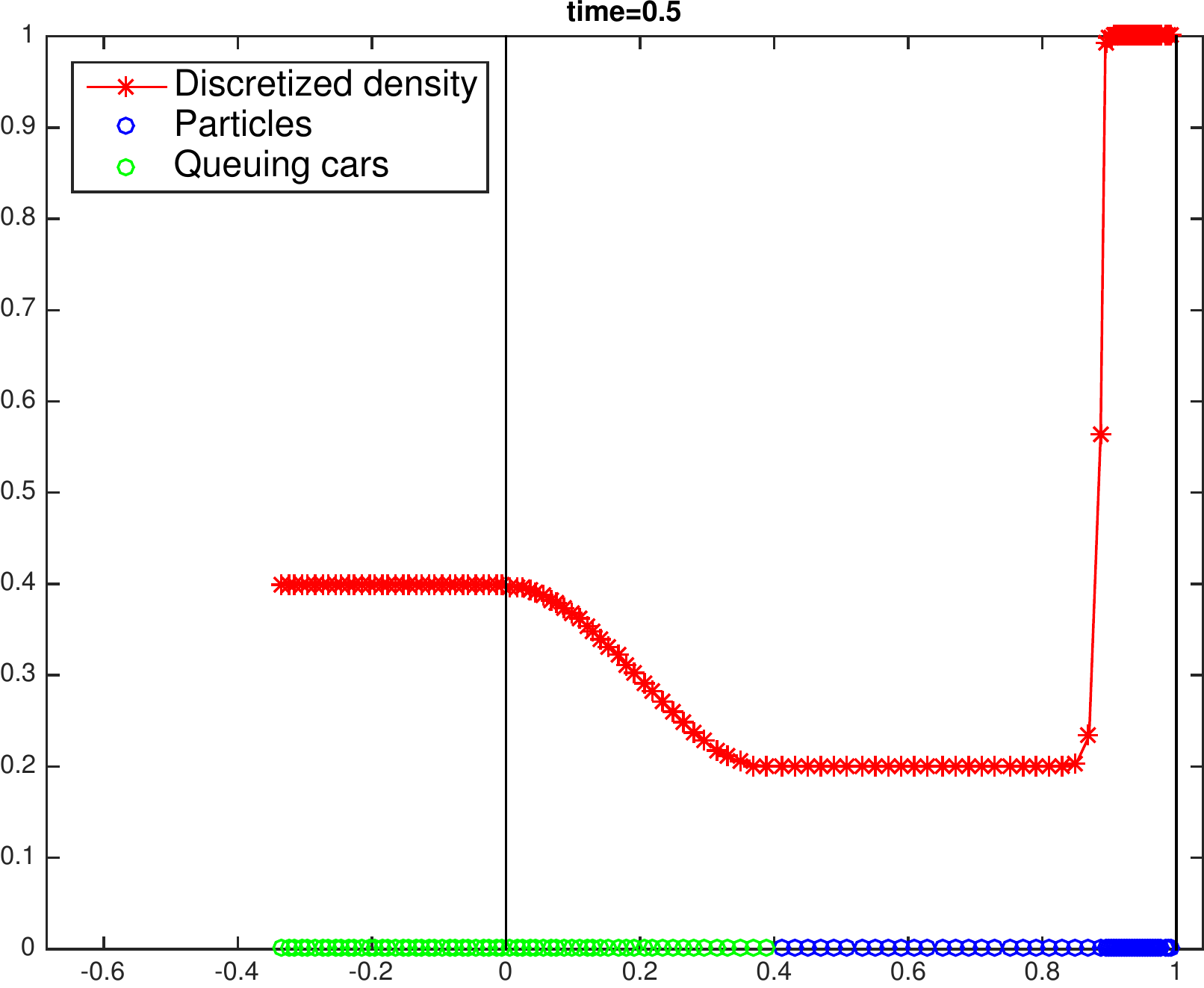}~
\includegraphics[width=.32\textwidth]{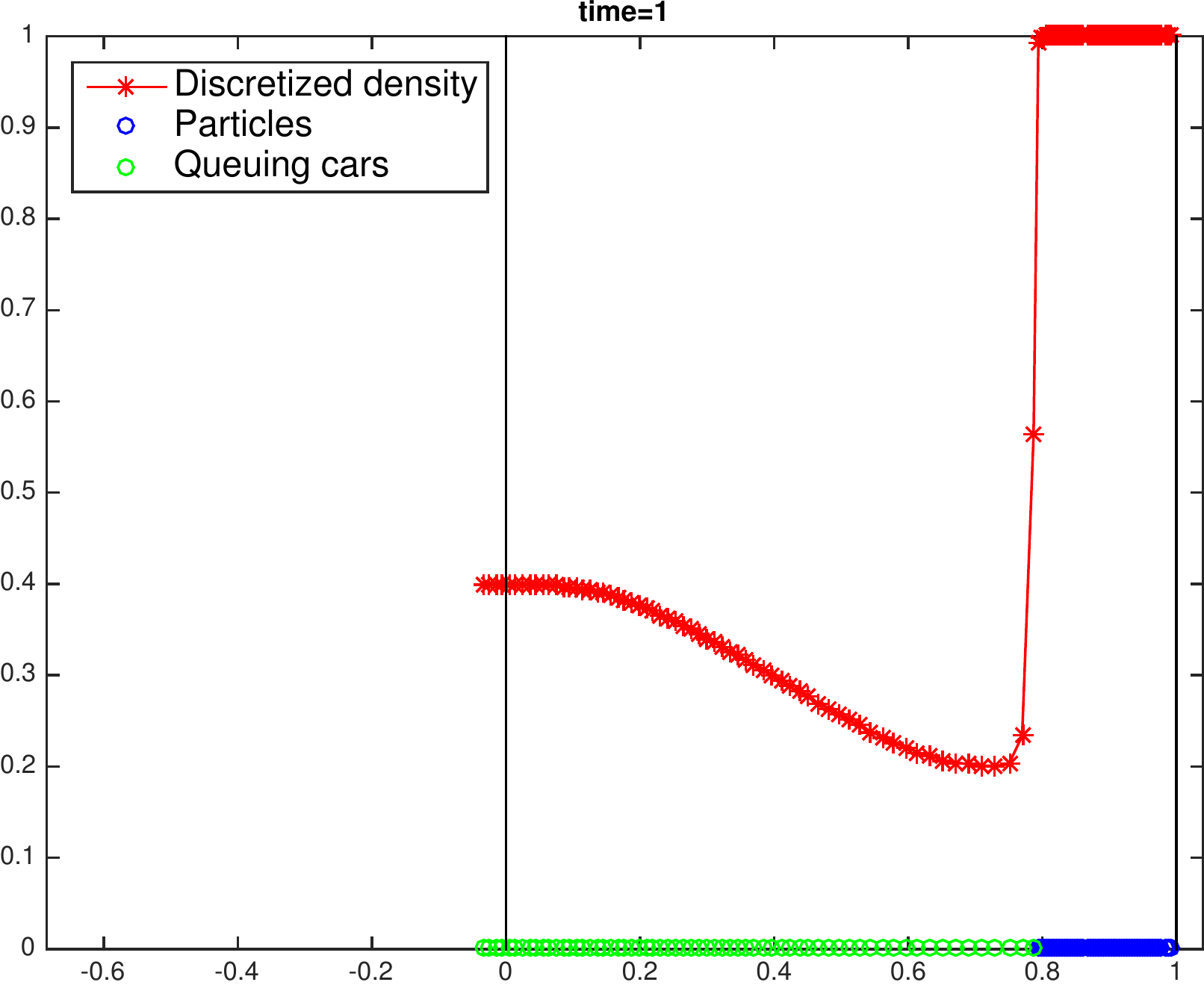}\\
\includegraphics[width=.4\textwidth]{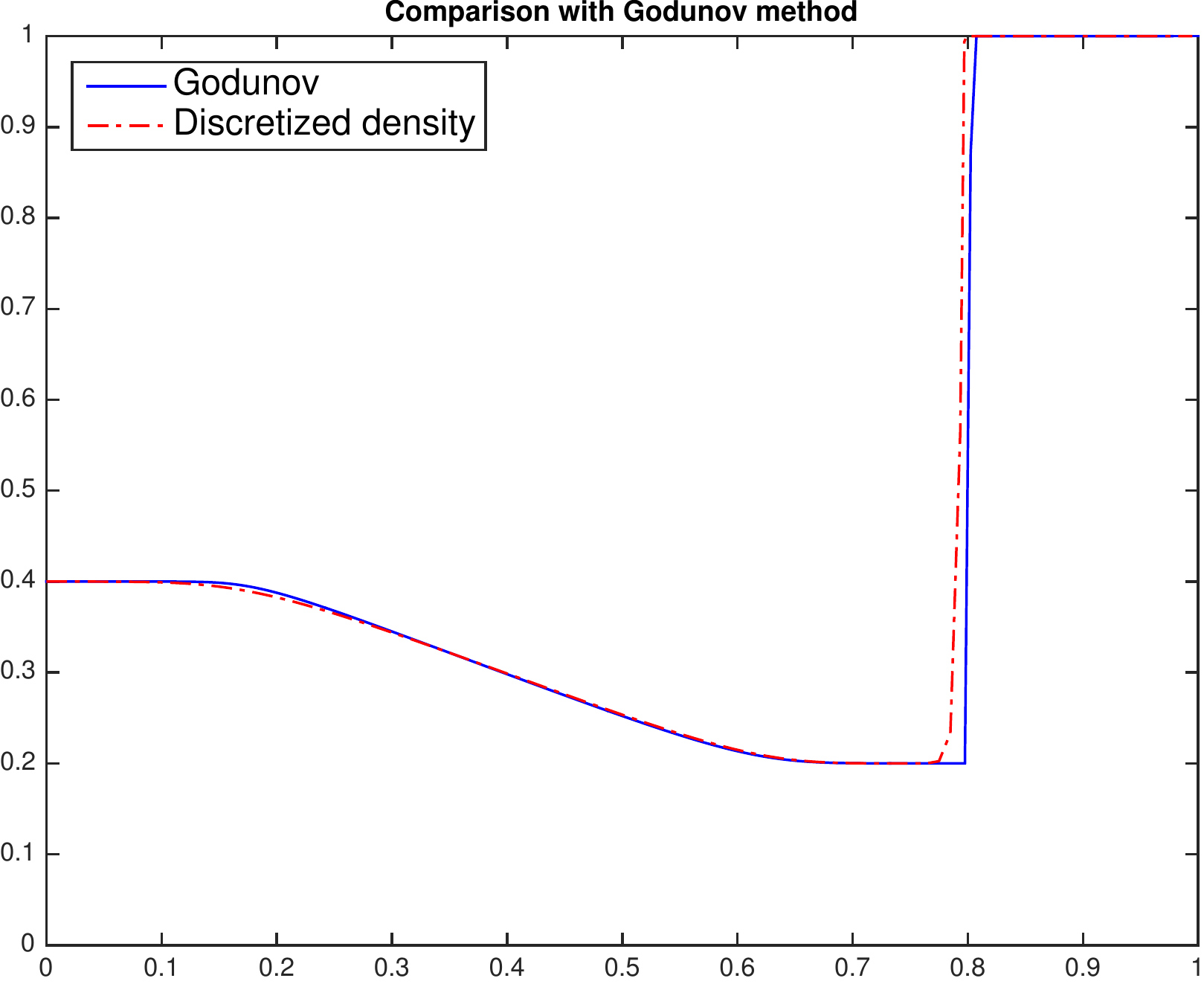}
\end{center}
\caption{In this situation a shock wave travels backward. The notation is similar to figure \ref{fig:Test1bc}}
\label{fig:Test2bc}
\end{figure}

\begin{figure}
\begin{center}
\includegraphics[width=.32\textwidth]{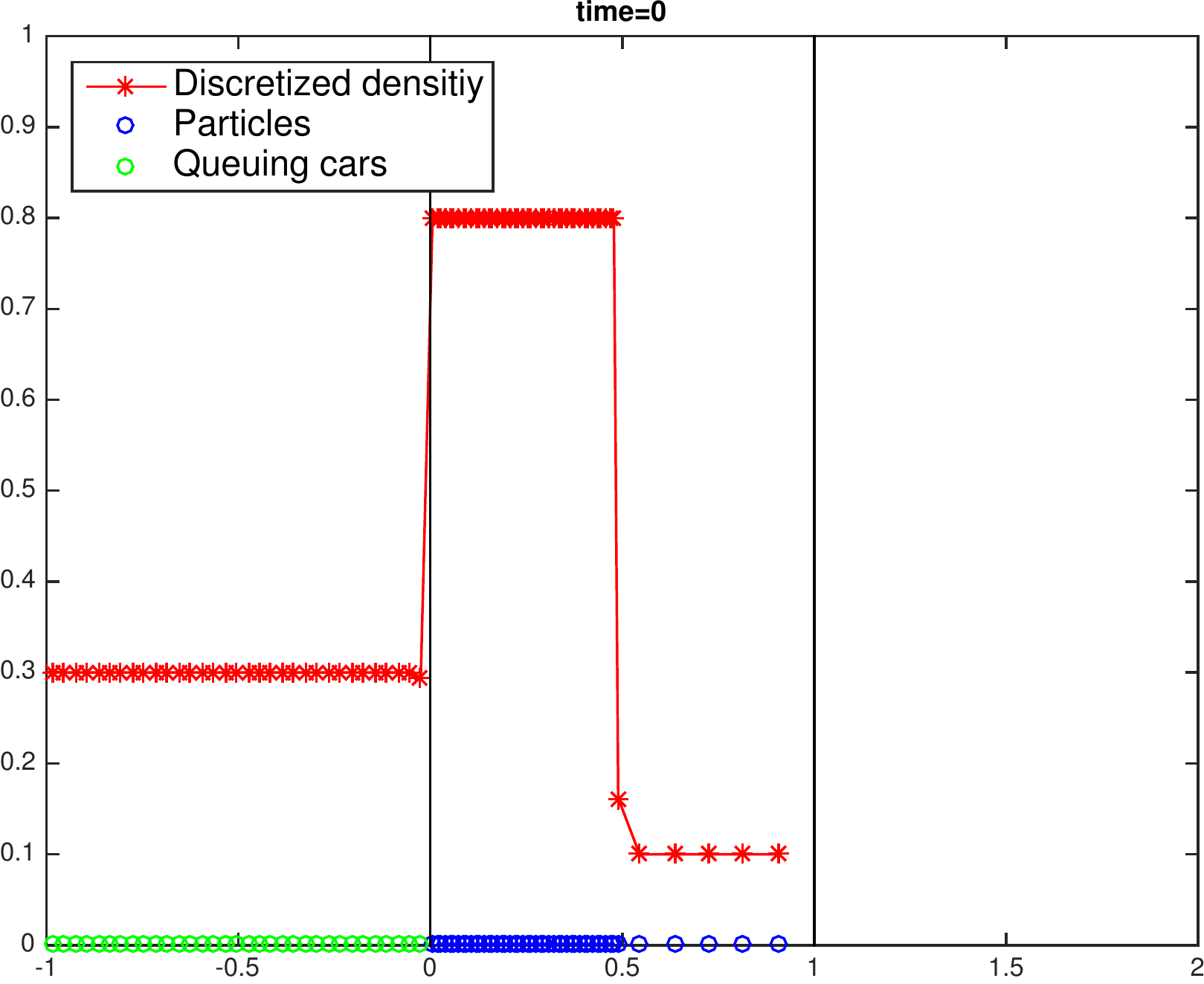}~
\includegraphics[width=.32\textwidth]{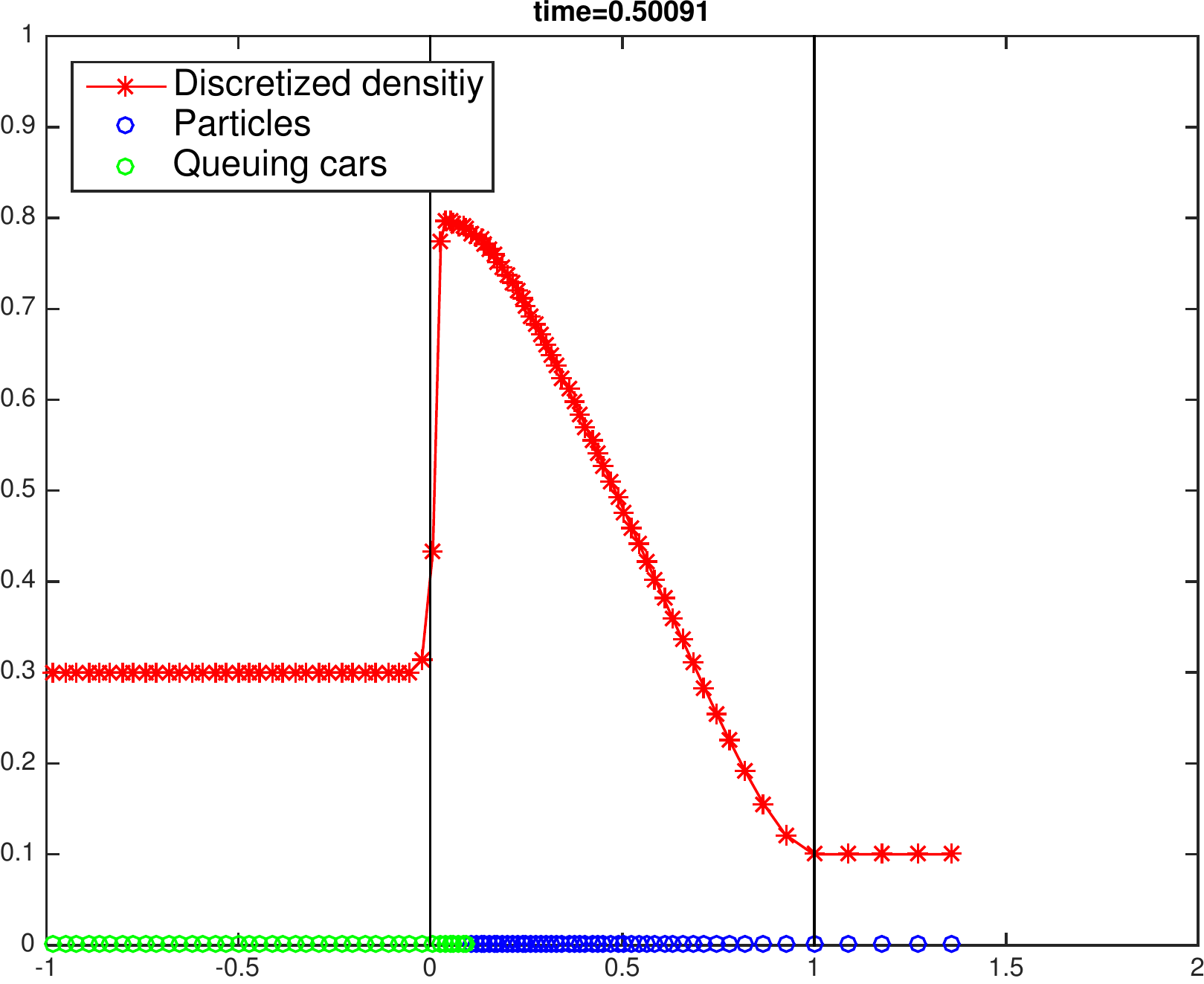}~
\includegraphics[width=.32\textwidth]{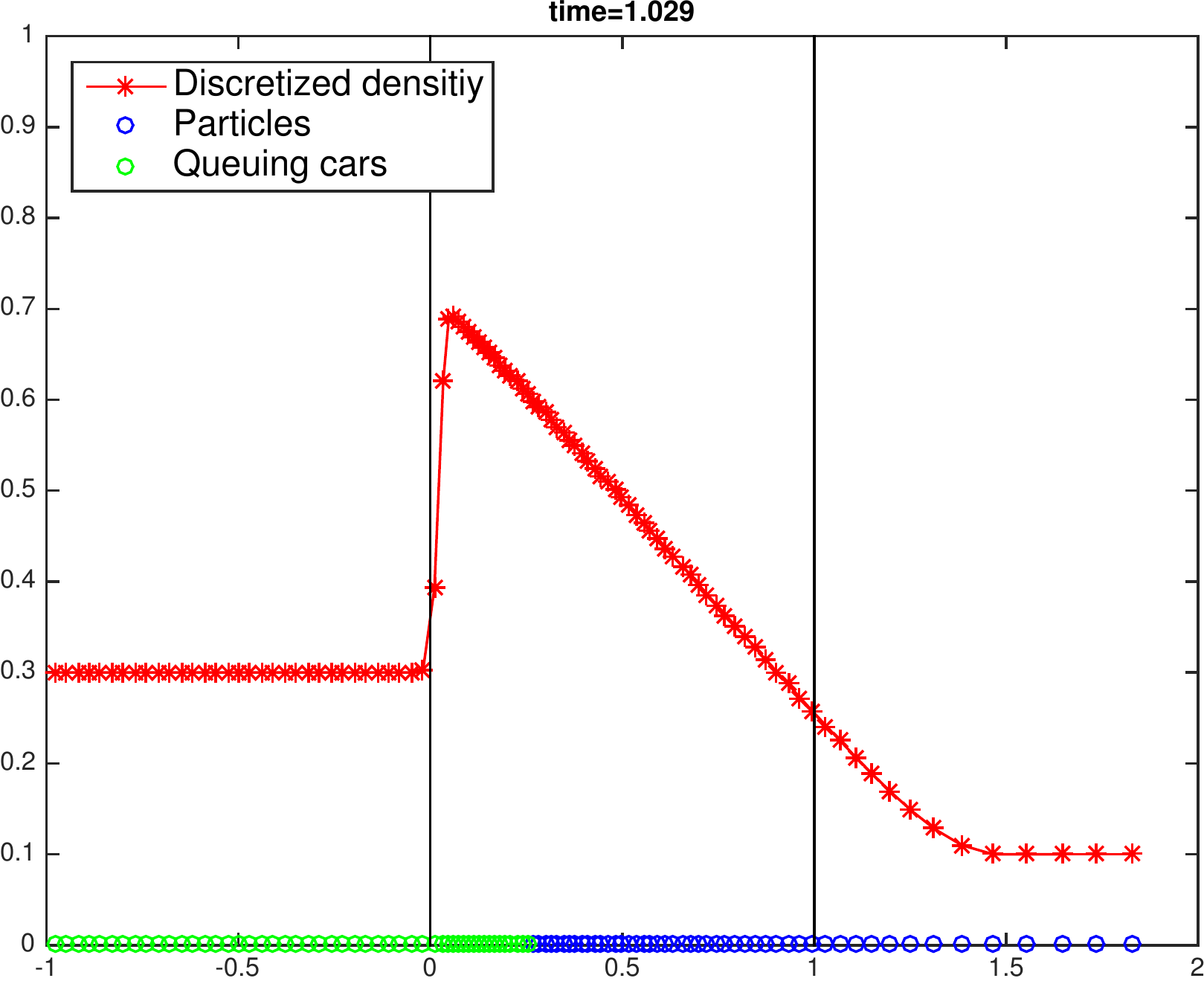}\\
\includegraphics[width=.32\textwidth]{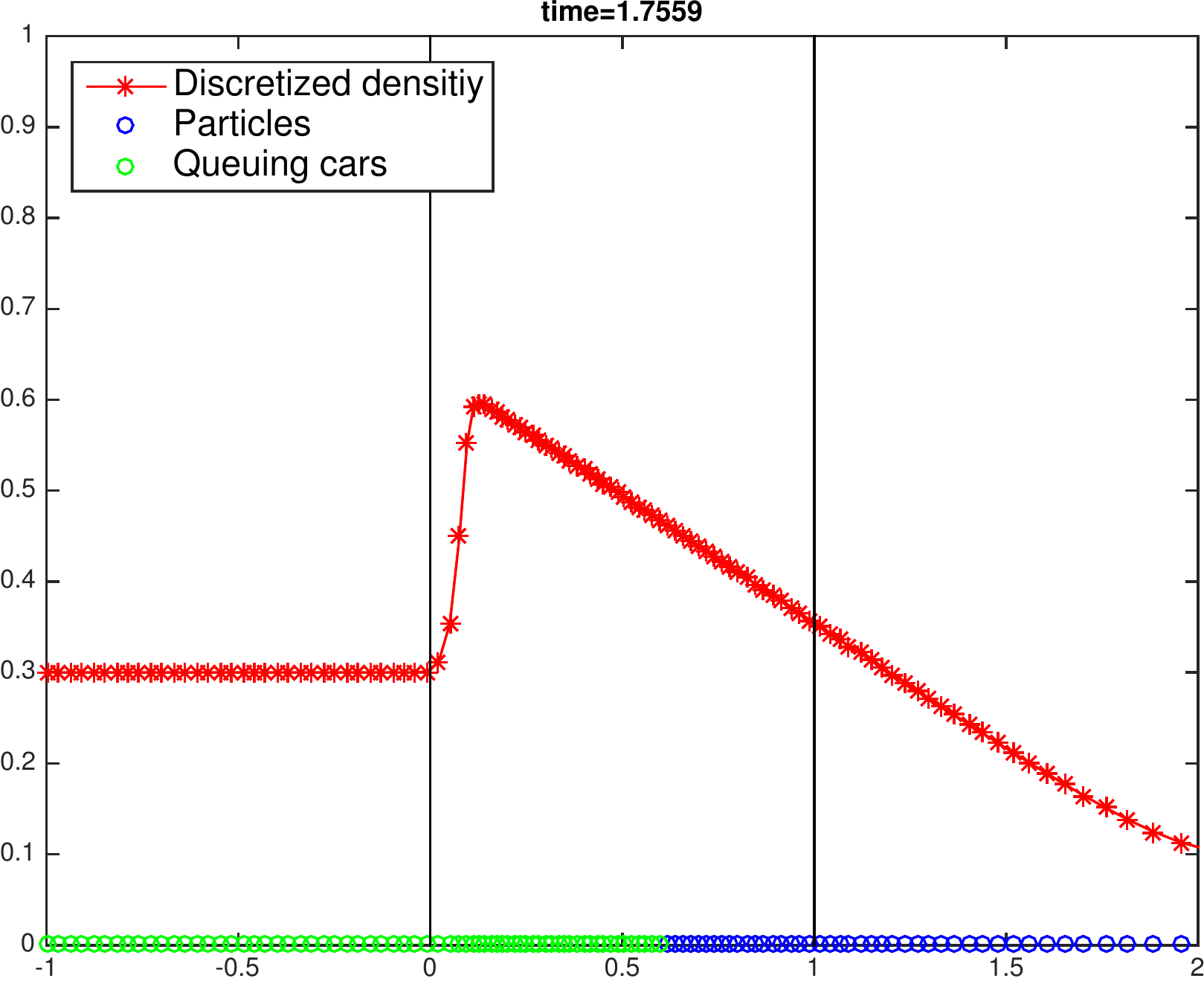}~
\includegraphics[width=.32\textwidth]{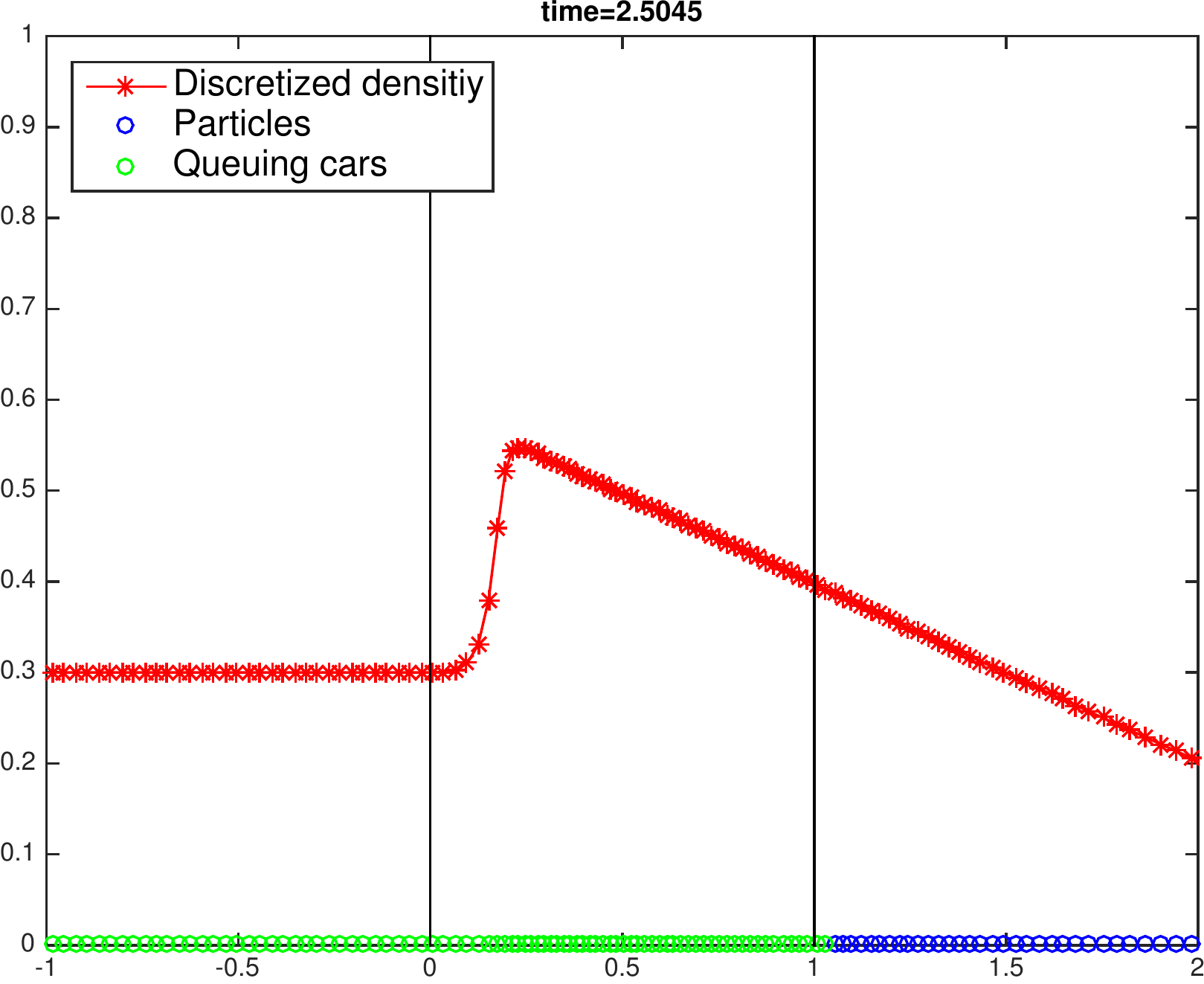}~
\includegraphics[width=.32\textwidth]{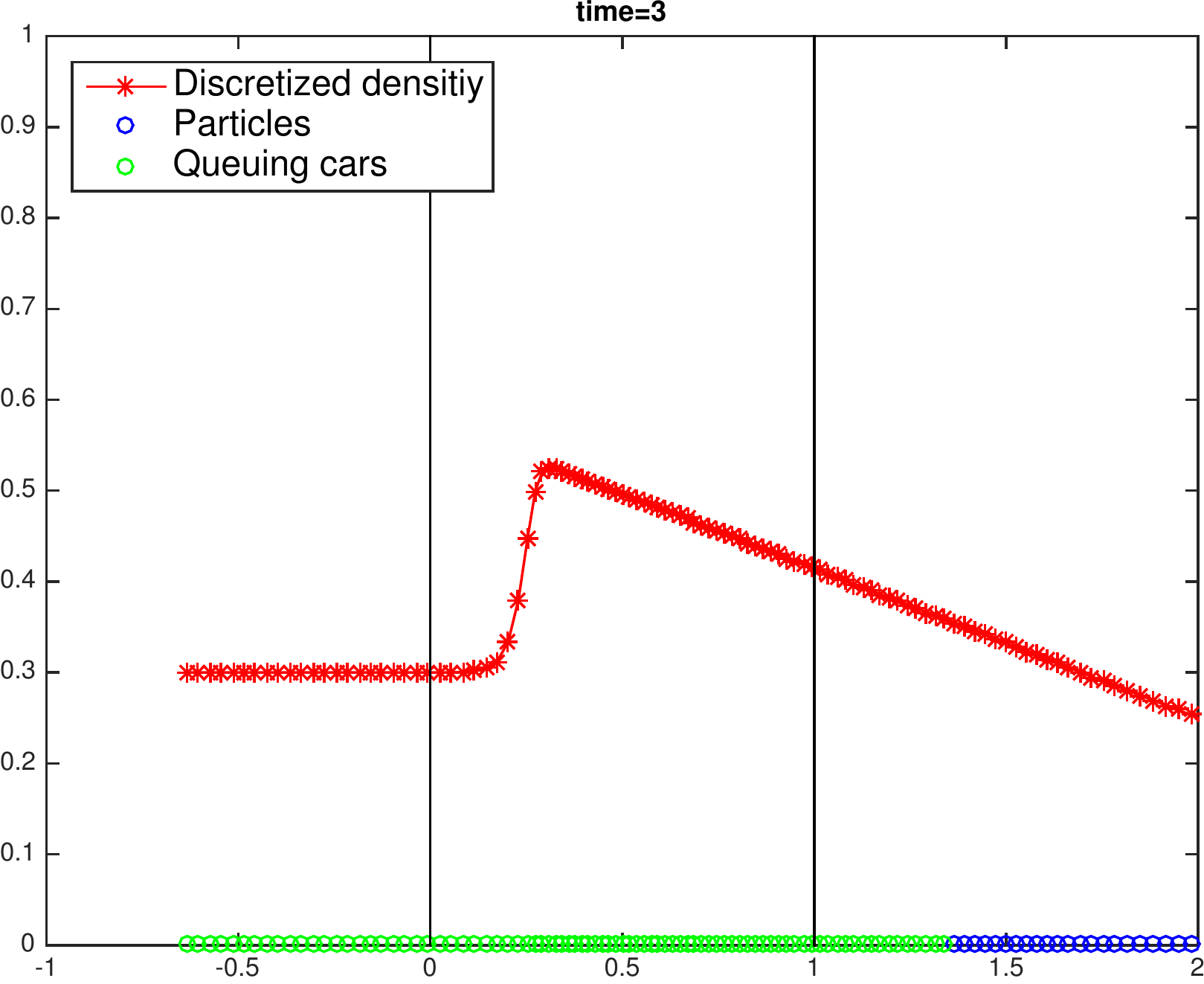}
\end{center}
\caption{Simulation for initial-boundary data given in \eqref{PP}.}
\label{fig:Test3bc}
\end{figure}

Time-dependent piecewise constant boundary data are considered in \figurename~\ref{fig:Test5bc}, where for $N=400$, we set
\begin{align}\label{PTD}
&\bar{\rho}(x)=0.3,
&\bar{\rho}_0(t)=\begin{cases}
 0.1 &\text{if } t\in [0,1], \\
 0.6 &\text{if } t\in (1,2],
\end{cases}&
&\bar{\rho}_1(t)=\begin{cases}
 0.9 &\text{if } t\in [0,1], \\
 0.1 &\text{if } t\in (1,2].
\end{cases}
\end{align}
Using these conditions one can built the exact solutions at time $T=2$
\begin{equation}\label{PTDex}
\rho_{ex}(2,x)=\begin{cases}
 0.5(1-x) &\text{if } x\in \left[0,0.8\right], \\
 0.1 &\text{if } x\in (0.8,0.2(9-2\sqrt{5})],\\
 0.5(2-x) &\text{if } x\in (0.2(9-2\sqrt{5}),1].
\end{cases}
\end{equation}
A comparison with the exact solution $\rho_{ex}$ is given in \figurename~\ref{fig:Test6bc}.

\begin{figure}
\begin{center}
\includegraphics[width=.32\textwidth]{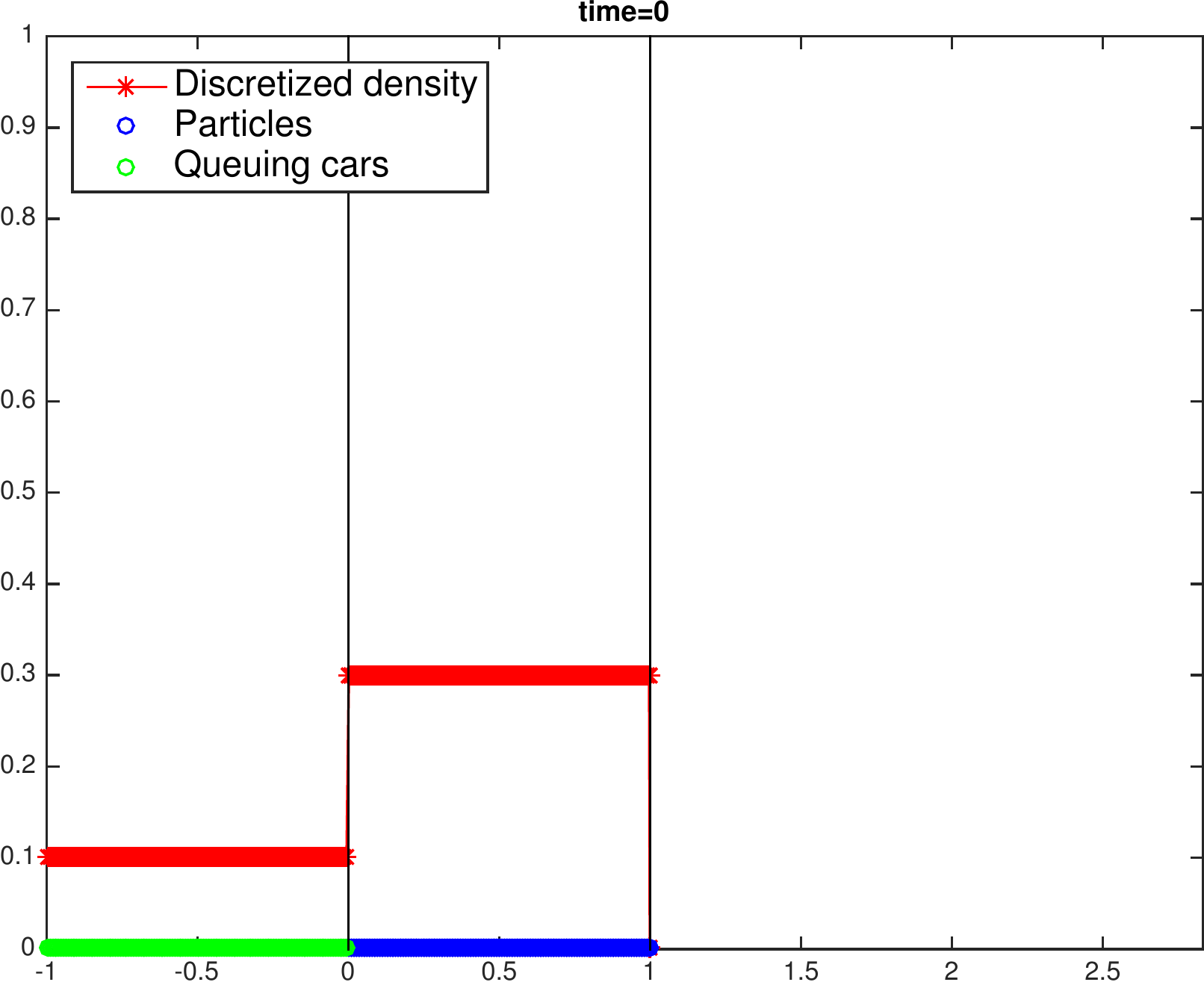}~
\includegraphics[width=.32\textwidth]{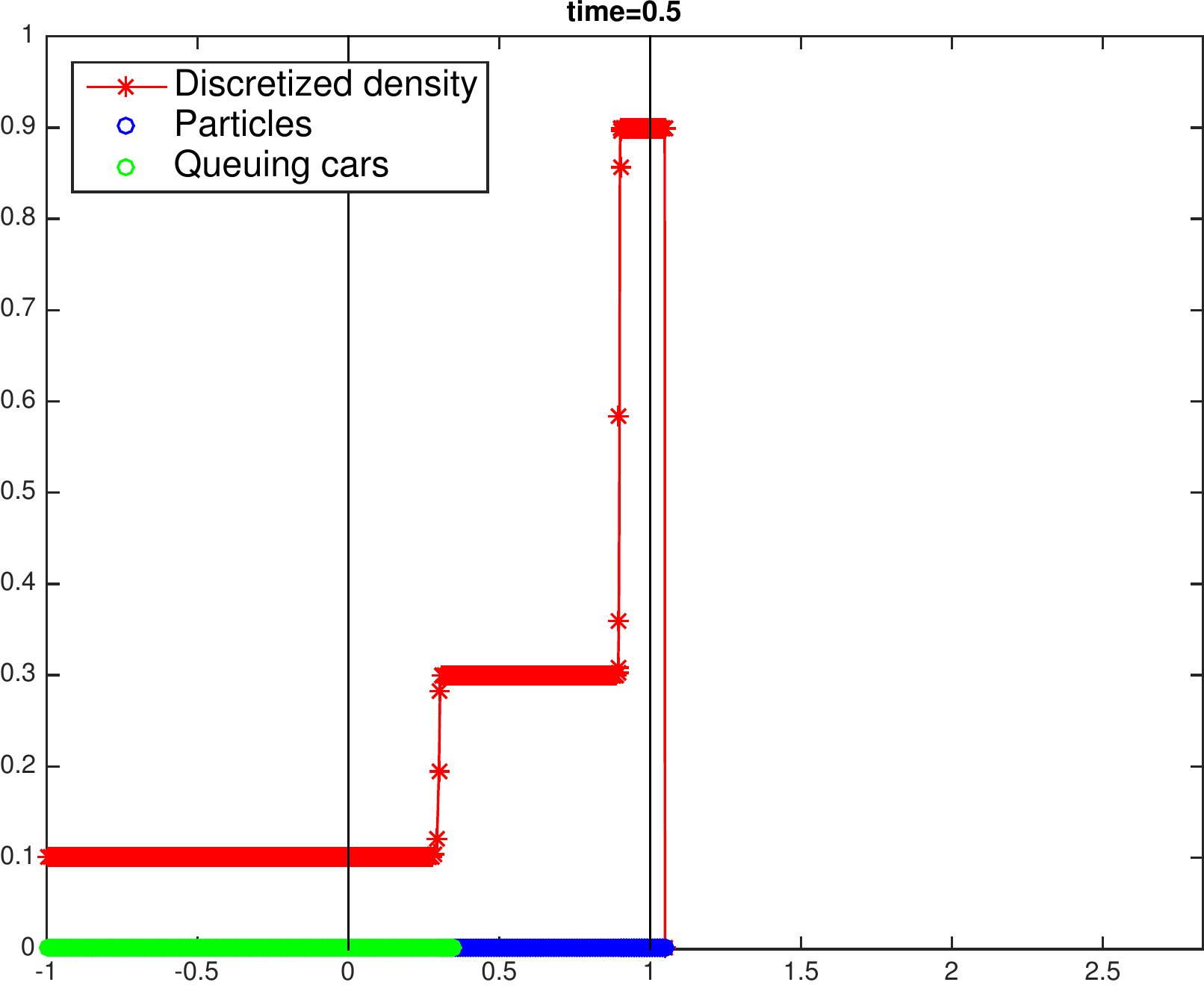}~
\includegraphics[width=.32\textwidth]{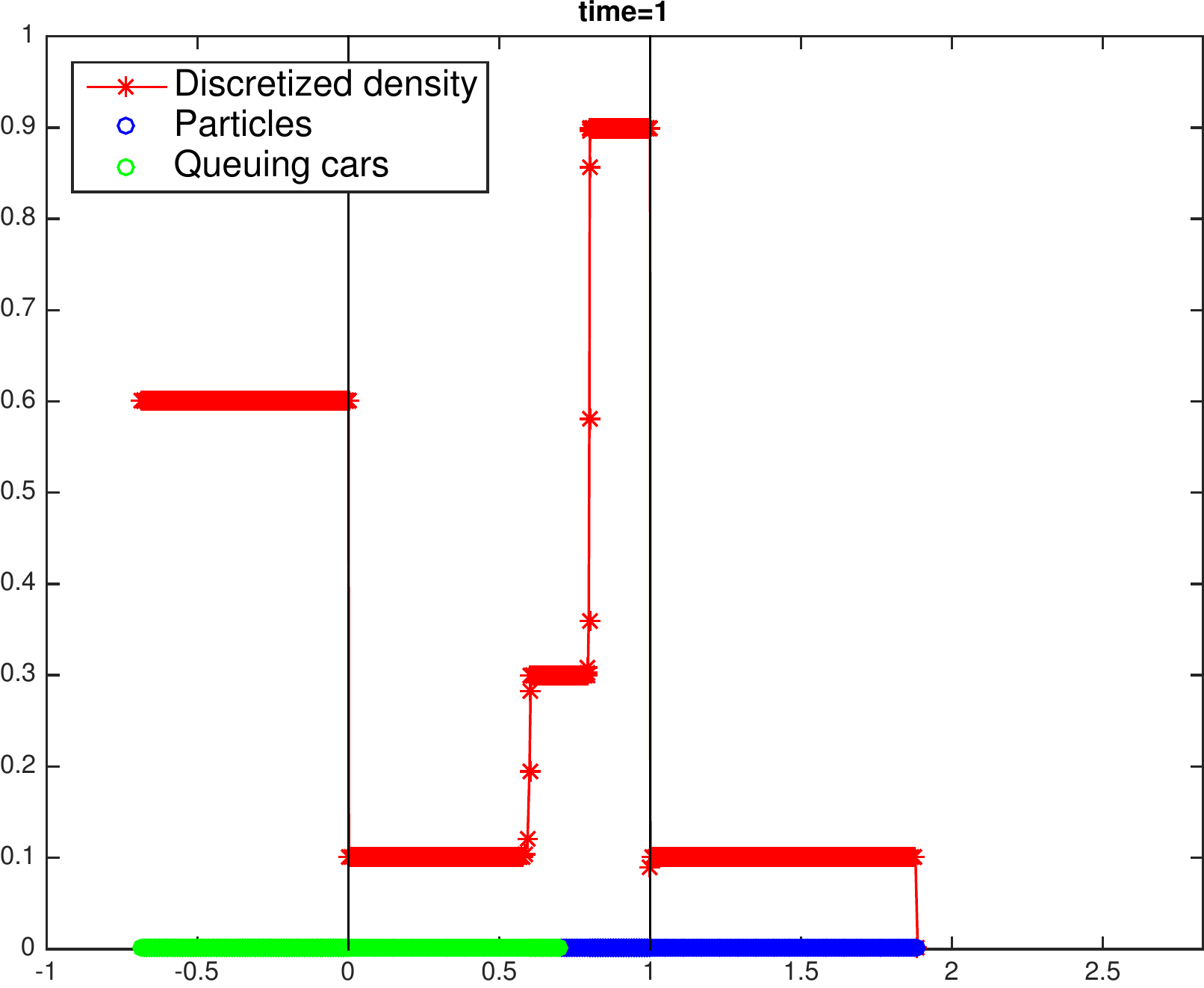}\\
\includegraphics[width=.32\textwidth]{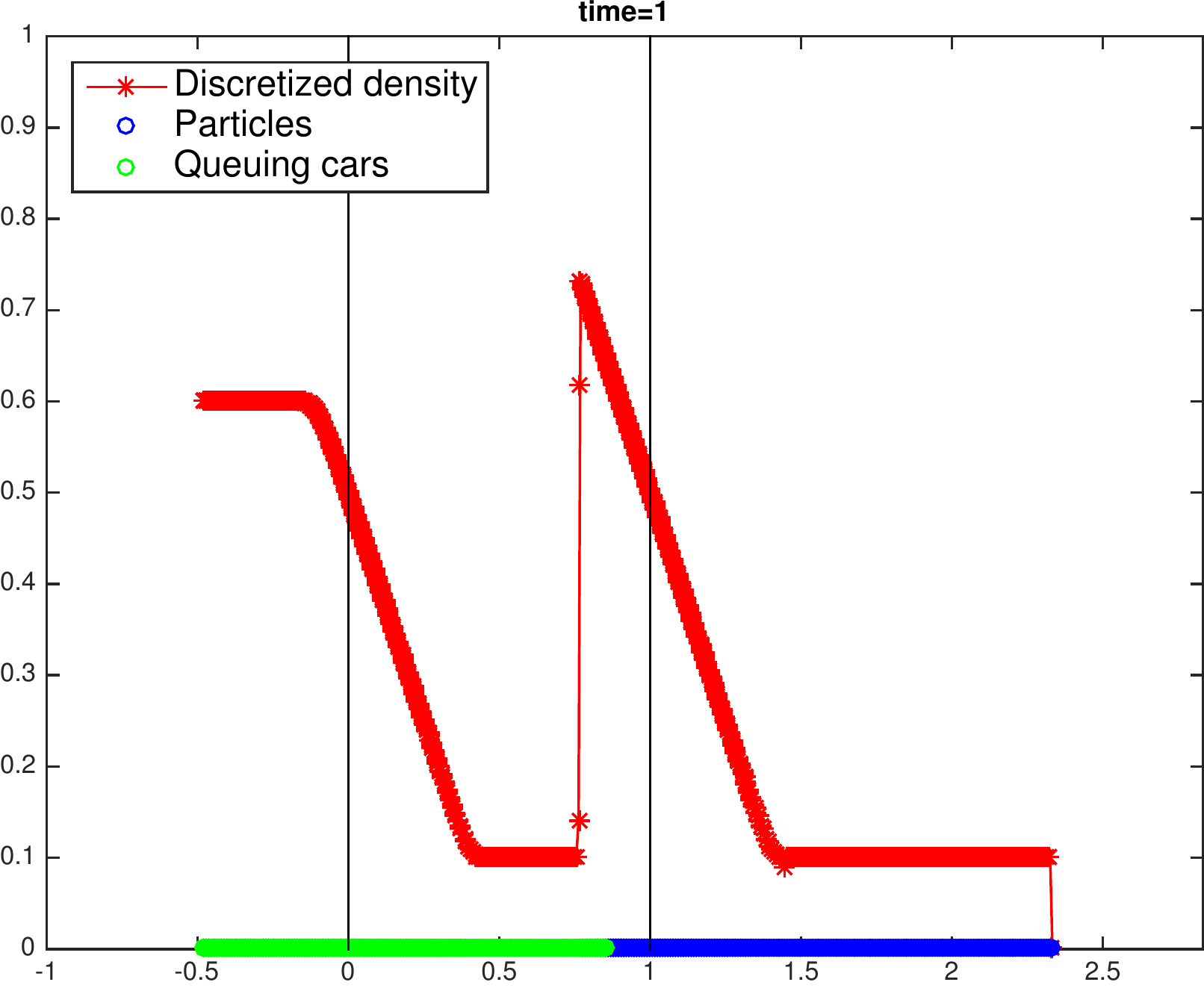}~
\includegraphics[width=.32\textwidth]{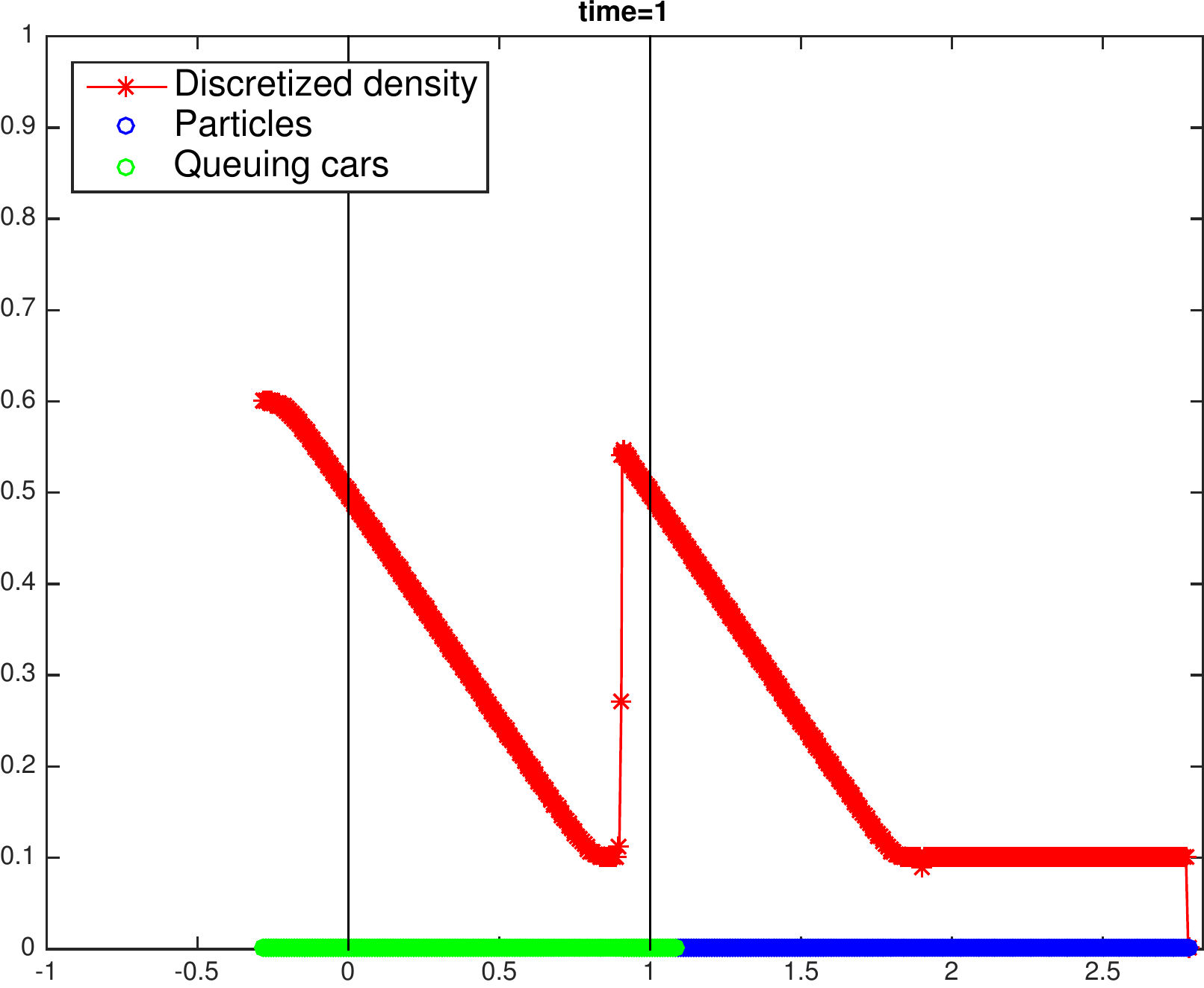}
\end{center}
\caption{Simulation for initial-boundary data given in \eqref{PTD}.}
\label{fig:Test5bc}
\end{figure}

\begin{figure}[htbp]\centering
\includegraphics[width=.4\textwidth]{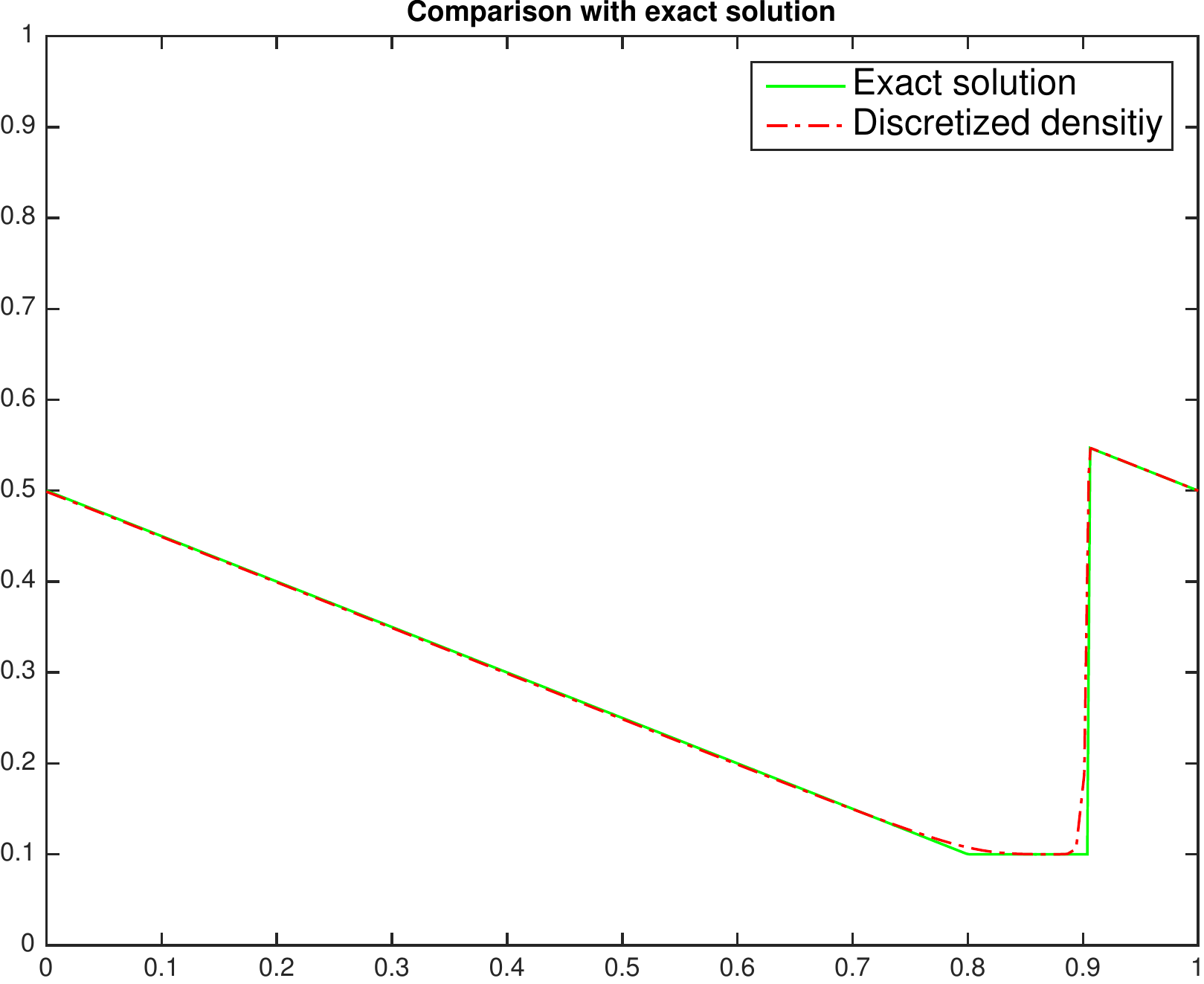}
\caption{Comparison between the approximate and the exact solution given in \eqref{PTDex} to the IBVP \eqref{eq:ibvp} with data given in \eqref{PTD}. \label{fig:Test6bc}}
\end{figure}

\subsection{The ARZ model}

For the ARZ model \eqref{eq:CauchyARZ}, we consider two examples of Riemann problem.
The first one coincides with that one shown in \cite[Section~4]{chalonsgoatin2007} and is used to check the ability of the scheme to deal with contact discontinuities.
The second one is the example given in \cite[Section~5]{AKMR2002} and is used to check the ability of the scheme to deal with vacuum. The qualitative results corresponding to $N=200$ and final time $T=0.2$ for the Test 1 and $T=1$ for Test 4 are presented in \figurename~\ref{fig:Test 1,2}.

\begin{figure}

\begin{minipage}[c]{.45\textwidth}
\begin{center}
\begin{tabular}{c}
    Test 1\\
  \hline
    $p(\rho)=1.4427\,\log(\rho)$,\\
    $\rho_\ell=0.5,~v_\ell=1.2$,\\
    $\rho_r=0.1,~v_r=1.6$,\\
  \hline
  \end{tabular}
\end{center}
\end{minipage}
\hfill
\begin{minipage}[c]{.45\textwidth}
\begin{center}
\begin{tabular}{c}
    Test 2\\
  \hline
    $p(\rho)=6\rho$,\\
    $\rho_\ell=0.05,~v_\ell=0.05$,\\
    $\rho_r=0.05,~v_r=0.5$,\\
  \hline
  \end{tabular}
\end{center}
\end{minipage}
\bigskip\\
\begin{minipage}[c]{.45\textwidth}
\begin{center}
\includegraphics[width=\textwidth]{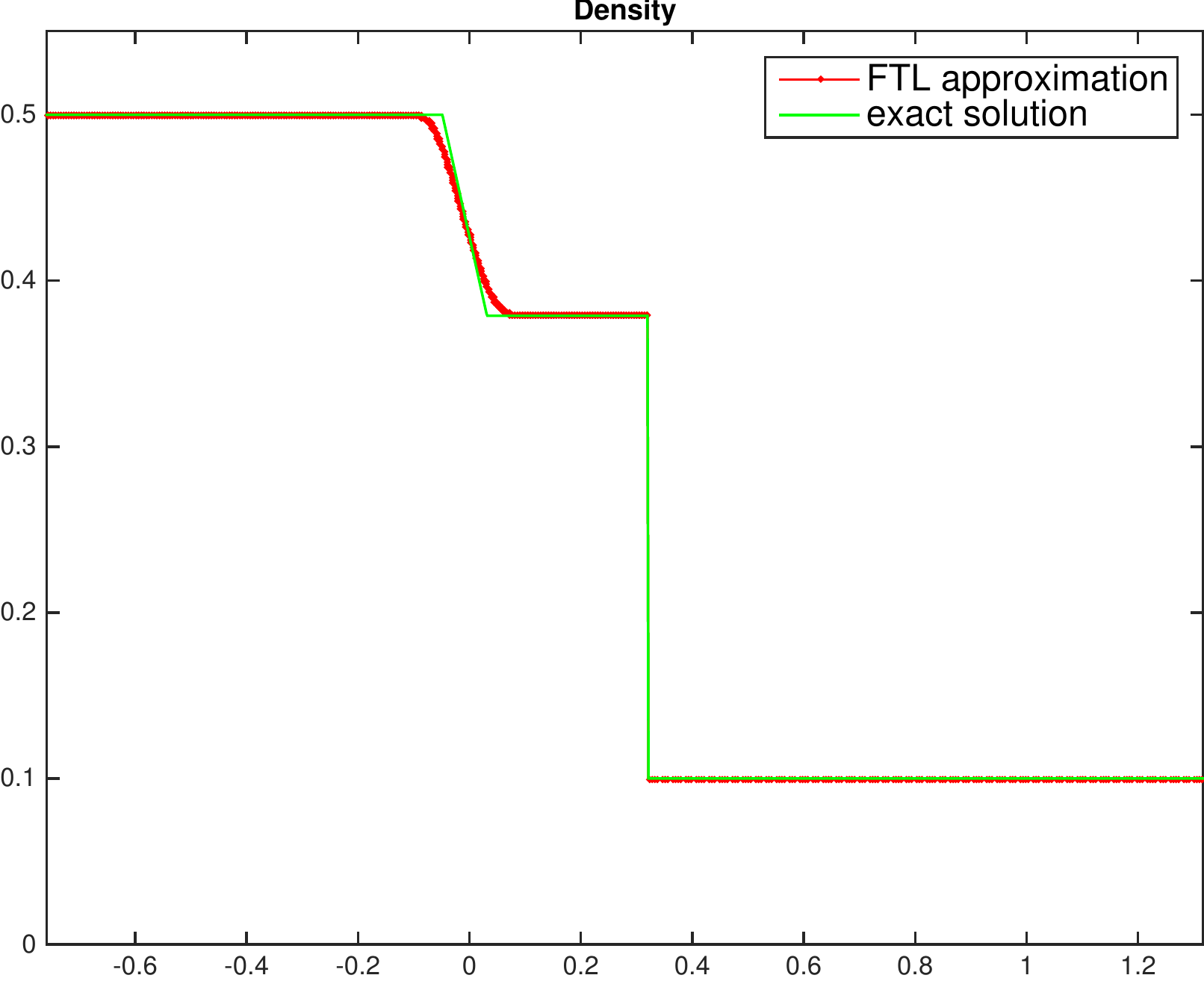}
\end{center}
\end{minipage}
\hfill
\begin{minipage}[c]{.45\textwidth}
\begin{center}
\includegraphics[width=\textwidth]{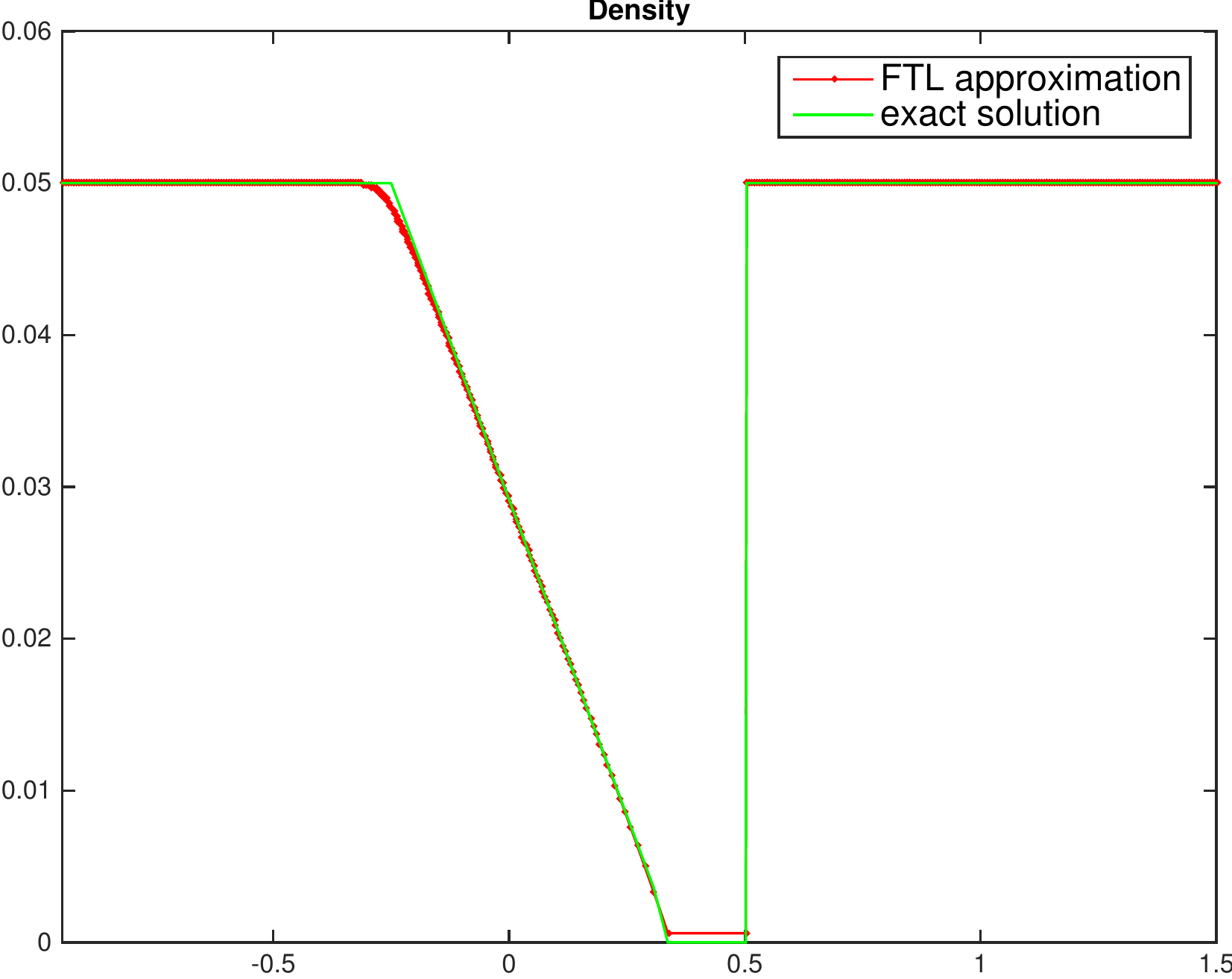}
\end{center}
\end{minipage}
\bigskip\\
\begin{minipage}[c]{.45\textwidth}
\begin{center}
\includegraphics[width=\textwidth]{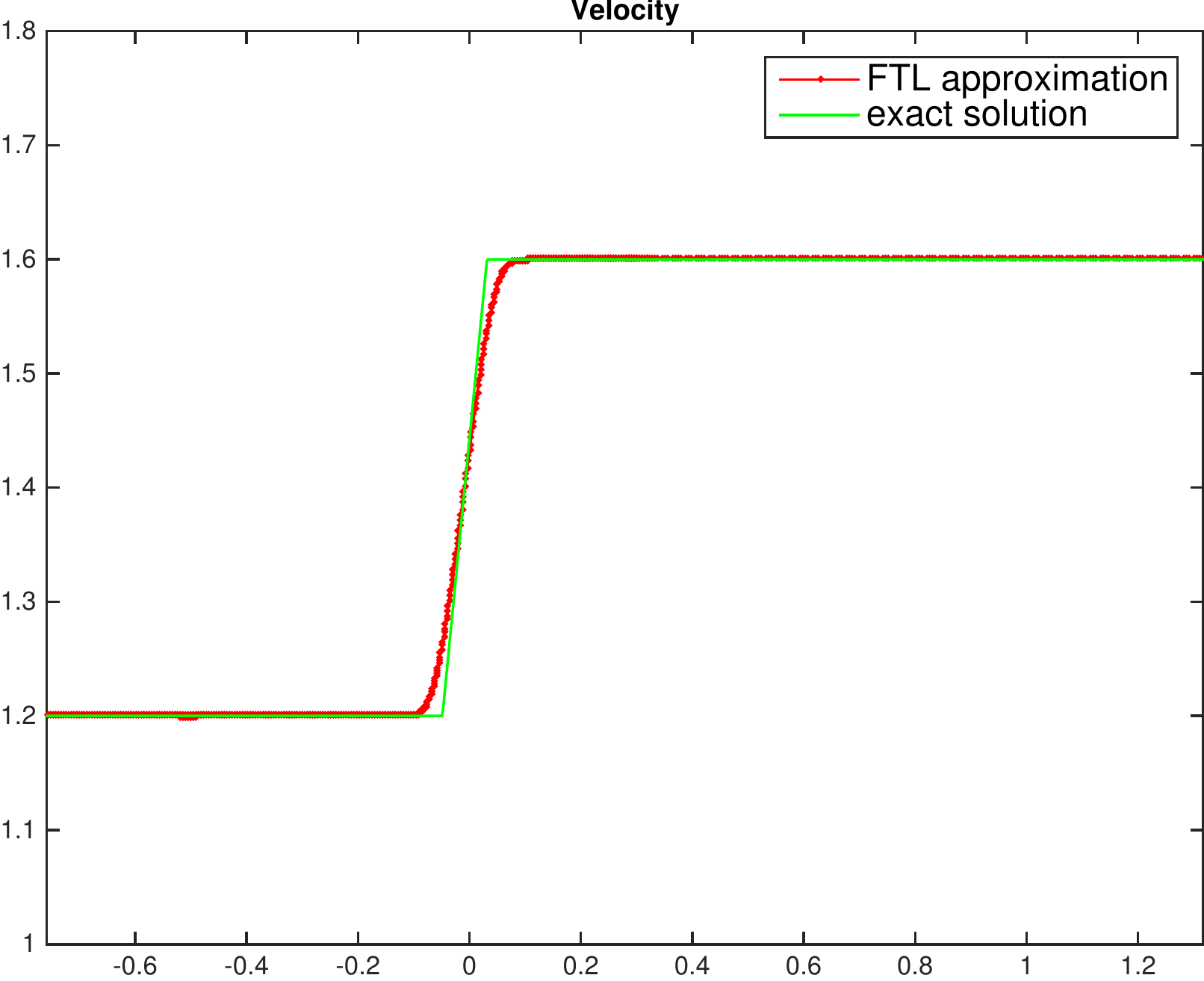}
\end{center}
\end{minipage}
\hfill
\begin{minipage}[c]{.45\textwidth}
\begin{center}
\includegraphics[width=\textwidth]{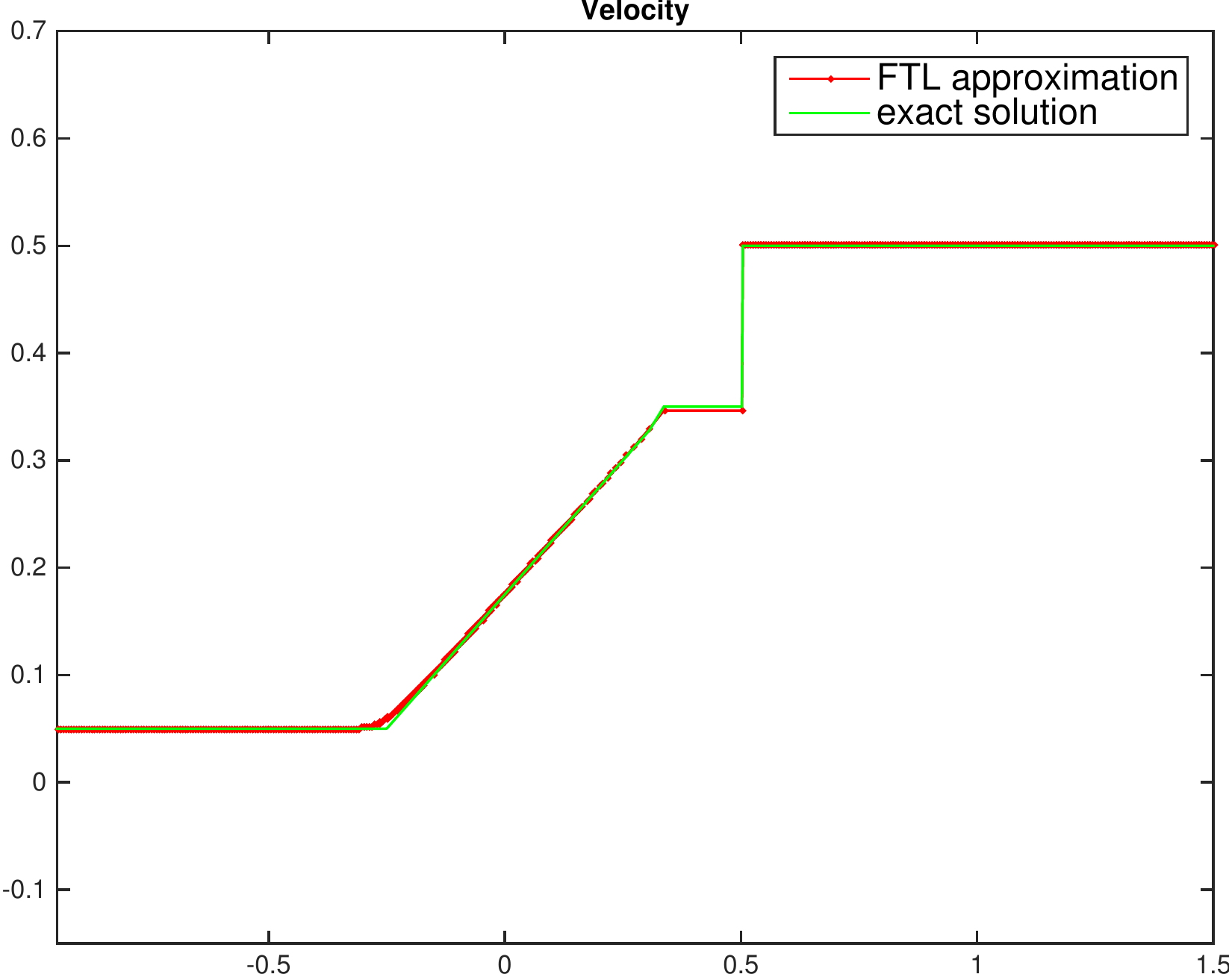}
\end{center}
\end{minipage}
\bigskip\\
\begin{minipage}[c]{.45\textwidth}
\begin{center}
\includegraphics[width=\textwidth]{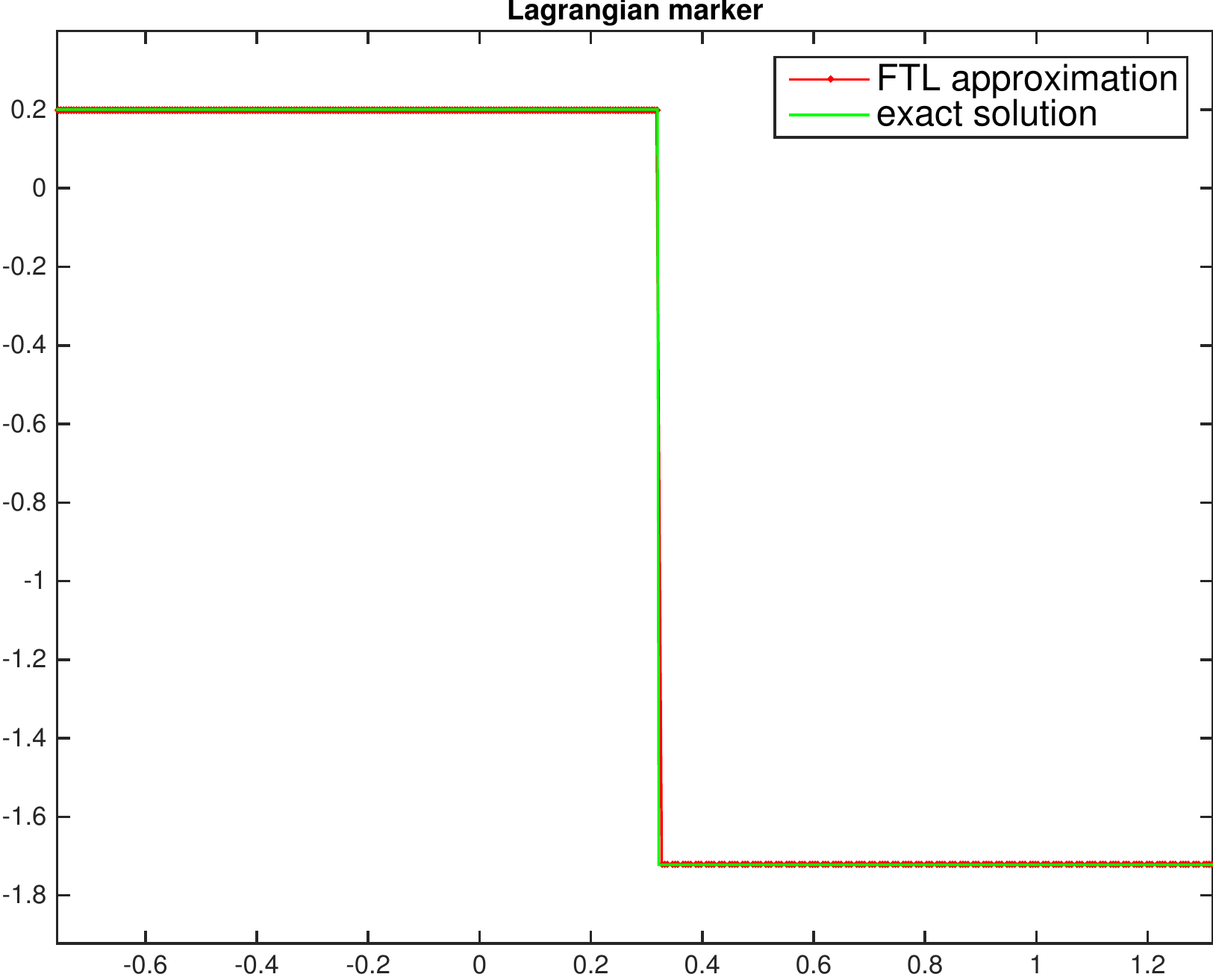}
\end{center}
\end{minipage}
\hfill
\begin{minipage}[c]{.45\textwidth}
\begin{center}
\includegraphics[width=\textwidth]{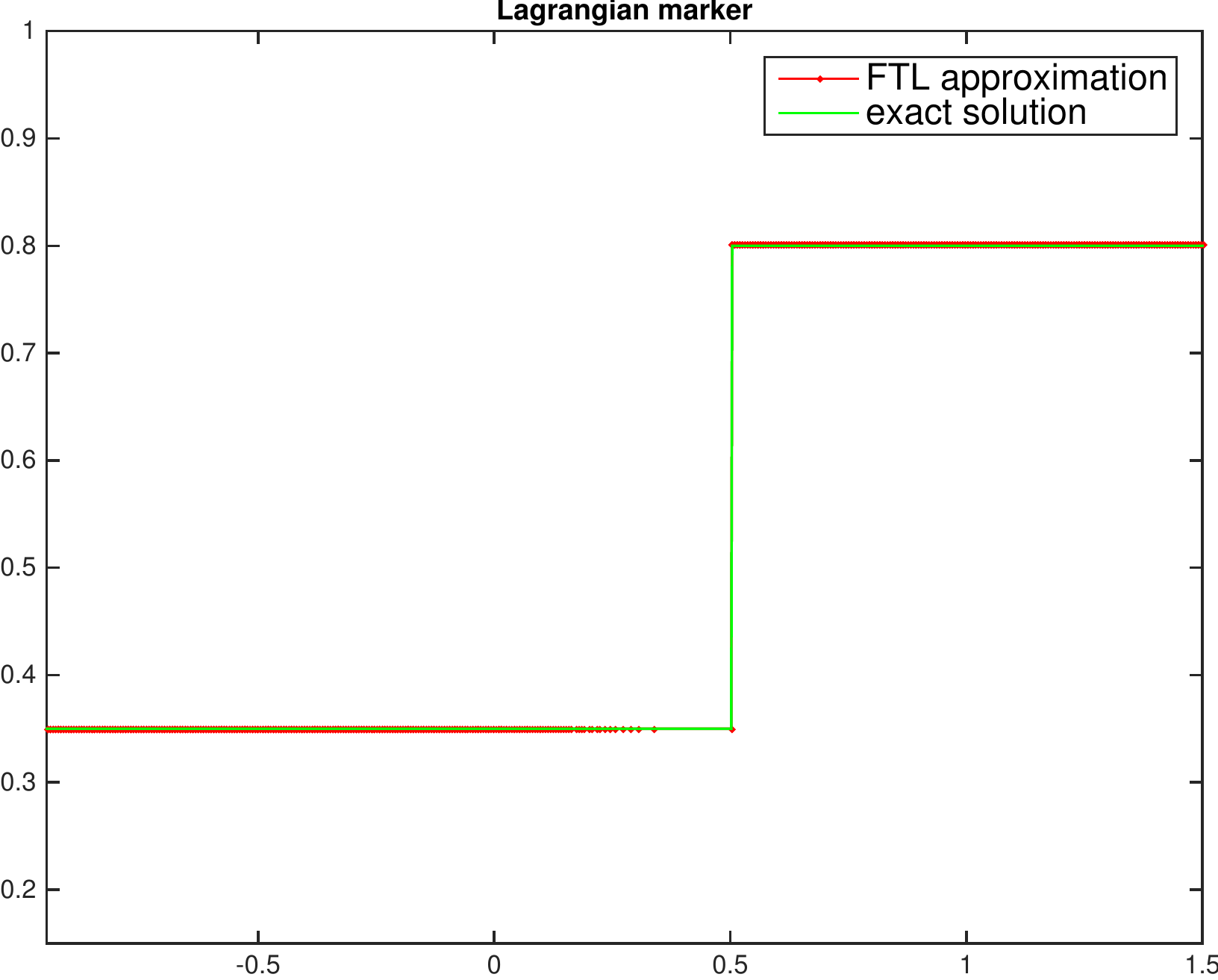}
\end{center}
\end{minipage}
\caption{Left column for Test 1 and right column for Test 2, with $N=200$. \label{fig:Test 1,2}}
\end{figure}

\subsection{The Hughes model for pedestrian movements}\label{sec:numerics-Hughes}

In this section we compare our discrete density for the Hughes model \eqref{eq:model} with approximate solutions obtained via Godunov scheme.
About the boundary conditions, as pointed out in Section~\ref{sec:particle}, we do not impose any boundary condition in the particle method.
For the Godunov method we create two extra ghost cells, one just at the left of $-1$ and one just at the right of $1$, setting $\rho = 0$ in those cells, to mimic `perfect exits'.  In the example reported, the choice for the cost function is $c(\rho)\doteq 1/v(\rho)$, with $v(\rho)\doteq1-\rho$, and we show time evolution of the discrete density $\rho^n$ given by \eqref{eq:density} in the domain $\left(-1,1\right)$. 
In order to compare our method with the tests performed in \cite{DiFrancescoMarkowichPietschmannWolfram, GoatinMimault}, in \figurename~\ref{fig:Test 5} we consider the three-step initial condition
\begin{equation}\label{Steps}
\bar{\rho}(x)=\begin{cases}
 0.8 &\text{if } -0.8< x\leq -0.5, \\
 0.6 &\text{if } -0.3< x\leq 0.3, \\
 0.9 &\text{if } 0.4< x\leq 0.75, \\
 0 & \text{otherwise}.
\end{cases}
\end{equation}
As shown in \figurename~\ref{fig:Test 5} and \figurename~\ref{fig:Comp}, this example exhibits the typical \emph{mass transfer} phenomenon occurring when the turning point $\xi(t)$ is not surrounded by a vacuum region. In such a case, particles are crossing $\xi(t)$, and a \emph{non-classical shock} starts from $\xi(t)$, see \cite[Remark~5]{AmadoriDiFrancesco}.
In the example we set $N=200$ and plot the particle positions and the discrete densities.
\begin{figure}
\centering
\includegraphics[width=.32\textwidth]{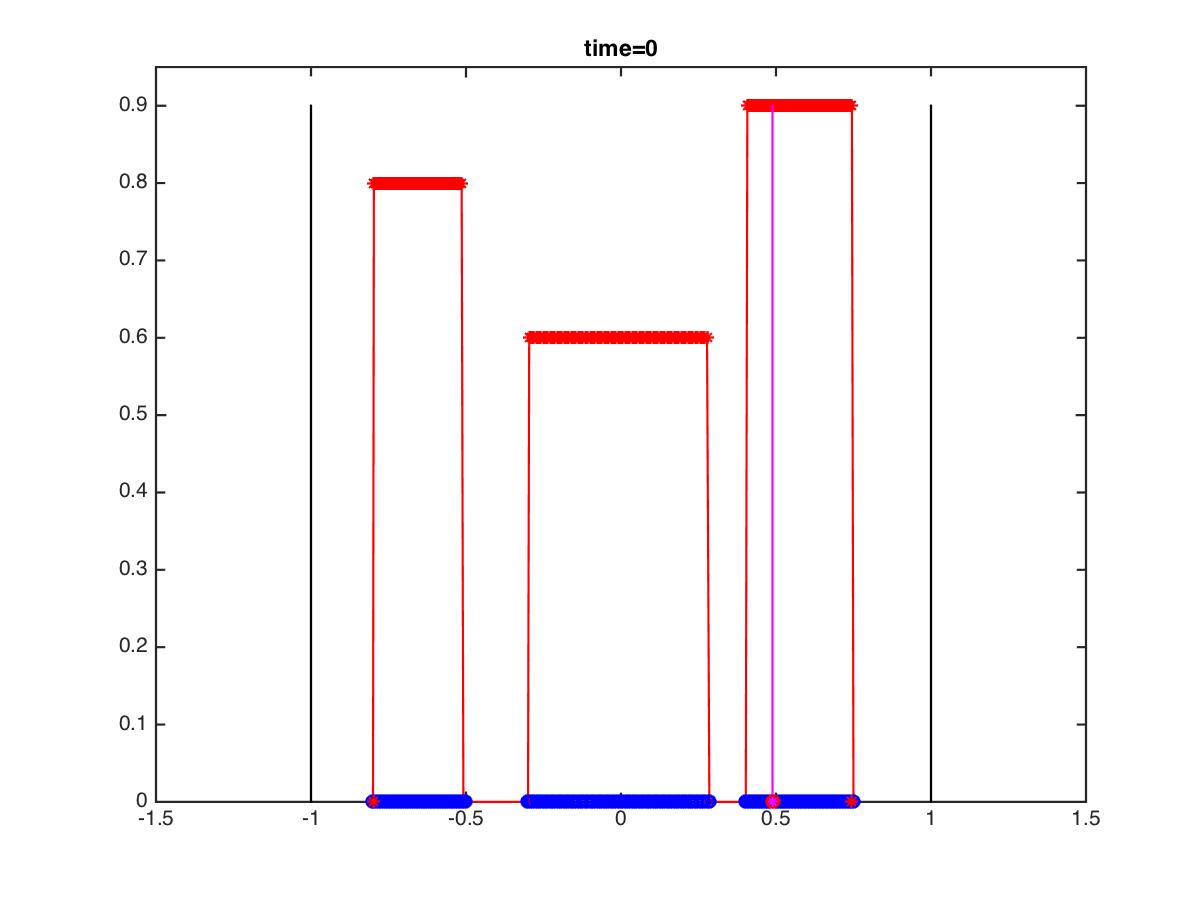}
\includegraphics[width=.32\textwidth]{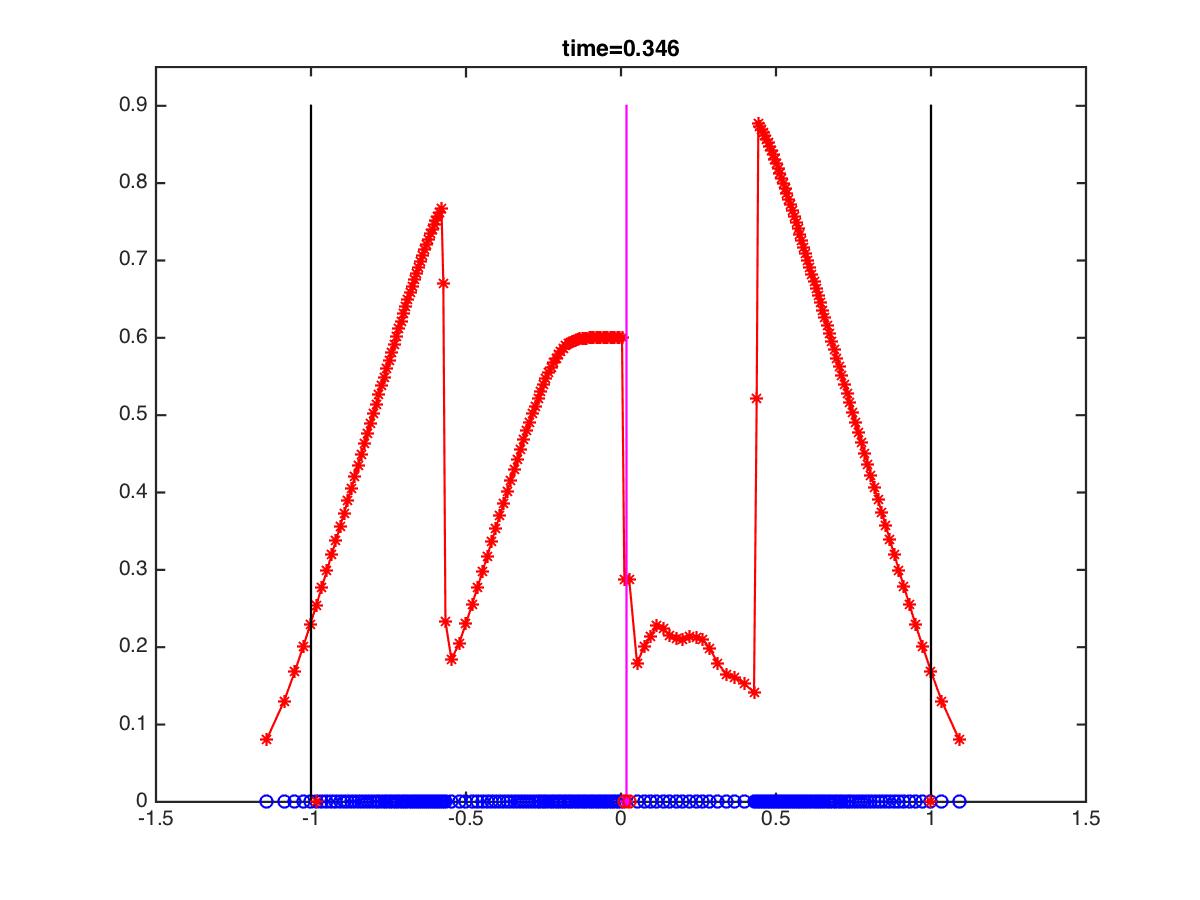}
\includegraphics[width=.32\textwidth]{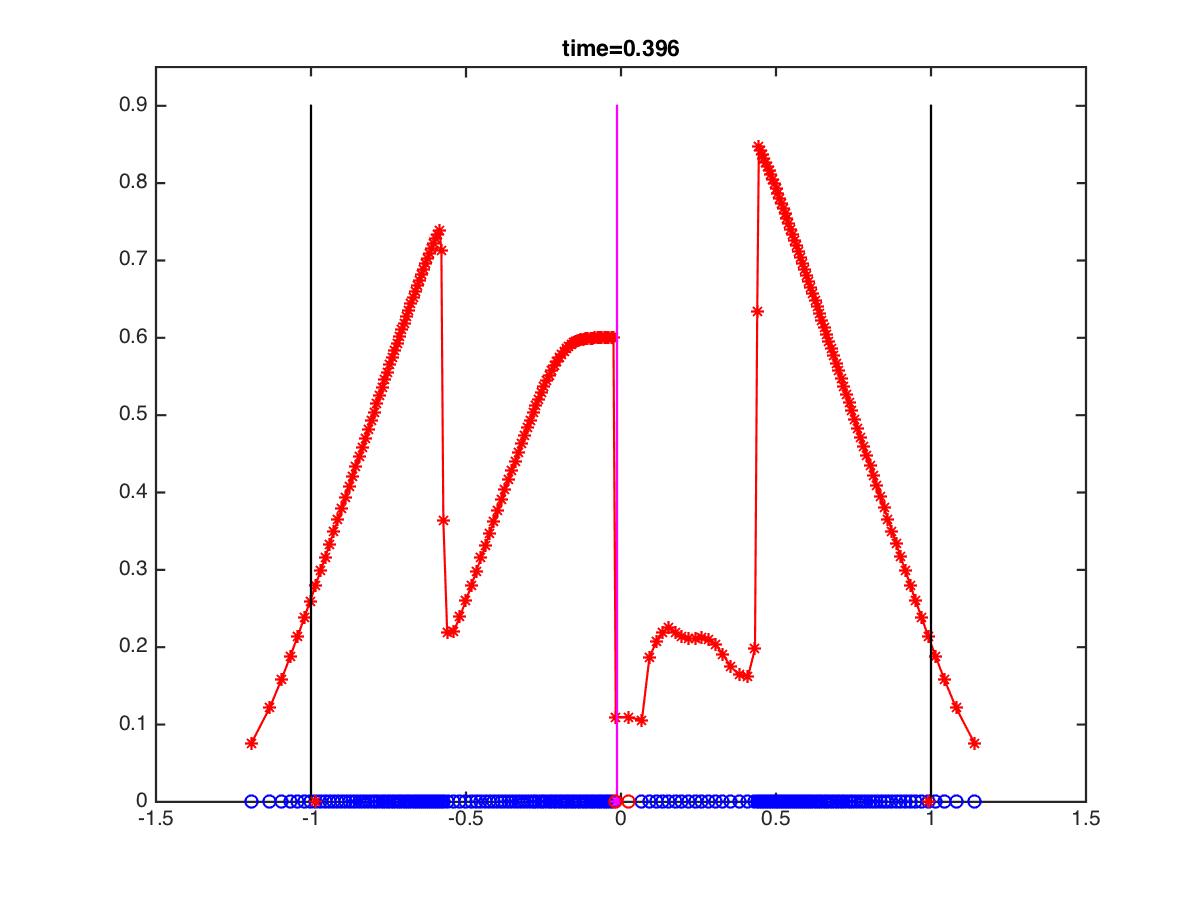}\\
\includegraphics[width=.32\textwidth]{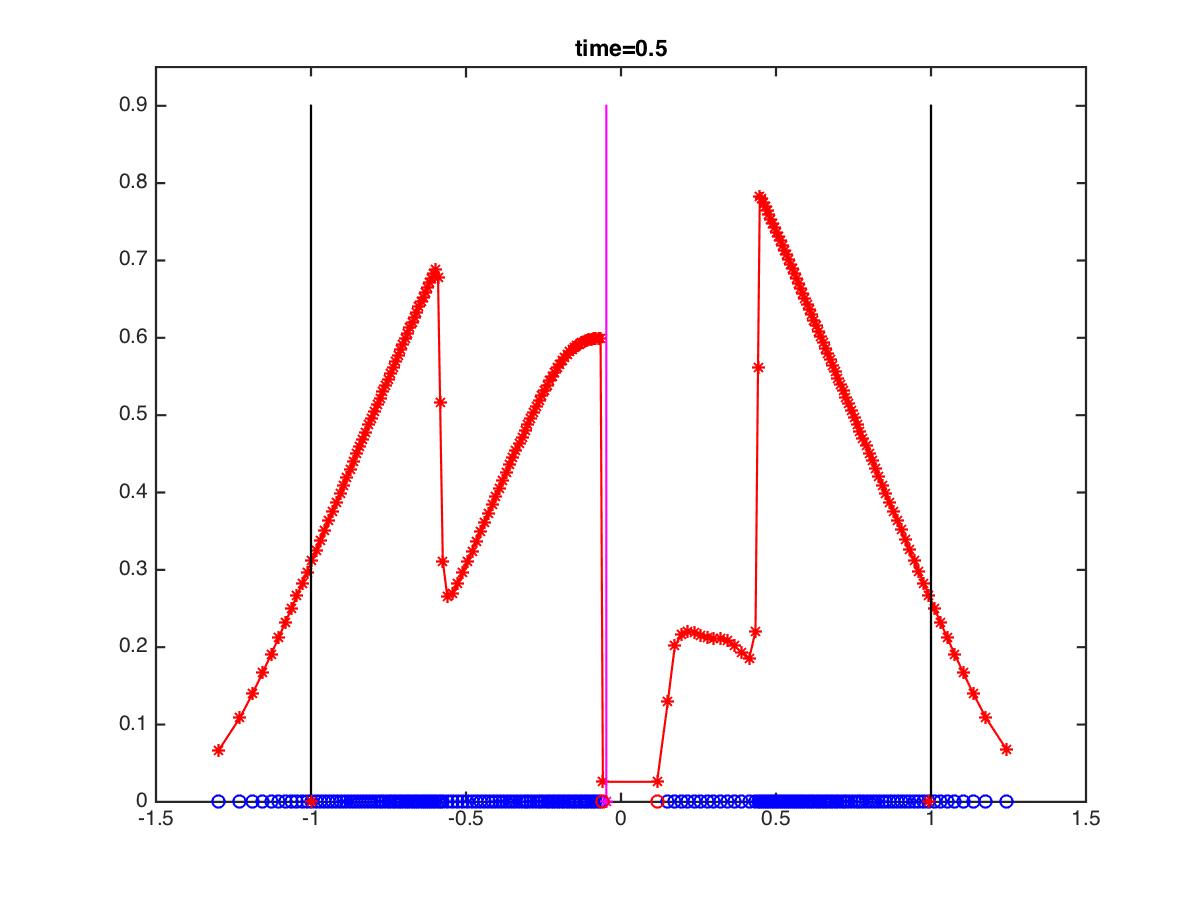}
\includegraphics[width=.32\textwidth]{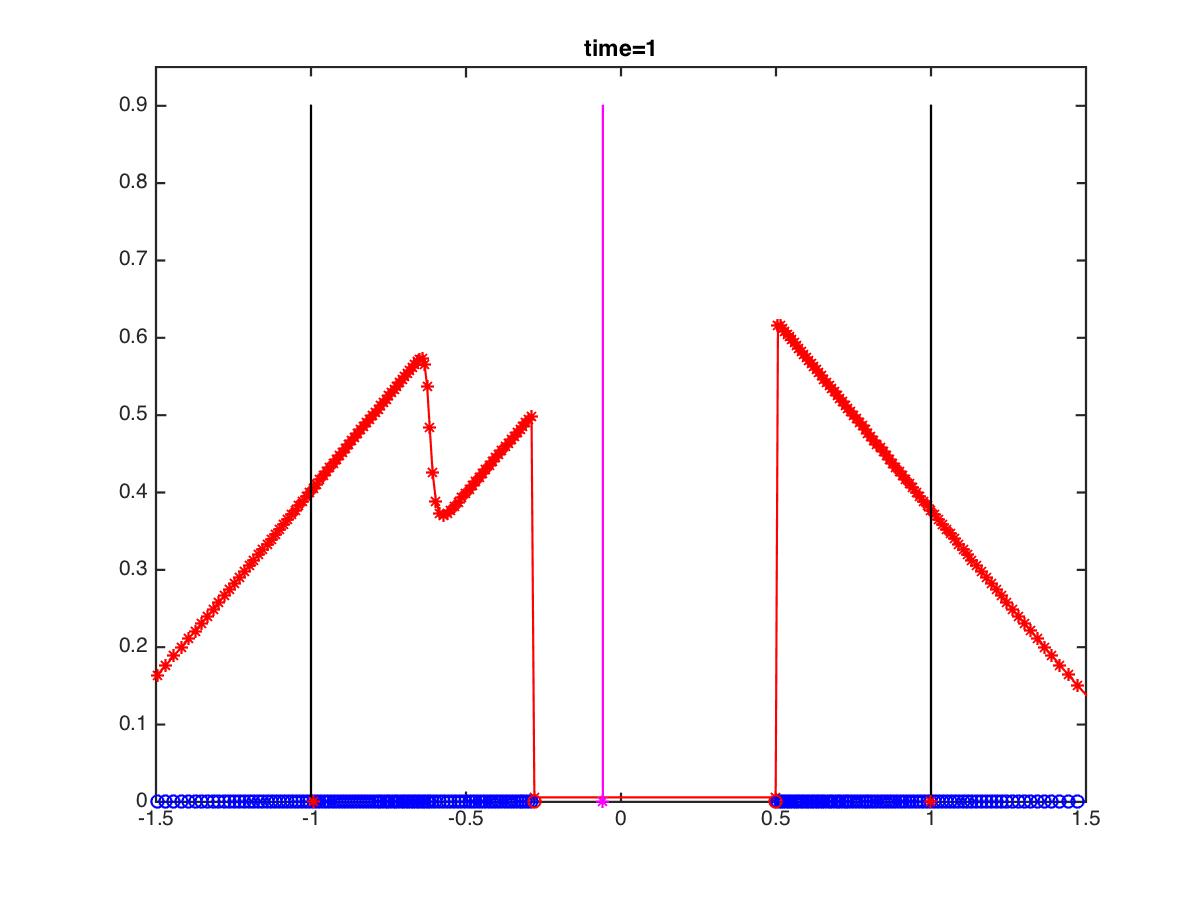}
\caption{Evolution of $\rho^n$ (in red) with initial data given in \eqref{Steps}.  The blu dots represent particles positions. The magenta vertical line is the turning point. In second and third snapshot we see mass transfer across the turning point and non-classical shock. \label{fig:Test 5}}
\end{figure}
In \figurename~\ref{fig:Comp} we compare the particle method and a classical Godunov scheme.
It is evident that the two methods, though conceptually different, produce approximate solutions are in a good agreement.

\begin{figure}
\centering
\includegraphics[width=.32\textwidth]{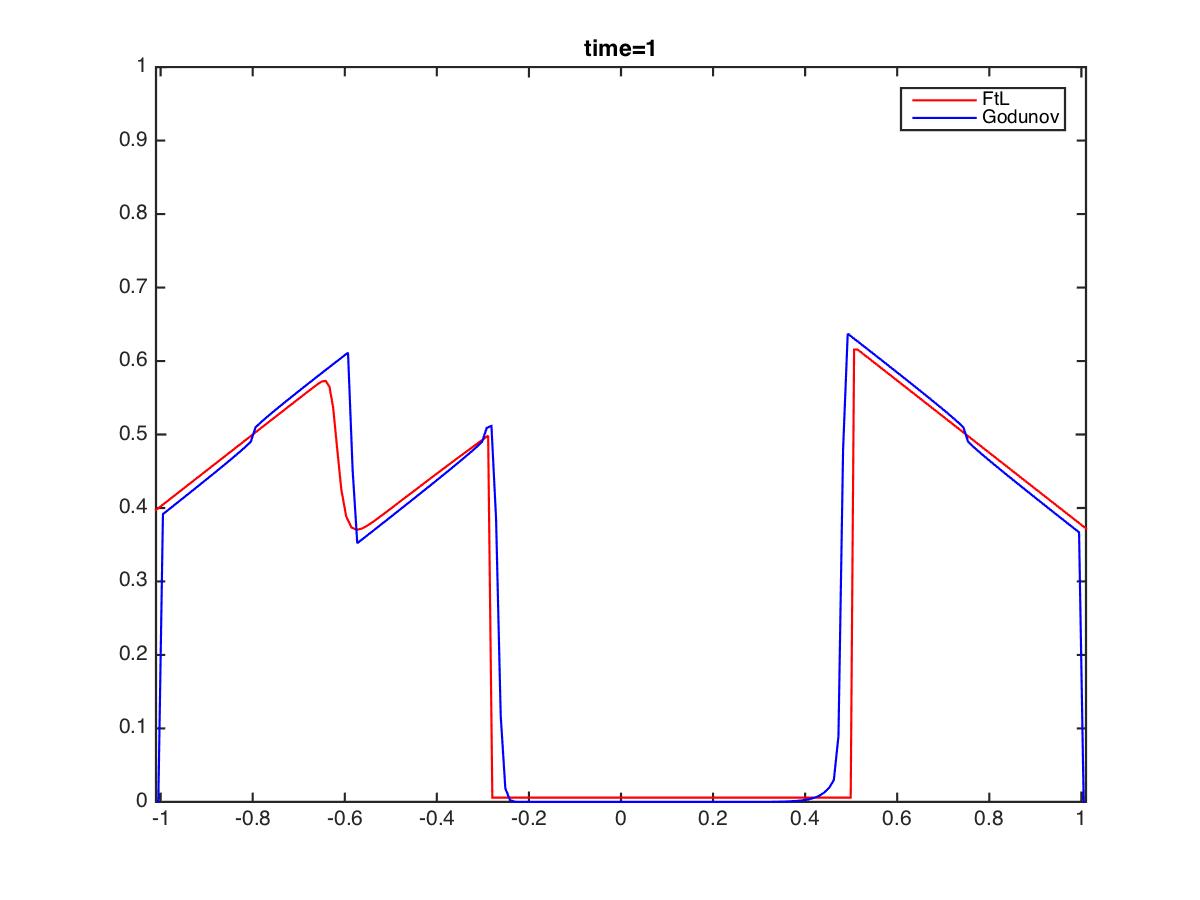}
\includegraphics[width=.32\textwidth]{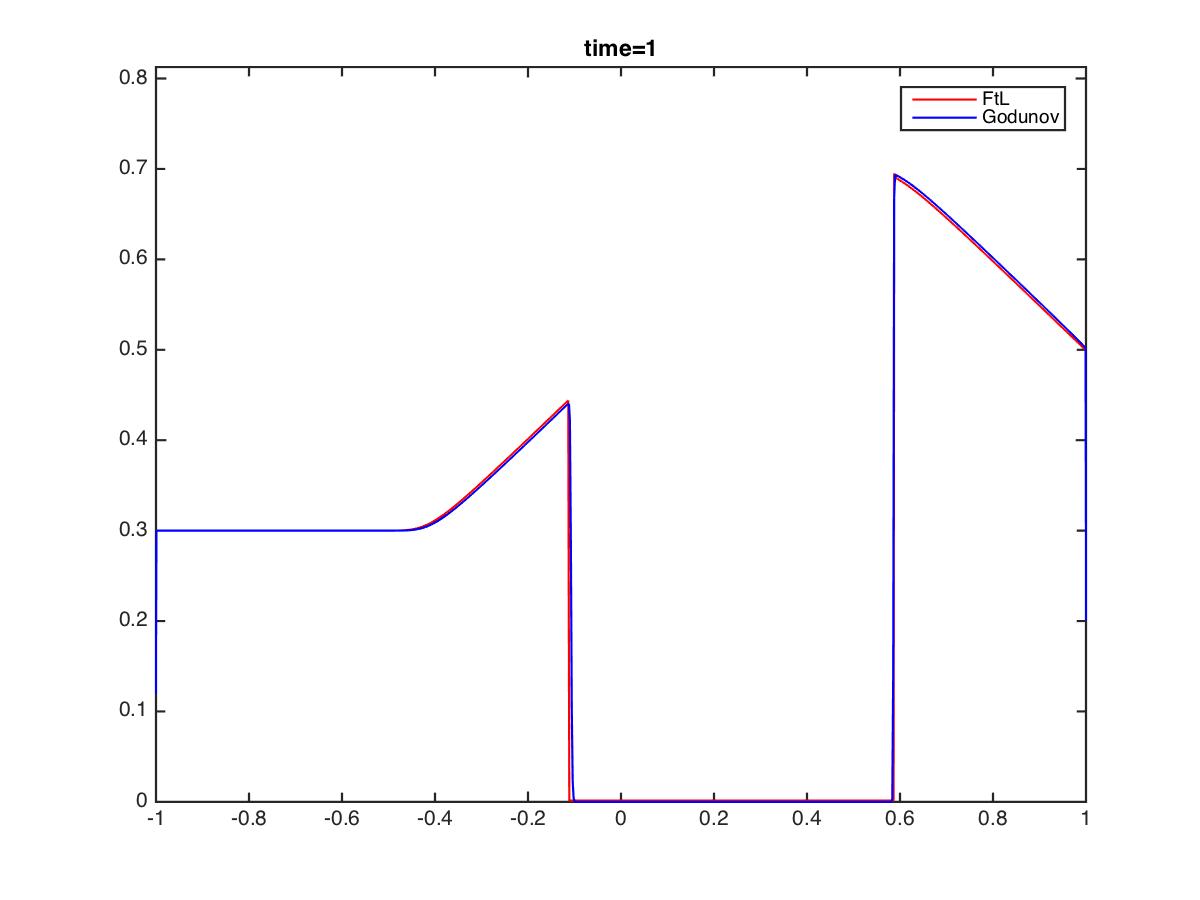}
\caption{Comparison between the FTL (in red) and the Godunov (in blu) schemes for the Hughes model \eqref{eq:model}.
On the left for the initial datum given in \eqref{Steps}, $N=100$ and $500$ time iterations.
On the right for the initial datum $\bar{\rho} \doteq 0.3 \, \mathbf{1}_{[-1,0]}+0.7 \, \mathbf{1}_{(0,1]}$, $N=1000$ and $1500$ time iterations.  \label{fig:Comp}}
\end{figure}

\begin{acknowledgement}
MDF and MDR are supported by the GNAMPA (Italian group of Analysis, Probability, and Applications) project \emph{Geometric and qualitative properties of solutions to elliptic and parabolic equations}. SF and MDR are supported by the GNAMPA (Italian group of Analysis, Probability, and Applications) project \emph{Analisi e stabilit\`a per modelli di equazioni alle derivate parziali nella matematica applicata}. GR was partially supported by ITN-ETN Marie Curie Actions ModCompShock - `Modelling and Computation of Shocks and Interfaces'.
\end{acknowledgement}

\end{document}